\documentclass{amsart}

\usepackage{amsmath, amsthm, amssymb, amsfonts}

\usepackage{hyperref}
\usepackage{color}
\usepackage{comment}

\theoremstyle{plain}

\newtheorem{theorem}{Theorem}[section]
\newtheorem{fact}[theorem]{Fact}
\newtheorem{lemma}[theorem]{Lemma}

\newtheorem{corollary}[theorem]{Corollary}

\newtheorem{question}[theorem]{Question}

\newtheorem{proposition}[theorem]{Proposition}

\theoremstyle{definition}

\newtheorem{definition}[theorem]{Definition}

\theoremstyle{remark}
\newtheorem{remark}[theorem]{Remark}\newtheorem{example}[theorem]{Example}

\DeclareMathOperator{\md}{md}

\DeclareMathOperator{\Int}{Int}
\DeclareSymbolFont{AMSb}{U}{msb}{m}{n}
\DeclareMathSymbol{\N}{\mathalpha}{AMSb}{"4E}
\DeclareMathSymbol{\R}{\mathalpha}{AMSb}{"52}
\DeclareMathSymbol{\Z}{\mathalpha}{AMSb}{"5A}
\DeclareMathSymbol{\D}{\mathalpha}{AMSb}{"44}
\DeclareMathSymbol{\s}{\mathalpha}{AMSb}{"53}

\newcommand{\CAT}{CAT}

\newcommand{\eps}{\varepsilon}
\renewcommand{\phi}{\varphi}

\newcommand{\e}{e}

\DeclareMathOperator{\supp}{supp}

\DeclareMathOperator{\tr}{tr}

\DeclareMathOperator{\Div}{div}

\DeclareMathOperator{\Ric}{Ric}
\DeclareMathOperator{\Lip}{Lip}
\DeclareMathOperator{\Hess}{Hess}

\DeclareMathOperator{\vol}{vol}

\renewcommand{\tilde}{\widetilde}

\DeclareMathOperator{\diam}{diam}

\newcommand{\ke}{{K}}
\newcommand{\lk}{k_2}
\newcommand{\uk}{\kappa}
\DeclareMathOperator{\Ch}{Ch}

\newcommand{\m}{m}

\DeclareMathOperator{\Ent}{Ent}

\DeclareMathOperator{\Dom}{Dom}
\DeclareMathOperator{\Cpl}{Cpl}
\def\S2k{\mathbb{S}^2_\kappa}
\def\RR{\mathbb R}
\def\co{\colon\thinspace}
\def\SS{\mathbb S}
\DeclareMathOperator{\curv}{curv}

\author{Vitali Kapovitch}\thanks{University of Toronto, vtk@math.utoronto.ca}
\author{Christian Ketterer}\thanks{University of Toronto, ckettere@math.toronto.edu}
\title{CD meets CAT}
\begin{document}
\begin{abstract}
We show that if a noncollapsed $CD(K,n)$ space $X$ with $n\ge 2$ has curvature bounded above by $\kappa$ in the sense of Alexandrov then $K\le (n-1)\kappa$ and $X$ is an Alexandrov space of curvature bounded below by $\ke-\uk (n-2)$. We also show that if a $CD(K,n)$ space $Y$ with finite $n$ has curvature bounded above then it is infinitesimally Hilbertian.
\end{abstract}
\maketitle
\tableofcontents
\section{Introduction}

It trivially follows from the definitions of sectional and Ricci curvature that if $(M^n,g)$ is a Riemannian manifold with $n\ge 2$ satisfying $\Ric_M\ge K, \sec_M\le \uk$ then $\uk (n-1)\geq \ke$ and $M$ also satisfies $\sec_M\ge \ke-\uk (n-2)$.
The main purpose of this paper is to show that the same holds true for metric measure spaces with intrinsically defined sectional and Ricci curvature bounds.

\begin{theorem}\label{main-thm}
Let $n\ge 2$ be a natural number and let $(X,d,\mathcal H_n)$ be a complete metric measure space which is $CBA(\uk)$ (has curvature bounded above by $\uk$ in the sense of Alexandrov) and satisfies $CD(K,n)$.
Then $\uk (n-1)\geq \ke$,  and $(X,d)$ is an Alexandrov space of curvature bounded below by $\ke-\uk (n-2)$. {In particular $X$ is infinitesimally Hilbertian.}
\end{theorem}
Examples given by  manifolds of constant sectional curvature show that the lower curvature bound provided by this theorem is optimal. 
%This theorem provides further evidence that  the notion of $RCD(K,N)$ is the right generalization of Riemannian manifolds with Ricci curvature bounded below.

Theorem \ref{main-thm} shows that $X$ has two sided curvature bounds in Alexandrov sense. By work of Alexandrov,  Berestovsky and Nikolaev (see~\cite{nikolaev}) this immediately gives the following corollary:

\begin{corollary}
Let  $(X,d,\mathcal H_n)$ be as in Theorem \ref{main-thm} then $X$ is a topological $n$-dimensional manifold with boundary and  $\Int X$ has a canonical open $C^3$ atlas of harmonic coordinates and a Riemannian metric $g$ which induces $d$ and such that 
 $g$ is in $C^{1,\alpha}\cap W^{2,p}$ in local harmonic charts for every $1\le p<\infty, 0<\alpha<1$.

\end{corollary}

Let us comment on the assumptions in the main theorem.

$CD(K,N)$ spaces for {$N\in [1,\infty)$} were introduced by Lott and Villani for $K=0$ in \cite{lottvillani}, and independently by Sturm for general $K\in \mathbb{R}$ in \cite{stugeo2}.
%\footnote{
%{\color{blue} Strictly speaking this is Sturm. Lott Villani only did the case $CD(0,N)$. But I guess they were aware on what was going on.}\color{red}
%This needs to be rephrased. Sturm didn't write the Lott-Villani paper }
Other curvature-dimension conditions are 
%the \textit{local curvature-dimension condition} $CD_{loc}(K-,N)$, 
the \textit{reduced curvature-dimension condition} $CD^*(K,N)$ \cite{bast}
and the \textit{entropic curvature dimension condition} $CD^e(K,N)$ \cite{erbarkuwadasturm} that
are simpler from an analytical viewpoint. For a Riemannian manifold each condition 
characterizes lower Ricci curvature bounds. However it is not known if the conditions $CD^*$ or $CD^e$ are in general equivalent to the original one by Sturm, or to each other. Moreover, in general they do 
not produce sharp estimates in geometric inequalities. 
 {Conditions $CD(K,N)$, $CD^*(K,N)$ and $CD^e(K,N)$ are known to be equivalent under the extra assumption that the space is 
 \emph{essentially non-branching} \cite{erbarkuwadasturm,cavmil}. In Proposition \ref{prop:nonbra} we prove that a {$CBA(\uk)$
 }
 %\footnote{\color{red} I changed CAT to CBA here because CAT is introduced later} 
 space which satisfies any of the conditions $CD(K,N)$, $CD^*(K,N)$ or $CD^e(K,N)$ with $N<\infty$  is  
\emph{non-branching} and therefore for {$CBA(\uk)$} spaces all these curvature-dimension conditions are equivalent.}

In the original version of this paper the main theorem had an extra assumption that $X$ is infinitesimally Hilbertian which might be considered a natural assumption in this setting, 
and under which all of the previous curvature-dimension conditions are equivalent as well. %(Remark \ref{rem-equiv-rcd}).}
However, as we show in  Theorem~\ref{CD+CAT implies RCD}  a space satisfying  any of the curvature dimension conditions
$CD(K,N)$,  $CD^*(K,N)$ or $CD^e(K,N)$, and $CBA(\uk)$ for $1\le N<\infty$, $K,\uk<\infty$ is automatically infinitesimally Hilbertian and hence $RCD(K,N)$. (Note that this includes the case $N=1$).

 Next, we exclude $n=1$ in the statement of the main theorem because 
if  $(X,d,m)$ is $RCD(K,1)$ then by ~\cite{Kit-Lak}  it is a point or a  smooth Riemannian 1-dimensional manifold (possibly with boundary) and thus is an Alexandrov space with curvature bounded below and above without an extra assumption of an a priori upper  curvature bound. 

Further, some assumptions on the measure $m$ in relation to $n$ in the main theorem are obviously necessary as the following simple example indicates 

\begin{example}\label{counterex}
{Let $f\co \RR^2\to\RR$ be given by $f(x)=4|x|^2$.  Let  $X=\overline{ B}_{\frac{1}{100}}(0)$. A simple computation shows that $\Ric^3_f\ge 2$ on $X$ where $\Ric^3_f$ is the $3$-Bakry-Emery Ricci tensor
of $(X,g_{Eucl},e^{-f}\mathcal{H}_2)$. Since all balls $B_r(0)$ are convex this  easily implies that 
%$f$ is $(2,1)$-convex on $X=\overline{ B}_{\frac{1}{100}}(0)$. Since $(\RR^2,d_{Eucl},\mathcal H_2)$ is $RCD(0,2)$, ${ B}_{\frac{1}{100}}(0)$ is convex in $\RR^2$,  $f$ is $(2,1)$-convex on $X$  it follows (e.g. by Proposition %\ref{prop-conf-conv}) that
$(X,d_{Eucl},e^{-f}\mathcal H_2)$ is $RCD(2,3)$.} On the other hand, $(X,d_{Eucl})$ is obviously $CBA(0)$. Thus, $X$ is $RCD(K, n)$ and $CBA(\uk)$ with $n=3, \ke=2, \uk=0$ but $K>\uk(n-1)$.
\end{example}
Note that while the space $X$ constructed in the above example  violates the conclusion of Theorem \ref{main-thm},  it nevertheless \emph{is} an Alexandrov space of curvature bounded below (it obviously has $\curv\ge 0$), just with a different lower curvature bound that the one claimed in Theorem ~\ref{main-thm}. In section~\ref{CD+CAT to RCD+CAT} we construct an example  of a compact $CBA(0)$, $RCD(-100,3)$ space which is not Alexandrov  of $\curv\ge \hat \uk$ for any  $\hat \uk$ (Example~\ref{cd+cat-not-alex}).

\begin{comment}
It follows from \cite{mondinonaber, kellmondino, giglipasqualetto} that if $(X,d,m)$ is $RCD(\ke,n)$ and $m(\mathcal R_n)>0$ (where $\mathcal R_n$ is the set of all points $x$ with  $T_xX\cong \RR^n$) then $m|_{\mathcal R_n}$ is absolutely continuous with respect to $\mathcal H_n$.  We conjecture that in such situation $\mathcal R_n$ has full measure in $X$, i.e.  $m$ is absolutely continuous with respect to $\mathcal H_n$ on all of $X$.
\end{comment}

In ~\cite{GP-noncol}  De~Philippis and Gigli (cf. also ~\cite{Kitab-noncol})  considered the class of $RCD(K,n)$ spaces where the background measure is $\mathcal H_n$.
Following  De~Philippis and Gigli we will call such spaces  \emph{noncollapsed}.
%By definition, in a weakly noncollapsed $RCD(K,n)$ space $m=f \mathcal H_n$ for some locally integrable $f$ with respect to $\mathcal H_n$. If $f=1$ then, again using the terminology of De~Philippis and Gigli, we will call $X$  \emph{noncollapsed}. 
%Conjecturally, a weakly noncollapsed $RCD(K,n)$ space $(X,d,f\mathcal H_n)$ satisfies $f=const $ a.e., i.e. $X$ is noncollapsed up to multiplying the background measure by a constant.

It follows from work of Cheeger--Colding \cite{cheegercoldingI} that a measured Gromov-Hausdorff limit of a sequence of complete $n$-dimensional Riemannian manifolds $(M_i,p_i)$ 
with convex boundary satisfying $\Ric_{M_i}\ge K, \vol (B(p_i,1))\ge v>0$ for some $K\in\RR,v>0$ is a noncollapsed $RCD(K,n)$ space in the above sense. 
This also follows from ~\cite{GP-noncol}  where it is shown more generally, 
that for any $v>0$ the class of noncollapsed $RCD(K,n)$ spaces $(X,d,m,p)$ satisfying $m(B_1(p))\ge 1$ is \emph{compact} in the pointed measured Gromov-Hausdorff topology.
%In particular, a noncollapsing limit of noncollapsed $RCD(K,n)$ spaces  is again noncollapsed.
{This includes the nontrivial statement that for a sequence $(X_i,d,\mathcal H_n,p_i)$ in the above class converging to $(X,d,m,p)$ the limit measure $m$ is automatically $\mathcal H_n$.}

Thus, noncollapsed $RCD(K,n)$ spaces are a natural synthetic generalization of noncollapsing Ricci limits.

The above discussion shows that requiring that the background measure be $\mathcal H_n$ in the context of Theorem \ref{main-thm} is a natural assumption.

Let us outline the structure of the proof of the main theorem. It consists of two  independent parts. Part one is to show that a space satisfying $CD(K,n)$ and $CBA(\uk)$  with finite $n$ is infinitesimally Hilbertian and hence $RCD(K,n)$. Part two is to show that a noncollapsed $RCD(K,n)$  space which is also  $CBA(\uk)$ is Alexandrov with  $curv\ge\ke-\uk (n-2)$.

Since small balls in $X$ are convex, using local-to-global results for RCD and Alexandrov spaces it's enough to prove both parts for small balls in $X$ which are $CAT(\uk)$ i.e. satisfy the upper curvature triangle comparison \emph{globally}. Thus, for most of the paper we only consider spaces $X$ which have small diameter and are $CAT(\uk)$ rather than $CBA(\uk)$.

The proof that $X$ satisfying $CD(K,n)$ and $\CAT(\uk)$  with finite $n$ must be infinitesimally Hilbertian consists of several steps. 
The main step is proving the splitting theorem (Proposition ~\ref{splitting theorem}) which says that if a space $X$ which is $CD(0,n)$ and $CAT(0)$ with $n<\infty$ then it must metrically split  as 
$Y\times \R$.

Recall that the usual scheme for proving the splitting theorem under various versions of nonnegative Ricci curvature involves a variation of the following argument ~\cite{cheegergromoll,giglisplitting}.

Let $\gamma\co \R\to X$  be a line in $X$. Consider the rays  $\gamma_+(t)=\gamma(t)$ and $\gamma_-(t)=\gamma(-t)$ for $t\ge 0$. 
Let $b_\pm$ be the corresponding Busemann functions.  From the triangle inequality it holds that $b=b_++b_-\ge 0$ on $X$. Also, $b|_\gamma\equiv 0$. Then the usual argument is to first show that $b_\pm$ are both superharmonic, hence $b$ is is superharmonic and hence it must be identically zero on $X$ by the maximum principle. However, this argument completely fails in our situation because knowing that $b_\pm$ are superharmonic \emph{does not} imply that $b$ is superharmonic too as the Laplace operator is not known to be linear yet - we are trying to prove that it is.

Our proof of the splitting theorem goes along very different lines. It relies on the Flat Strip Theorem for $CAT(0)$ spaces to conclude that $b\equiv 0$ and to get the splitting.

%The geodesic tangent cone  $T_p^gX$ embeds isometrically into any metric-measure tangent cone $T_pX$ but this embedding need not be onto for points in a general  $CD^*(K,n)$ and $CAT(\uk)$ space. As mentioned, Ahlfors regularity implies that it's always onto. 
  By \cite{gmr} "tangents of tangents are tangents" a.e., i.e. there is a set $A\subset X$ of full measure such that for \emph{every} point $p\in A$ for \emph{any} tangent cone $(T_pX,d_p,m_p)$ and \emph{any} point $y\in T_pX$  any tangent cone $T_y(T_pX)$ is a tangent cone at $p$.  Using the splitting theorem this easily implies that there exists a tangent cone at $p$ isometric to $\R^k$ for some $k\le n$.
  
  Now infinitesimal Hilbertianness of $X$ easily follows by an application of Cheeger's celebrated generalization of Rademacher's theorem to doubling metric measure spaces which satisfy the Poincar\'e inequality ~\cite{cheegerlipschitz}.

The second major part in the proof of the main theorem is showing that it holds if $X$ is $RCD(K,n)$, $CAT(\uk)$ and $m=\mathcal H_n$.

The obvious proof which works for Riemannian manifolds does not easily generalize as there is no notion of curvature or Ricci tensors on $X$.
Let us describe an argument that does generalize. Let $(M^n,g)$ be a complete Riemannian manifold with $\sec\le \uk, \Ric\ge K$. Fix any $\hat\kappa<\ke-\uk (n-2)$. To verify that $\sec_M\ge \hat\kappa$ it's enough to show that for any $p\in M$ the distance function to $p$ is more concave than the distance function in the simply connected space form of constant curvature $k$. For points $q$ near $p$ this is equivalent to checking that 

\begin{align}
\Hess ( d_p|_q)(V,V)\leq \cot_{\hat{\kappa}}(d_p(q)) \ \mbox{ for any unit } V\in T_qM \mbox{ orthogonal to } \nabla d_p
\end{align}

where $\cot_k(t)$ is the generalized cotangent function (see section \ref{CAT and CBB}  for the definition).

The condition that $\sec_M\le \uk$ implies that

\begin{align}
\Hess ( d_p)(V,V)\ge \cot_{\uk}(d_p(q)) \ \mbox{ for any unit } V\in T_qM \mbox{ orthogonal to } \nabla d_p
\end{align}

On the other hand,  since $\Ric_M\ge K$, by Laplace comparison we have that
%(which follows from the Bochner inequality and Ricatti comparison)  we have that

\begin{align}
\Delta d_p(q)=\sum_i\Hess ( d_p)(V_i,V_i)\leq  (n-1)\cot_{K/(n-1)}(d_p(q)) \ 
\end{align}
where $V_1,\ldots,V_{n-1}$ is an orthonormal basis of  $\nabla d_p^\perp\subset T_qM$.
 %This is done in section ~\ref{CD+CAT to RCD+CAT}.  

Combining the above inequalities gives that $\uk (n-1)\geq \ke$ and that for any $i=1,\ldots n-1$

\[
\Hess ( d_p)(V_i,V_i)\le  (n-1)\cot_{K/(n-1)}(d_p(q))-(n-2)\cot_{\uk}(d_p(q))\le \cot_{\hat{\kappa}}(d_p(q)) 
\]
when $d(p,q)$ is sufficiently small. Hence $\sec M\ge \hat\kappa$. Since $\hat\kappa<\ke-\uk (n-2)$ was arbitrary this shows that $\sec M\ge \ke-\uk (n-2)$.

There are a number of technical challenges in generalizing this argument to the setting of Theorem \ref{main-thm}. 
The first one is to get a \emph{lower} laplacian bound on the distance functions using the upper curvature bound. 
To do this we first show that the set of regular points $X_{reg}$ is open, \emph{convex} and is a topological $n$-manifold.
A crucial point in showing convexity of $X_{reg}$ is proving that the density function is semiconcave on $X$ (Lemma~\ref{dens-semiconcave}). 
This uses the CAT property of $X$ and need not be true for general noncollapsed $RCD(K,n)$ spaces.

By a homological argument the fact that $X_{reg}$ is a manifold implies that geodesics on $X_{reg}$ are locally extendible.
Once this has been established it follows from contraction properties of the inverse gradient flow of $d_p$ that $\Delta d_p$ is bounded below on 
$X_{reg}$. $RCD(K,n)$ condition implies that it's bounded above which implies 
that distance functions locally lie in the domain of the laplacian. This allows us to apply to the distance functions analytic tools we develop in 
Section \ref{sec:convexity-and-hessian} which relate convexity properties of  functions in the domain of the laplacian on $RCD$ spaces to bounds on their Hessians.  
Using the calculus of tangent modules developed by Gigli \cite{giglinonsmooth} and a result of Han \cite{hanriccitensor} that for a sufficiently regular function $f$ on a noncollapsed $RCD(K,n)$ 
space $\Delta  f=\tr \Hess f$, we are able to carry out the Riemannian argument that was outlined earlier to obtain the same concavity properties of distance functions locally on $X_{reg}$.  By a globalization result of Petrunin this implies that $X$ is Alexandrov.

 The paper is structured as follows.
In Section \ref{sec:preliminaries} we provide preliminaries on synthetic Ricci curvature bounds, calculus for metric measure spaces %, the notion of tangent module, the Bakry-Emery characterization of lower Ricci curvature bounds, Rectifiability, the generalized  Hessian operator, 
and curvature bounds for metric spaces in the sense of Alexandrov.

In Section \ref{sec:laplace bounds} we prove a lower Laplace bound for distance functions in the context of metric spaces which are topological manifolds and satisfy RCD and CBA bounds. %that satisfy an upper curvature bound and also a lower Ricci curvature bound provided the space is a topological manifold.

In Section \ref{sec:convexity-and-hessian} we establish a result that gives a characterization of local $\kappa$-convexity of Lipschitz functions that are in the domain of the Laplace operator, in terms of almost everywhere lower bounds for the Hessian.

In Section ~\ref{sec:RCD+CAT-to-Alexandrov}  we prove the main theorem  under an extra assumption that $X$ is infinitesimally Hilbertian making use of several tools and results for the Laplace operator and the tangent module of metric measure spaces. 
 
Finally, in Section~\ref{CD+CAT to RCD+CAT} we prove that a space satisfying $CD^*(K,n)$ and $CBA(\uk)$ for finite $n$ must be infinitesimally Hilbertian (Theorem~\ref{CD+CAT implies RCD}).
Combined with the results of Section ~\ref{sec:RCD+CAT-to-Alexandrov} this finishes the proof of the main theorem.

 \subsection{Acknowledgments} The authors are grateful to Robert Haslhofer for helpful conversations and comments. { The authors also want to thank Nicola Gigli for comments and remarks that helped to improve an earlier version of this article.}
The first author was supported in part by a Discovery grant from NSERC.
This work was done while the second author was participating in Fields Thematic Program on ``Geometric Analysis'' from July til December 2017. Both authors want to thank the Fields Institute for providing an excellent and stimulating research environment.

\section{Preliminaries}\label{sec:preliminaries}
\subsection{Curvature-dimension condition for metric measure spaces}
\begin{definition}
Let $[a,b]\subset \mathbb{R}$ be an interval. 
We say a lower semi-continuous function $u:[a,b]\rightarrow (-\infty,\infty]$ is $(K,N)$-convex for $K\in\mathbb{R}$ and $N\in (0,\infty]$ if
$u$ is absolutely continuous and
\begin{align*}
u''\geq K+\frac{1}{N}(u')^2
\end{align*}
%$e^{-\frac{1}{N}u}:[a,b]\rightarrow [0,\infty)$ satisfies 
%\begin{align*}
%v''+\frac{K}{N}v\leq 0 
%\end{align*}
\noindent
holds
in the distributional sense  where $\frac{1}{\infty}=:0$.
We say that $u$ is $K$-convex if $u$ is $(K,\infty)$-convex, and we say that $u$ is $K$-concave if $-u$ is $-K$-convex.
%We follow the sign convention in \cite{erbarkuwadasturm}.
%We say $u$ is $(0,1)$-convex iff $e^{-u}$ is $1$-concave.

If $N<\infty$, we define 
$
U_N(t)=e^{-\frac{1}{N}t}$.
Then, 
$u:[a,b]\rightarrow (-\infty,\infty]$ is $(K,N)$-convex if and only if $U_N(u)=:v:[a,b]\rightarrow [0,\infty)$ satisfies $v''\leq -K/N v$ on $[a,b]$ in the distributional sense.
\end{definition}
Let $(X,d)$ be a metric space. If $A\subset X$, the induced metric on $A$ is denoted by $d_A$. We say a rectifiable constant  speed curve $\gamma:[a,b]\rightarrow X$ is a \textit{minimizing geodesic} or just \textit{geodesic} if $\mbox{L}(\gamma)=d(\gamma(a),\gamma(b))$.
We say $(X,d)$ is a \textit{geodesic metric space} if for any pair $x,y\in X$ there exists a geodesic between $x$ and $y$. For a geodesic $\gamma$ between points $x,y\in X$ we will also use the notation $[x,y]$ 
where in this case we think of the geodesic as its
image in $X$. Similarly $]x,y[=[x,y]\backslash \left\{x,y\right\}$.
\begin{definition}\label{def:metricconvex}
Let $(X,d)$ be a metric space, and let $V:X\rightarrow (-\infty,\infty]$ be a lower semi-continuous function. We set $\Dom V:= \left\{x\in X: V(x)<\infty\right\}$. Let $K\in \mathbb{R}$ and $N\in (0,\infty]$.
\begin{itemize}
 \item[(i)]
 We say that $V$ is weakly $(K,N)$-convex if for every pair $x,y\in \Dom V$ there exists
a unit speed geodesic $\gamma:[0,d(x,y)]\rightarrow X$ between $x$ and $y$ such that $V\circ \gamma:[0,d(x,y)]\rightarrow (-\infty,\infty]$ is $(K,N)$-convex.
%In this case $\gamma_t\in \Dom V$ for any $t\in [0,1]$.
\item[(ii)]
If $(X,d)$ is a geodesic metric space
we say $V$ is $(K,N)$-convex if $V\circ\gamma$ is $(K,N)$-convex for any unit speed geodesic $\gamma:[0,L]\rightarrow \Dom V$. 
\item[(iii)]
We say $V:X\rightarrow (-\infty,\infty]$ is semi-convex if for any $x\in X$ there exists a neighborhoud $U$ of $x$ and $K\in\mathbb{R}$ such that $V|_U$ is $K$-convex.
%\item[(iv)]
%We say $V:X\rightarrow[-\infty,\infty)$ is $(K,N)$-concave if $-V$ is $(K,N)$-convex.
\end{itemize}
\end{definition}
$\mathcal{P}^2(X)$ denotes the set of Borel probability measures $\mu$ on $(X,d)$ such that $\int_Xd(x_0,x)^2d\mu(x)<\infty$ for some $x_0\in X$. For any pair $\mu_0,\mu_1\in \mathcal{P}^2(X)$ 
we denote with $W_2(\mu_0,\mu_1)$ the \textit{$L^2$-Wasserstein distance} that is finite and defined by
\begin{align}\label{def:W2}
W_{2}(\mu_{1}, \mu_{2})^{2}:=\inf_{\pi \in  \Cpl(\mu_1,\mu_2)}\int_{X^2} d^{2}(x,y) d\pi(x,y),
\end{align}
where $\Cpl(\mu_1,\mu_2)$ is the set of all couplings between $\mu_1$ and $\mu_2$, i.e. of all the probability measures $\pi\in \mathcal{P}(X^2)$ 
such that $(P_i)_{\sharp}\pi=\mu_i$, $i=1,2$,   $P_1,P_2$ being the projection maps. $(\mathcal{P}^{2}(X),W_2)$ becomes a separable metric space that is
a geodesic metric space provided $X$ is a geodesic metric space.
A coupling $\pi\in \Cpl(\mu_1,\mu_2)$ is optimal if it is a minimizer for (\ref{def:W2}).
Optimal couplings always exist
We call the metric space $(\mathcal{P}^2(X),W_2)$ the \textit{$L^2$-Wasserstein space} of $(X,d)$.
The subspace of probability measures with bounded support is denoted with
$\mathcal{P}^2_b(X)$.
\begin{definition}
A \textit{metric measure space} is a triple $(X,d,\m)=:X$ where $(X,d)$ is a complete and separable metric space and $\m$ is a locally finite measure.
\end{definition}
%
%Let $(X,d)$ be a metric space that is complete and separable, and let $\m$ be a locally finite Borel measure. The triple $(X,d,\m)=:X$ is called a \textit{metric measure space}.
The space of $\m$-absolutely continuous probability measures in $\mathcal{P}^2(X)$ is denoted by $\mathcal{P}^2(X,\m)$.
The \textit{Shanon-Boltzmann entropy}  of a metric measure space $(X,d,\m)$ is defined as 
\begin{align*}\Ent_{\m}:\mathcal{P}^2_b(X)\rightarrow (-\infty,\infty],\ \
\Ent_{\m}(\mu)= \int \log \rho d\mu\mbox{ if $\mu=\rho\m \mbox{ and } (\rho\log \rho )_+\mbox{ is $\m$-integrable}$,}
\end{align*}
and $\infty$ otherwise.  Note $\Dom\Ent_{\m}\subset \mathcal{P}^2(X,\m)$, and $\Ent_{\m}:\mathcal{P}^2_b(X)\rightarrow (-\infty,\infty]$ is lower semi-continuous. By Jensen's inequality one has $\Ent_{\m}\mu\geq -\log\m(\supp\mu)$ if $\m(\supp\mu)<\infty$.

\begin{definition}[\cite{stugeo1,lottvillani,erbarkuwadasturm}]
A metric measure space $(X,d,\m)$ satisfies the \textit{curvature-dimension condition} $CD(\ke,\infty)$ for $\ke\in \mathbb{R}$ if 
$\Ent_{\m}$ is weakly $\ke$-convex.

A metric measure space $(X,d,\m)$ satisfies the \textit{entropic curvature-dimension condition} $CD^e(\ke,N)$ for $\ke\in \mathbb{R}$ and $N\in (0,\infty)$ if 
$\Ent_{\m}$ is weakly $(\ke,N)$-convex.
%for every pair $\mu_0,\mu_1\in \mathcal{P}^2(X,d,\m)$ there exist a 
%geodesic $(\mu_t)_{t\in [0,1]}$ in $\mathcal{P}^2(X)$ such that 
%\begin{align*}
%t\in [0,1]\mapsto \Ent\circ\mu_t
%\end{align*}
%is $(\ke\Theta^2,N)$-convex for $\Theta=W_2(\mu_0,\mu_1)$.
\end{definition}
The \textit{$N$-Renyi entropy} is defined as
\begin{align*}
S_N(\cdot|\m):\mathcal{P}^2_b(X)\rightarrow (-\infty,0],\ \ S_N(\mu|\m)=-\int \rho^{1-\frac{1}{N}}d\m\ \mbox{ if $\mu=\rho\m$, and }0\mbox{ otherwise}.
\end{align*}
Note that $\mu=\rho\m\in \mathcal{P}(X,\m)$ implies $\rho\in L^{1-\frac{1}{N}}(\m)$, and therefore $S_N$ is well-defined.
$S_N$ is lower semi-continuous, and $S_N(\mu)\geq - \m(\supp\mu)^{\frac{1}{N}}$ by Jensen's inequality.
%{\color{red} Should we also define $CD^*(K,N)$ since it's used in the main theorem?}

\begin{definition}
For $\kappa\in \mathbb{R}$ we define $\cos_{\kappa}:[0,\infty)\rightarrow \mathbb{R}$ as the solution of 
\begin{align*}
v''+\kappa v=0 \ \ \ v(0)=1 \ \ \& \ \ v'(0)=0.
\end{align*}
$\sin_{\kappa}$ is defined as solution of the same ODE with initial value $v(0)=0 \ \&\ v'(0)=1$. That is 
\begin{align*}
\cos_{\kappa}(x)=\begin{cases}
 \cosh (\sqrt{|\kappa|}x) & \mbox{if } \kappa<0\\
1& \mbox{if } \kappa=0\\
\cos (\sqrt{\kappa}x) & \mbox{if } \kappa>0
                \end{cases}
                \quad
   \sin_{\kappa}(x)=\begin{cases}
\frac{ \sinh (\sqrt{|\kappa|}x)}{\sqrt{|\kappa|}} & \mbox{if } \kappa<0\\
x& \mbox{if } \kappa=0\\
\frac{\sin (\sqrt{\kappa}x)}{\sqrt \kappa} & \mbox{if } \kappa>0
                \end{cases}                 
                \end{align*}
Let $\pi_\kappa$ be the diameter of a simply connected space form $\S2k$ of constant curvature $\kappa$, i.e.
\[
\pi_\kappa= \begin{cases}
 \infty \ &\textrm{ if } \kappa\le 0\\
\frac{\pi}{\sqrt \kappa}\ &  \textrm{ if } \kappa> 0

\end{cases}
\]
For $K\in \mathbb{R}$, $N\in (0,\infty)$ and $\theta\geq 0$ we define the \textit{distortion coefficient} as
\begin{align*}
t\in [0,1]\mapsto \sigma_{K,N}^{(t)}(\theta)=\begin{cases}
                                             \frac{\sin_{K/N}(t\theta)}{\sin_{K/N}(\theta)}\ &\mbox{ if } \theta\in [0,\pi_{K/N}),\\
                                             \infty\ & \ \mbox{otherwise}.
                                             \end{cases}
\end{align*}
Note that $\sigma_{K,N}^{(t)}(0)=t$.
Moreover, for $K\in \mathbb{R}$, $N\in [1,\infty)$ and $\theta\geq 0$ the \textit{modified distortion coefficient} is defined as
\begin{align*}
t\in [0,1]\mapsto \tau_{K,N}^{(t)}(\theta)=\begin{cases}
                                            \theta\cdot\infty \ & \mbox{ if }K>0\mbox{ and }N=1,\\
                                            t^{\frac{1}{N}}\left[\sigma_{K,N-1}^{(t)}(\theta)\right]^{1-\frac{1}{N}}\ & \mbox{ otherwise}.
                                           \end{cases}\end{align*}
\end{definition}
\begin{definition}[\cite{stugeo2,lottvillani,bast}]
We say $(X,d,\m)$ satisfies the \textit{curvature-dimension condition} $CD(\ke,N)$ for $\ke\in \mathbb{R}$ and $N\in [1,\infty)$ if for every $\mu_0,\mu_1\in \mathcal{P}_b^2(X,\m)$ 
there exists an $L^2$-Wasserstein geodesic $(\mu_t)_{t\in [0,1]}$ and an optimal coupling $\pi$ between $\mu_0$ and $\mu_1$ such that 
$$
S_N(\mu_t|\m)\leq -\int \left[\tau_{K,N}^{(1-t)}(d(x,y))\rho_0(x)^{-\frac{1}{N}}+\tau_{K,N}^{(t)}(d(x,y))\rho_1(y)^{-\frac{1}{N}}\right]d\pi(x,y)
$$
where $\mu_i=\rho_id\m$, $i=0,1$.

We say $(X,d,\m)$ satisfies the \textit{reduced curvature-dimension condition} $CD^*(\ke,N)$ for $\ke\in \mathbb{R}$ and $N\in (0,\infty)$ if we replace 
in the previous definition the modified distortion coefficients $\tau^{(t)}_{K,N}(\theta)$ by the usual distortion coefficients $\sigma_{K,N}^{(t)}(\theta)$.

If $K=0$, the condition $CD(K,N)$ coincides with the condition $CD^*(K,N)$ and is simply convexity of the $N$-Renyi entropy functional.
\end{definition}
\begin{remark}
We note that if a metric measure space $(X,d,\m)$ satisfies a curvature dimension condition $CD(K,N)$, $CD^*(K,N)$ or $CD^e(K,N)$ for $N<\infty$,
the support $\supp\m$ of $\m$ with the induced metric $d_{\supp\m}$ becomes a geodesic space. 
This follows since $(\supp\m,d_{\supp\m})$ is complete and a curvature-dimension condition 
yields that $\supp\m$ is a length space, and is locally compact by
Bishop-Gromov-type comparison that holds in any case \cite{stugeo2,lottvillani,bast,erbarkuwadasturm}.
In this paper we will always assume that $\supp m=X$.
\end{remark}
\begin{theorem}[\cite{stugeo1,stugeo2,lottvillani,gmsstability,erbarkuwadasturm}]\label{th:stability+BM}
All of the previous curvature-dimension conditions $CD(K,N)$, $CD^*(K,N)$ and $CD^e(K,N)$  are stable under pointed measured Gromov-Hausdorff convergence, and yield  Brunn-Minkowski-type inequalities if $N$ is finite.
In the case $K=0$ the latter is the same statement for every curvature-dimension condition: For each pair of measurable subsets $A_0,A_1\subset X$ it holds that
\begin{align}\label{ineq:BM}
\m(A_t)^{\frac{1}{N}}\geq (1-t)\m(A_0)^{\frac{1}{N}}+t\m(A_1)^{\frac{1}{N}}
\end{align}
where $A_t$ is the set of $t$-midpoints of geodesics with endpoints in $A_0$ and $A_1$ respectively. 
\end{theorem}
The next fact, the next lemma and the next theorem collect a number of important properties of spaces that satisfy a curvature-dimension condition.
\begin{fact}[\cite{stugeo2,bast,erbarkuwadasturm}]\label{fact:cd}
$(X,d,m)$ satisfies a condition $CD(\ke,N)$ for $\ke\in\mathbb{R}$ and $N\in (0,\infty)$. Then
\begin{itemize}
 \item[(i)] $(\supp\m,d_{\supp\m})$ is locally compact, a geodesic space and satisfies a Bishop-Gromov-type comparison and a doubling property.\medskip
 \item[(ii)] $(X,\alpha d, \beta\m)$ satisfies the condition $CD(\alpha^{-2}\ke,N)$ for every $\alpha,\beta>0$.\medskip
 \item[(iii)] If\ $U\subset X$ is geodesically convex and closed, $(U,d_{U},\m|_U)$ satisfies the condition $CD(\ke,N)$.\medskip
 \item[(iv)] If $(X,d,\m)$ satisfies a condition $CD(K,N)$ for $N<\infty$, then it satisfies $CD(K,\infty)$.
\end{itemize}
Each of the previous statements holds for $CD^*(K,N)$ and $CD^e(K,N)$ as well.
\end{fact}
\begin{lemma}[\cite{erbarkuwadasturm}]\label{conv->CD}
Let $(X,d,\m)$ be a metric measure space satisfying a condition $CD(\ke,\infty)$, $CD(\ke,N)$, $CD^*(\ke,N)$ or $CD^e(\ke,N)$ for some $\ke\in\mathbb{R}$ and $N\in (0,\infty)$, 
and let $f:X\rightarrow \mathbb{R}$ be $\kappa$-convex for $\alpha\in\mathbb{R}$ and bounded from below. Then the metric measure space
$(X,d,e^{-f}\m)$ satisfies the condition $CD(\ke+\kappa,\infty)$.
\end{lemma}
\begin{theorem}[\cite{rajala1,rajala2}]\label{th:rajala}
A metric measure space $(X,d,\m)$ that satisfies $CD(\ke,N)$, $CD^*(\ke,N)$ or $CD^e(\ke,N)$ for $\ke\in \mathbb{R}$, $N\in (0,\infty)$, admits a weak local 1-1 Poincar\'e inequality.
\end{theorem}
{
\cite{rajala1,rajala2} proves this for $CD(\ke,N)$ and $CD^*(\ke,N)$ but it's easy to see that the proof also works for $CD^e(\ke,N)$.}

\subsection{Cheeger energy and calculus for metric measure spaces} 
In the following we present the framework for calculus on metric measure spaces by Ambrosio, Gigli and Savar\'e \cite{agslipschitz,agsheat,agsriemannian,giglistructure}.
Let $(X,d,m)$ be a metric measure space, and $\Lip(X)$ and $\Lip_b(X)$ be the space of Lipschitz functions, and bounded Lipschitz functions respectively. For $f\in \Lip(X)$ its local slope is
\begin{align*}
\mbox{Lip}(f)(x)=\limsup_{y\rightarrow x}\frac{|f(x)-f(y)|}{d(x,y)}, \ \ x\in X.
\end{align*}
If $f\in L^2(m)$ a function $g\in L^2(\m)$ is called \textit{relaxed gradient} if there exists sequence of Lipschitz functions $f_n$ which $L^2$-converges to $f$, and there exists $\tilde{g}$ such that 
$\mbox{Lip}f_n$ weakly converges to $\tilde{g}$ in $L^2(m)$ and $\tilde{g}\leq g$ $\m$-a.e.\ . We say $g$ is the \textit{minimal relaxed gradient} of $f$ if it is a relaxed gradient and minimal w.r.t. to the $L^2$-norm amongst all relaxed gradients.

The \textit{Cheeger energy} $\Ch^X:L^2(m)\rightarrow [0,\infty]$ is defined as 
\begin{align*}
2\Ch^X(f)=\liminf_{f_n\in \Lip(X)\overset{L^2(\m)}{\longrightarrow}f}\int_X \mbox{Lip}(f_n)^2 dm.
\end{align*}
The space of \textit{$L^2$-Sobolev functions} is then $$W^{1,2}(X):= D(\Ch^X):= \left\{ f\in L^2(\m): \Ch^X(f)<\infty\right\}.$$
For any $f\in W^{1,2}(X)$ there exists a minimal relaxed that is denoted with $|\nabla f|$ and unique up to set of measure $0$. One also calls $|\nabla f|$ the \textit{minimal weak upper gradient} of $f$. Then, one has

\begin{align*}
\Ch^X(f)=\frac{1}{2}\int |\nabla f|^2 d\m.
\end{align*}
The space 
$W^{1,2}(X)$ equipped with the norm 
$
\left\|f\right\|_{W^{1,2}(X)}^2=\left\|f\right\|^2_{L^2}+\left\||\nabla f|\right\|_{L^2}^2
$
is a Banach space.
If $W^{1,2}(X)$ is a Hilbert space, we say $(X,d,m)$ is \textit{infinitesimally Hilbertian.}
\begin{remark}
Note that in general $|\nabla u|\neq \Lip u$ for a Lipschitz function $u$ unless $(X,d,\m)$ satisfies a Poincar\'e inequality and a doubling property \cite[Theorem 5.1]{cheegerlipschitz}.
By~\cite{rajala1,rajala2} spaces satisfying $CD(K,N)$, $CD^*(K,N)$ or $CD^e(K,N)$ with $N<\infty$ do satisfy a  1-1 Poincar\'e inequality (Theorem \ref{th:rajala}).
Hence for such spaces \cite{cheegerlipschitz} applies and $|\nabla u|= \Lip u$  a.e. for any Lipschitz $u$.
\end{remark}
If $(X,d,\m)$ is infinitesimally Hilbertian, by polarization of $|\nabla f|^2$ we can define 
\begin{align*}
(f,g)\in W^{1,2}(X)^2\mapsto \langle \nabla f,\nabla g\rangle := \frac{1}{4}|\nabla (f+g)|^2-\frac{1}{4}|\nabla (f-g)|^2\in L^1(m).
\end{align*}
%In particular, the bilinear form $\mathcal{E}(f,g)=\int \langle \nabla f,\nabla g\rangle dm $ becomes a strongly local, quasi-regular Dirichletform.
We say that $f\in W^{1,2}(X)$ is in the domain of the \textit{Laplace operator $\Delta$} if there exists a function $g=:\Delta f\in L^2(m)$ such that for every $h\in W^{1,2}(X)$
\begin{align*}
\int \langle\nabla f,\nabla h\rangle dm=-\int h\Delta f dm.
\end{align*}
In this case we say that $f\in D(\Delta)$. The vector space $D(\Delta)$ is equipped with the operator norm 
\begin{align*}
\left\|f\right\|_{D(\Delta)}^2=\left\|f\right\|_{L^2}^2+\left\|\Delta f\right\|_{L^2}^2.
\end{align*}
Convergence in $D(\Delta)$ implies convergence in $W^{1,2}(X)$.
If $\mathbb{V}$ is any subspace of $L^2(\m)$ and we have that $\Delta f\in \mathbb{V}$, we write $f\in D_{\mathbb{V}}(\Delta)$. $(P_t)_{t\in (0,\infty)}$ denotes the heat semi-group associated to $\Delta$.

More generally, -- assuming $X$ is locally compact -- if $U$ is an open subset of $X$, we say $f\in W^{1,2}(X)$ is in the domain $D({\bf \Delta},U)$ of the \textit{measure valued Laplace} ${\bf \Delta}$ on $U$ if there exists a signed Radon measure $\mu=:{\bf \Delta}f$ such that
for every Lipschitz function $g$ with bounded support in $U$ we have
\begin{align*}
\int\langle \nabla g,\nabla f\rangle dm = -\int g d{\bf \Delta}f.
\end{align*}
%In analogy to the previous notation we also write $D_{L^2(m)}({\bf \Delta})=D(\Delta)$.
\begin{proposition}[\cite{giglistructure}]\label{prop:chainrulelaplacian}
Let $(X,d,\m)$ be an infinitesimally Hilbertian metric measure space, let $f\in D({\bf \Delta},U)\cap \Lip(X)$ for an open subset $U\subset X$, let $I\subset \mathbb{R}$ be an open subset, assume $\m(f^{-1}(\mathbb{R}\backslash I))=0$, and let $\phi\in C^2(I)$. 
Then $\phi\circ f\in D({\bf \Delta},U)$ and 
\begin{align*}
{\bf\Delta}(\phi\circ f)= \phi'(f){\bf\Delta} f + \phi''(f)|\nabla f|^2\m \ \mbox{ on }U.
\end{align*}

\end{proposition}

\begin{definition}[\cite{agsriemannian,giglistructure}]\label{def:rcd}
A metric measure space $(X,d,m)$ satisfies the Riemannian curvature-dimension condition $RCD(\ke,N)$ for $\ke\in \mathbb{R}$ and {$N\in [1,\infty]$} if it satisfies the curvature-dimension
condition $CD(\ke,N)$, and if it is infinitesimally Hilbertian.
%
%Moreover,
%a metric space $(X,d)$ satisfies the noncollapsed Riemannian curvature-dimension condition $ncRCD(\ke,n)$ for $\ke\in \mathbb{R}$ and $n\in\mathbb{N}$ if the metric measure space $(X,d,\mathcal{H}^n)$ satisfies the condition 
%$RCD(\ke,n)$.
\end{definition}
\begin{remark}
For $N=\infty$ the condition was first introduced and studied in \cite{agsriemannian}, for $N<\infty$ in \cite{giglistructure}.
\end{remark}
\begin{remark}\label{rem-equiv-rcd}
By the globalization theorem of Cavalletti and Milman \cite{cavmil}, and by results in \cite{erbarkuwadasturm}, \cite{agsriemannian} and \cite{rajalasturm} in the previous definition it is equivalent to require the condition 
$CD^*(\ke,N)$ or the condition $CD^e(\ke,N)$. More precisely, since each condition implies $CD(K,\infty)$, together with infinitesimally Hilbertianness $(X,d,\m)$ satisfies the condition $RCD(K,\infty)$ in 
the sense of \cite{agsriemannian}. Therefore, the Boltzmann-Shanon entropy is even strongly $K$-convex, and hence $(X,d,\m)$ is essentially non-branchning by \cite{rajalasturm}. Then, first we know that
$CD^*(K,N)$ is equivalent to $CD^e(K,N)$ by \cite{erbarkuwadasturm}. Second, from the globalization result in \cite{cavmil} we have that $CD^*(K,N)$ is equivalent to $CD(K,N)$.
\end{remark}

\begin{definition}
Further,  we define $\md_\kappa:[0,\infty)\rightarrow [0,\infty)$ as the solution of 
\begin{align*}
v''+ \kappa v=1 \ \ \ v(0)=0 \ \ \&\ \ v'(0)=0.
\end{align*}
More explicitly
\begin{align*}
\md_{\kappa}(x)=\begin{cases} \frac{1}{\kappa}\left(1-\cos_{\kappa}x\right)\ &\ \mbox{ if } \kappa\neq 0,\\
                 \frac{1}{2}x^2\ &\ \mbox{ if }\kappa =0.
                \end{cases}
                \end{align*}
\end{definition}
\begin{theorem}[\cite{giglistructure}]\label{th:upperbound}
Assume $(X,d,\m)$ satisfies the condition $RCD(\ke,N)$ for $N<\infty$, and for $x\in X$ we define $d_x:X\rightarrow [0,\infty)$ via $d_x(y)=d(x,y)$. 
Then
\begin{align*}
\frac{1}{2}d_y^2\in D({\bf\Delta})\ \ \mbox{ and }\ \ {\bf\Delta}\frac{1}{2}d_y^2\leq \left[1 + (N-1)d_y\cot_{K/(N-1)}d_y\right]\m
\end{align*}
where $\cot_k(x)=\frac{\cos_k(x)}{\sin_k(x)}$.
In particular, for $\ke=0$ the estimate is precisely ${\bf\Delta}\frac{1}{2}d_y^2\leq N\m$.
\end{theorem}
\begin{remark}
The Theorem is actually proven by Gigli under the condition $CD(K,N)$ and it is sharp in this context. 
%Therefore it also holds and is sharp for spaces satisfying $RCD(K,N)$ - no matter if we require $CD(K,N)$, $CD^*(K,N)$ or $CD^e(K,N)$ in Definition \ref{def:rcd} by \cite{cavmil}.
%This can be seen in two ways. On the one hand, there is a recent result by Cavalletti and Milman \cite{cavmil} that proves that the reduced (and therefore the entropic version as in our definition) of the curvature-dimension condition is equivalent to the one of Sturm in 
%\cite{stugeo2} provided the space is essentially non-branching. The latter always is true for $RCD$-spaces. 
%On the other hand,
Moreover, for the sharp comparison result of $\frac{1}{2}d_y^2$ only the weaker $MCP(\ke,N)$-condition is needed that follows from
the reduced curvature-dimension condition $CD^*(\ke,N)$ by \cite{cavallettisturm} provided the space is essentially non-branching.
\end{remark}
\begin{corollary}\label{cor-mdk-laplace-comp}
Assume $(X,d,\m)$ satisfies the condition $RCD(\ke,N)$ for $N<\infty$. %and if $K>0$ assume in addition that $\diam_X<\pi_K$.
Then for any $y\in X$ we have
\begin{align*}\label{mdk-laplace-comp}
\md_{\kappa}d_y\in D({\bf\Delta})\ \ \mbox{ for any $\kappa\in\RR$ and }\ \ 
{\bf\Delta}\md_{\ke/(N-1)}d_y+ \frac{N\ke}{N-1}\md_{\ke/(N-1)}d_y\leq N\m.
\end{align*}

%{\color{red}Furthermore, $\md_\kappa d_y\in D({\bf\Delta})$ for any $\kappa\in\RR$.}
\end{corollary}
\begin{proof}
 Indeed, By the chain rule Theorem~\ref{th:upperbound} implies (see also ~\cite[Corollary 5.15]{giglistructure}) that $d_y\in D({\bf\Delta}, X\backslash \{y\}) $ and
 \[
 {\bf\Delta} d_y\le [(N-1)\cot_{K/(N-1)} d_y]m \mbox{ on } X\backslash \{y\}
 \] 
 Applying the chain rule one more time this yields
 
\begin{align}\label{eq-mdk-laplace-1}
 \md_{\ke/(N-1)}d_y\in D({\bf\Delta},X\backslash \{y\}) \mbox{ and }
{\bf\Delta}\md_{\ke/(N-1)}d_y+ \frac{N\ke}{N-1}\md_{\ke/(N-1)}d_y\leq N\m \mbox{ on } X\backslash \{y\}.
\end{align}
 
  Further note that \begin{align*}
\md_{\kappa}(x)= \frac{1}{2}x^2 - \frac{\kappa}{24}x^4 + \dots \ . 
\end{align*}

and hence $\md_{\kappa}(x)=\phi_{\kappa}(x^2/2)$ where $\phi_{\kappa}$ is a smooth function on $\RR$ given by

\[
\phi_{\kappa}(x)=x - \frac{\kappa}{6}x^2 + \dots 
\]

Therefore $\md_{\kappa}d_y=\phi_k(\frac{d_y^2}{2})\in D({\bf\Delta})$ for any $\kappa\in \R$ and \eqref{eq-mdk-laplace-1} can be improved to

\[
{\bf\Delta}\md_{\ke/(N-1)}d_y+ \frac{N\ke}{N-1}\md_{\ke/(N-1)}d_y\leq N\m \mbox{ on all of } X.
\]

\begin{comment}
This follows, since we can write $\md_{\kappa}(x)=\phi_{\kappa}(x^2/2)$ for some function $\phi_{\kappa}$ and any $\kappa\in \mathbb{R}$. Indeed, if we compute the Taylor expansion for $\md_{\kappa}$ at $0$, we obtain
\begin{align*}
\md_{\kappa}(x)= \frac{1}{2}x^2 - \frac{\kappa}{24}x^4 + \dots \ . 
\end{align*}
Hence, $\phi_{\kappa}(x)=x - \frac{\kappa}{6}x^2 + \dots $ is smooth on $\mathbb{R}$ and does the job. Then, we can apply the chain rule formula for the measure valued Laplacian.
\end{comment}
\end{proof}
%\begin{remark} If $\ke\neq 0$, then
%\begin{align*}
%{\bf\Delta}\left(\frac{N-1}{\ke}\cos_{\ke/(N-1)}d_y\right)+
%N \cos_{\ke/(N-1)}d_y\m\geq 0
%\end{align*}

%\end{remark}
\subsection{Tangent modules of metric measure spaces}
In this section we present a general construction of {\it tangent spaces} for metric measure spaces due to Gigli (inspired by an idea of Weaver \cite{weaver}).
%For later purposes, we also introduce the so-called $L^2$-cotangent module $L^2(T^*X)$ of $(X,d,m)$ that is - formally - the family of measurable vector fields on $X$. 
\begin{definition}[\cite{giglinonsmooth}]
Let $(X,d,\m)$ be a metric measure space, and let $\mathcal{M}$ be a Banach space. $\mathcal{M}$ is called an $L^2(\m)$-normed $L^{\infty}(\m)$-module provided it is endowed with a bilinear
map $L^{\infty}(\m)\times \mathcal{M}:(f,v)\mapsto fv\in\mathcal{M}$, and a 
function $|\cdot|:\mathcal{M}\rightarrow L^2(\m)^+$ which satisfy the following properties:
\begin{itemize}
\item[(i)] $f(gv)=(fg)v$ for all $f,g\in L^{\infty}(\m)$ and $v\in\mathcal{M}$,
\medskip
\item[(ii)] ${\bf 1}v=v$ for any $v\in\mathcal{M}$ where ${\bf 1}\in L^{\infty}(\m)$ is the function equal $1$, 
\medskip
\item[(iii)] $\left\||v|\right\|_{L^2}=\left\|v\right\|_{\mathcal{M}}$ for any $v\in \mathcal{M}$,
\medskip
\item[(iv)] $|fv|=|f||v|$ $\m$-a.e. for any $f\in L^{\infty}(\m)$ and $v\in\mathcal{M}$.
\end{itemize}
\end{definition}
\noindent
Consider a Borel measurable set $A\subset X$. The restriction $\mathcal{M}|_A$ of $\mathcal{M}$ to $A$ is defined as 
\begin{align*}
\mathcal{M}|_A=\left\{v\in \mathcal{M}:1_{A^c}v=0\ \m\mbox{-a.e.}\right\}.
\end{align*}
$\mathcal{M}|_A$ inherits the structure of $L^2(\m)$-normed $L^{\infty}(\m)$-module.

Given two $L^2(\m)$-normed $L^{\infty}(\m)$-modules $\mathcal{M}$ and $\mathcal{N}$ we say that a map $T:\mathcal{M}\rightarrow \mathcal{N}$ is \textit{module morphism} provided $T$ is linear, continuous and it satisfies
\begin{align*}
T(fv)=fT(v) \mbox{ for every }f\in L^{\infty}(\m) \ \mbox{ and every }v\in \mathcal{M}.
\end{align*}
\begin{definition}
Given an $L^2(\m)$-normed $L^{\infty}(\m)$-module $\mathcal{M}$ we define the dual module $\mathcal{M}^*$ as the space of all 
maps $T$ from $\mathcal{M}$ to $L^1(\m)$ that are $L^{\infty}$-linear (that is additive and $L^{\infty}$-homogenuous) and continuous.
$\mathcal{M}^*$ again has a natural structure of an $L^2(\m)$-normed $L^{\infty}(\m)$-module. For details we refer to \cite{giglinonsmooth}.
\end{definition}
\begin{definition}\label{def:hilbertmodule}
We call an $L^{\infty}(\m)$-module $\mathcal{M}$ an \textit{Hilbert module} if $\mathcal{M}$ is an Hilbert space. In this case $\mathcal{M}$ becomes an $L^2(\m)$-normed $L^{\infty}(\m)$-module, and the pointwise norm $|\cdot|$ satisfies
\begin{align*}
|v+w|^2 + |v-w|^2 = 2|v|^2 + 2|w|^2 \ \mbox{ for any }v, w\in \mathcal{M}.
\end{align*}
We define the pointwise inner product $\mathcal{M}^2 \rightarrow L^1(\m)$ via $4\langle v,w\rangle = |v+w|^2-|v-w|^2$. Following \cite{giglinonsmooth} 
we also note that $\mathcal{M}$ and $\mathcal{M}^*$ are canonical isomorphic as Hilbert modules.
\end{definition}

\begin{definition}\label{def:basis}
Let $\mathcal{M}$ be an $L^2(\m)$-normed $L^{\infty}(\m)$-module, and let $A\subset X$ be a Borel set such that $\m(A)>0$. We say that
\begin{itemize}
 \item[(i)] $v_1,\dots,v_n\in \mathcal{M}$ are \textit{independent} on $A$ if for $f_1,\dots,f_n\in L^{\infty}(\m)$ it holds that
 \begin{align*}
 1_A\sum_{i=1}^n f_i v_i=0 \ \ \Rightarrow \ \ f_1,\dots,f_n=0\ \m\mbox{-a.e. in $A$,}
 \end{align*}
 \item[(ii)] $S\subset \mathcal{M}$ generates $\mathcal{M}|_{A}$ if $\mathcal{M}|_A$ is the $L^2$-closure of elements $v$ in $\mathcal{M}|_A$ such that there is a decomposition of $\left\{A_i\right\}_{i\in\mathbb{N}}$ of $A$, vectors $v_{i,1},\dots,v_{i,m_i}\in S$, and functions
 $f_{i,1},\dots,f_{i,m_i}\in L^{\infty}(\m)$ which satisfy
 \begin{align*}
 1_{A_i}v=\sum_{k=1}^{m_i}f_{i,k}v_{i,k} \ \m\mbox{-a.e.} \ \mbox{ for each }i\in \mathbb{N},
 \end{align*}
\item[(iii)] $v_1,\dots,v_n\in \mathcal{M}$ is a \textit{(module) basis} on $A$ if they are independent on $A$ and generate $\mathcal{M}|_{A}$. If $A$ admits a basis of finite cardinality $n\in \mathbb{N}$, we say that $A$ has \textit{local dimension} $n$. 
If $A$ admits no basis of finite cardinality, we say $A$ has infinite local dimension.
\end{itemize}
\end{definition}
\begin{remark}\label{rem:generator}
It is easy to see that in (ii) one only needs to require that $\mathcal{M}|_{A}$ is the $L^2$-closure of finite $L^{\infty}$-linear combinations of elements in $S$ where an $L^{\infty}$-linear combination is defined by
\begin{align*}
\sum_{l=1}^m f_l v_l \ \mbox{ where }f_l\in L^{\infty}(m)\mbox{ and }v_l\in S.
\end{align*}
The more general statement (ii) - that also appears in \cite{giglinonsmooth} - is to deal with $L^{\infty}$-modules that are not necessarily $L^p(m)$-normed. 
\end{remark}
\begin{proposition}[Proposition 1.4.4. \cite{giglinonsmooth}]
Local dimension is well-defined: If both $v_1,\dots,v_n$ and $w_1,\dots,w_n$ are bases of an $L^2(\m)$-normed $L^{\infty}(\m)$-module $\mathcal{M}$ on $A$ for $n,m\in\mathbb{N}$, then $m=n$.
\end{proposition}

\begin{proposition}[Proposition 1.4.5. \cite{giglinonsmooth}]\label{prop:dim}
There is a unique partition $\left\{E_k\right\}_{k\in\mathbb{N}\cup\left\{\infty\right\}}$ of $X$ such that for any $k\in \mathbb{N}$ with $\m(E_k)>0$ $E_k$ has local dimension $k$, and any $E\subset E_{\infty}$ with $\m(E)>0$ has infinite local dimension.
\end{proposition}

\begin{proposition}[{Proof of Theorem 1.4.11. \cite{giglinonsmooth}}]\label{prop:unitorth}
Let $\mathcal{M}$ be an Hilbert module. 
Then, for every $n\in \mathbb{N}$ and any Borel set $B\subset X$ that has local dimension $n$ and finite measure, there exists a \textit{unit orthogonal basis} $e_1,\dots,e_n\in \mathcal{M}$
on $B$. That is $\langle e_i,e_j\rangle=\delta_{i,j}$ $\m$-almost everywhere.
\end{proposition}
\begin{remark}\label{rem:pyth}
If $(e_i)_{i=1,\dots,n}$ is a unit orthogonal basis  on $B$, and $v1_{A}=\sum_{i=1}^n f_ie_i\in \mathcal{M}$ for $f_i\in L^{\infty}$ and $A$ a Borel subset in $B$, then 
\begin{align*}
|v1_A|^2= \sum_{i=1}^n |f_i|^2\ \mbox{ and }\ \left\|v1_A\right\|_{\mathcal{M}}^2=\sum_{i=1}^n\left\|f_i\right\|^2_{L^2(\m)}.
\end{align*}
\end{remark}

%\begin{assumption}
%For the rest of the paper it will be sufficient to assume that $\m(E_k)=\emptyset$ for $k\neq n$ and $n\in\mathbb{N}$, and in the following refer to $n$ as the geometric dimension of $X$. Moreover, 
%it will be sufficient to assume $\m(E_n)<\infty$. Therefore, if $(v_i)_{i=1,\dots,n}$ is a basis on $E_n$ by the Gram-Schmidt procedure, we can find a unit orthogonal basis $(e_i)_{i=1,\dots,n}$ on $E_n$.
%\end{assumption}
%\begin{definition}
%Given an $L^2(\m)$-normed $L^{\infty}$-module $\mathcal{M}$
%the local dimension of a measurable set $A\subset X$ is defined as the cardinality $n\in\mathbb{N}\cup \left\{\infty\right\}$ of a basis of the Banachspace $\mathcal{M}|_A$.
%\end{definition}
\begin{theorem}[\cite{giglinonsmooth}]
Let $(X,d,\m)$ be metric measure space. 
There exists a unique (up to module isomorphisms) couple $(\mathcal{M},d)$ where $\mathcal{M}$ is an $L^2(\m)$-normed $L^{\infty}(\m)$-module and $d$ is a linear map $W^{1,2}(\m)\rightarrow \mathcal{M}$ such that
\begin{itemize}
\medskip
 \item[(i)] $|df|=|\nabla f|$ holds $\m$-a.e. on $X$, and \medskip
 \item[(ii)] $\left\{ df\in\mathcal{M}: f\in W^{1,2}(\m)\right\}$ generates $\mathcal{M}$ on $X$. \medskip
\end{itemize}
If two couples $(\mathcal{M},d)$ and $(\mathcal{M}',d')$ satisfy the properties above then there exists a unique module isomorphism $\Phi:\mathcal{M}\rightarrow \mathcal{M}'$ such that $\Phi\circ d=d'$.

The unique module above is called the cotangent module of $(X,d,\m)$, and it is denoted with $L^2(T^*X)$. Its dual module is called the tangent module and denoted with $L^2(TX)$. Elements in $L^2(T^*X)$ are called
$1$-forms, and elements in $L^2(TX)$ are called vector fields. The map $d$ is called differential.

If $(X,d,\m)$ is infinitesimally Hilbertian, then $L^2(T^*X)$ is a Hilbert module, and we have $\Phi: L^2(T^*X)\equiv L^2(TX)$ for the Hilbert module isomorphism $\Phi(X)=\langle X,\cdot\rangle: L^2(TX)\rightarrow L^1(\m)$. $\Phi^{-1} \circ d =\nabla$ is called gradient. 
\end{theorem}

\subsection{Bakry-Emery condition}
\noindent 
The following was introduced in \cite{agsbakryemery}.
Let $(X,d,\m)$ be a metric measure space that is infinitesimally Hilbertian but does not necessarily satisfy a curvature-dimension condition. For $f\in D_{W^{1,2}(X)}(\Delta)$ and $\phi\in D_{L^{\infty}}(\Delta)\cap L^{\infty}(\m)$ we define the \textit{carr\'e du champ operator} as
\begin{align*}
\Gamma_2(f;\phi)=\int \frac{1}{2}|\nabla f|^2\Delta \phi d\m - \int\langle\nabla f,\nabla \Delta f\rangle \phi d\m.
\end{align*}
%In addition, we define the modified carr\'e du champ operator for $f\in D_{L^2}(\Delta)$ and $\phi\in D_{L^{\infty}}(\Delta)\cap L^{\infty}(\m)\cap \Lip(X)$ as
%\begin{align*}
%\Gamma_2'(f;\phi)=\int \frac{1}{2}|\nabla f|^2\Delta\phi d\m + \int (\Delta f)^2 \phi d\m + \int \langle\nabla f,\nabla \phi\rangle \Delta f d\m.
%\end{align*}
%In the intersection of the domains of $\Gamma_2$ and $\Gamma_2'$ integration by parts shows that $\Gamma_2(f;\phi)$ and $\Gamma_2'(f;\phi)$ coincide.
We say that $(X,d,\m)$ satisfies the \textit{Bakry-Emery condition} $BE(\ke,N)$ for $\ke\in \mathbb{R}$ and $N\in (0,\infty]$ if it satisfies the weak Bochner inequality
\begin{align*}
\Gamma_2(f;\phi)\geq \frac{1}{N}\int (\Delta f)^2 \phi d\m + \ke\int |\nabla f|^2 \phi d\m.
\end{align*}
for any $f\in D_{W^{1,2}(X)}(\Delta)$ and $\phi\in D_{L^{\infty}}(\Delta)\cap L^{\infty}(\m), \, \phi\ge 0$.

We say a metric measure space satisfies the \textit{Sobolev-to-Lipschitz} property \cite{giglistructure} if
\begin{align*}
 \left\{f\in W^{1,2}(X):|\nabla f|\in L^{\infty}(\m)\right\}=\Lip(X)
\end{align*}
More precisely, for any Sobolev function in with bounded minimal weak upper gradient there exist a Lipschitz function $\bar{f}$ that coincides $\m$-almost everywhere with $f$ such $\Lip(\bar{f})\geq |\nabla f|$.
\begin{theorem}[\cite{erbarkuwadasturm, agsbakryemery}]\label{th:be}
Let $(X,d,\m)$ be a metric measure space. Then, the condition $RCD(\ke,N)$ for $\ke\in \mathbb{R}$ and $N>1$ holds if and only if $(X,d,\m)$ is infinitesimally Hilbertian, it satisfies the 
Sobolev-to-Lipschitz property and it satisfies the Bakry-Emery condition $BE(\ke,N)$.
\end{theorem}
\begin{remark}
The case $N=\infty$ was proved in \cite{agsbakryemery}, the case $N<\infty$ in \cite{erbarkuwadasturm}.
Shortly after \cite{erbarkuwadasturm} an alternative proof for the finite dimensional case - following a completely different strategy - was established in \cite{amsnonlinear}.
\end{remark}
\subsection{Rectifiability}
Following \cite{giglipasqualetto} we say a family $\left\{A_i\right\}_{i\in\mathbb{N}}$ is an $\m$-partition of $E\subset X$ if it is a partition of some Borel set $F\subset X$ such that $\m(E\backslash F)=0$.
\begin{definition}
A metric measure space $(X,d,\m)$ is {\textit strongly $\m$-rectifiable} if there exists a $\m$-partition $\left\{A_k\right\}_{k\in \mathbb{N}}$ of $X$ into measurable sets $A_k$ such that for each $k\in \mathbb{N}$ and every $\epsilon>0$
there exists an $\m$-partition $\left\{U_i\right\}_{i\in\mathbb{N}}$ of $A_k$ and measurable maps $\phi_i:U_i\rightarrow \mathbb{R}^k$ such that for every $i\in\mathbb{N}$
\begin{align*}
\phi_i:U_i\rightarrow \phi(U_i) \ \mbox{ is } \ (1+\epsilon)\mbox{-biLipschitz }\ \& \ \ (\phi_i)_{\star}(\m|_{U_i})\ll\mathcal{L}^k.
\end{align*}
The partition $\left\{A_k\right\}_{k\in \mathbb{N}}$ that is unique up to a $\m$-negligible set is called dimensional partition of $X$.
\end{definition}
\begin{theorem}[\cite{mondinonaber, kellmondino, giglipasqualetto}]
Let $(X,d,\m)$ be a metric measure space that satisfies the curvature-dimension condition $RCD(\ke,N)$ for $N\in (0,\infty)$. Then $(X,d,\m)$ is strongly $\m$-rectifiable, and $\m(A_k)=0$ for $k>N$. 
\end{theorem}
%\begin{definition}
%Let $(X,d,\m)$ be a $\m$-rectifiable space. A chart is a couple $(U,\phi)$ where $U$ is a Borel set in $A_k$ for some $k\in\mathbb{N}$, and $\phi:U\rightarrow \phi(U)\subset \mathbb{R}^k$ is a biLipschitz map such that 
%\begin{align*}
%C^{-1}\mathcal{L}^k|_{\phi(U)}\leq \phi_*(\m|_{U})\leq C \mathcal{L}^k|_{\phi(U)}\ \ \mbox{ for some constant }C>1.
%\end{align*}
%An atlas for $(X,d,\m)$ is a family $\left\{U_i^k,\phi_i^k\right\}_{i,k\in\mathbb{N}}$ of charts such that $\left\{U_i\right\}_{i\in\mathbb{N}}$ is an $\m$-partition of $A_k$.
%A chart $(U,\phi)$ is said to be an $\epsilon$-chart provided $\phi$ is a $(1+\epsilon)$-biLipschitz map, and an atlas is said to be an $\epsilon$-atlas provided any charts is an $\epsilon$-chart.
%Moreover, we always assume that $U$ has finite $\m$-measure.
%\end{definition}
%\begin{remark}
%Any $\m$-rectifiable metric measure space admits an atlas, and any strongly $\m$-rectifiable metric measure space admits an $\epsilon$-atlas for any $\epsilon>0$.
%\end{remark}

\begin{theorem}[\cite{giglipasqualetto}]
Let $(X,d,\m)$ be a metric measure space that satisfies the condition $RCD(\ke,N)$ for $N\in (0,\infty)$, and let $A_k$ its dimensional decomposition.
Then, the local dimension of $A_k$ is $k\in\mathbb{N}\cup\left\{\infty\right\}$. Hence $A_k=E_k$ for any $k\in \mathbb{N}$ where $E_k$ is as in Proposition \ref{prop:dim}.
%\\
%If $\phi:U\rightarrow \phi(U)\subset \mathbb{R}^k$ is a chart for $A_k$ with coordinate functions $\phi^1,\dots,\phi^k:U\rightarrow \mathbb{R}$, the co-vectorfields $(d\phi^i)_{i=1,\dots,k}$ form a basis of $L^2(T^*U)$.
\end{theorem}
\begin{definition}
If $(X,d,\m)$ satisfies the condition $RCD(\ke,N)$, we say $x_0\in A_k$ is a regular point if the Gromov-Hausdorff tangent cone at $x_0$ is $\mathbb{R}^k$ where $\left\{A_k\right\}_{k\in\mathbb{N}}$ is the dimensional decomposition of $X$. 
We denote the set of all regular points $X_{reg}$.
\end{definition}
\subsection{Hessian operator}
\noindent 
Recall from \cite{giglinonsmooth} that $L^2((T^*X)^{\otimes 2})$ is the subset of elements $A$ in the $L^0$-module-tensor product $L^0((T^*X)^{\otimes 2})$ such that 
\begin{align*}
\left\||A|_{HS}\right\|_{L^2(\m)}^2=\int |A|_{HS}^2d\m <\infty.
\end{align*}
$|A|_{HS}$ is the \textit{pointwise Hilbert-Schmidt norm} whose construction can be found in \cite{giglinonsmooth}. Similar, one constructs $L^2((TX)^{\otimes 2})$, and if $L^2(T^*X)$ and $L^2(TX)$ are Hilbert modules,
$L^2(T^*X^{\otimes 2})$ and $L^2(TX^{\otimes 2})$ are isomorphic Hilbert modules as well. $L^2(T^{*}X^{\otimes 2})$ can be seen
as the space of all continuous bilinear
forms $A:L^2(TX)^2\rightarrow L^0(\m)$ such that $\left\||A|_{\scriptscriptstyle{HS}}\right\|_{\scriptscriptstyle{L^2(\m)}}<\infty$. $L^0$-continuity corresponds for $\m$ finite and normalized to \textit{convergence in probability}. Recall also that
for $A\in L^2(T^*X^{\otimes 2})$ 
\begin{align*}
A(X,Y)\leq |A|_{HS}|X||Y| \ \m\mbox{-a.e.}
\end{align*}
when $X,Y\in L^2(TX)$. In particular, $A(X,Y)\notin L^1(\m)$ in general.
%We define the trace of $A\in L^2((T^*X)^{\otimes 2})$ as element in $L^2(\m)$ via
%\begin{align*}
%\tr A|_B=\sum_{i=1}^k A(e_i,e_i)
%\end{align*}
%where $B$ is a subset of $E_k$ with finite measure and $\left\{e_i\right\}_{i=1,\dots,k}$ a unit orthogonal basis on $B$. Since $\int (\tr A|_B)^2d\m \leq k\int_B |A|_{HS}d\m$, $\tr A$ is indeed an element in $L^2(\m)$.
%\\
%\\
\begin{definition}
The space of test functions is 
\begin{align*}
\mathbb{D}_{\infty}= D_{W^{1,2}(X)}(\Delta)\cap L^{\infty}(\m)\cap \left\{f\in W^{1,2}(X):|\nabla f|\in L^{\infty}(\m)\right\}.
\end{align*}
\end{definition}
If $(X,d,\m)$ satisfies a Riemannian curvature-dimension condition, then
$P_tL^{\infty}(\m)$ is a subset of $\mathbb{D}_{\infty}$, it is dense in $W^{1,2}(X)$ and in $D_{L^2(m)}(\Delta)$ w.r.t. the Sobolev norm and the graph norm of the operator $\Delta$ respectively, and by the 
Sobolev-to-Lipschitz property we have $\mathbb{D}_{\infty}\subset \Lip_b(X)$. Moreover, the co-tangent module $L^2(T^*X)$ is generated by
$$T\mathbb{D}_{\infty}=\left\{\sum_{i=1}^n g_i df_i:n\in \mathbb{N}, \ g_i, f_i\in \mathbb{D}_{\infty}, i= 1,\dots,n \right\}.$$

\begin{definition}[\cite{giglinonsmooth}]
The space $W^{2,2}(X)\subset W^{1,2}(X)$ is the set of all functions $f\in W^{1,2}(X)$ for which there exists $A\in L^2(T^*X\otimes T^*X)$ such that for all $f,g_1,g_2\in \mathbb{D}_{\infty}$ we have
\begin{align}\label{hessian}
&2\int hA(dg_1,dg_2)d\m\nonumber\\
&\hspace{0.4cm}=-\int h \langle \nabla f, \nabla \langle \nabla g_1,\nabla g_2\rangle\rangle dm -\int \langle \nabla f,\nabla g_1\rangle \Div(h\nabla g_{2})d\m-\int \langle \nabla f,\nabla g_2\rangle \Div(h\nabla g_{1})d\m 
\end{align}
where $\Div(h\nabla g)=\langle \nabla h,\nabla g\rangle + h\Delta g$. In this case the operator $A$ will be called the \textit{Hessian} of $f$ and will be denoted with $\Hess f$. $W^{2,2}(X)$ is equipped with the norm
\begin{align*}
\left\|f\right\|_{W^{2,2}(X)}^2=\left\|f\right\|_{L^2}^2+\left\||\Hess f|_{HS}\right\|_{L^2}^2.
\end{align*}
We say that $\Hess f\geq \kappa$ on a measurable subset $B$ for $\kappa\in \mathbb{R}$ if for any $u\in \mathbb{D}_{\infty}$ 
\begin{align*}
\Hess f(1_{B}\nabla u,\nabla u)\geq \kappa |1_B\nabla u|^2=\kappa1_B|\nabla u|^2 \ \m\mbox{-almost everywhere.}
\end{align*}
%Since $\mathbb{D}_{\infty}$ is dense in $W^{1,2}(X)$, and since $\Hess f$ is a continuous bilinear form on $W^{1,2}(X)$, it is enough to check the previous inequality for $u\in \mathbb{D}_{\infty}$.
\end{definition}
\begin{remark}
We note that the density property of $\mathbb{D}_{\infty}$ and the fact that $\mathbb{D}_{\infty}$ is an algebra ensure that $\Hess f\in L^2(T^*X\otimes T^*X)$ is uniquely determined by (\ref{hessian}).
Then it is clear that $\Hess f$ depends linearly on $f$, and $W^{2,2}(X)$ therefore becomes a vector space.
\end{remark}
\begin{lemma}
Consider $f\in W^{2,2}(X)$, and assume $\Hess f\geq \kappa$ $m$-a.e. on $B$. Then $\Hess f(1_BV,1_BV)\geq \kappa |V|^21_B$ $\m$-a.e. for every $V\in L^2(TX)$.
\end{lemma}
\begin{proof}
Let $V\in L^2(TX)$. It is enough to consider the case $B=X$. Since $T\mathbb{D}_{\infty}$ generates $L^2(TX)$, Remark \ref{rem:generator} we find functions 
$f_{k,i}, g_{k,i}\in \mathbb{D}_{\infty}$ with $i=1,\dots,n_{k}\in \mathbb{N}$ and $k\in \mathbb{N}$ such that 
\begin{align*}
V_k=\sum_{i=1}^{n_{k}}f_{k,i}\nabla g_{k,i} \rightarrow V\ \mbox{ in }L^2(TX).
\end{align*}
Then, since every $f_{k,i}$ can be approximated in $L^2$-sense by measurable functions that take only finitely many values in $\mathbb{R}$, and since $\nabla$ is linear, we see that $X$ is approximated 
in $L^2$-sense by vector fields $W_k$ of the form
\begin{align*}
W_k=\sum_{i=1}^{n_k}1_{B_{k,i}}\nabla h_{i,k}
\end{align*}
for measurable decompositions $\left\{B_{k,i}\right\}_{i\in\mathbb{N}}$ of $X$, and $h_{k,i}\in \mathbb{D}_{\infty}$. Hence, we have
\begin{align*}
\Hess f (W_k,W_k)&=\sum^{n_k}_{i,j=1}\Hess f(1_{B_{k,i}}\nabla h_{i,k}, 1_{B_{k,j}}\nabla h_{j,k})\\
&=\sum_{i=1}^{n_k}1_{B_{k,i}}\Hess f(\nabla h_{k,i},\nabla h_{k,i})\geq \sum_{i=1}^{n_k}\kappa |\nabla h_{k,i}|^2=\kappa|W_k|^2 \ \ \m\mbox{-a.e. on }B.
\end{align*}
The second equality is the $L^{\infty}$-homogeneity of $\Hess f(\cdot,\cdot)$.
Hence, by $L^2$-convergence of $W_k$ in $L^2(TX)$ the right hand side converges $\m$-a.e. to $|V|$ after taking a subsequence. Moreover, by continuity of the bilinear for $\Hess f: L^2(TX)^2\rightarrow L^0(\m)$ 
the left hand side converges $\m$-a.e. to $\Hess f(V,V)$ after taking
another subsequence. Then, the claim follows.
\end{proof}
\begin{theorem}[\cite{giglinonsmooth, savareself}]\label{th:hesscont}
Let $f\in D_{L^2}(\Delta)$. Then $f\in W^{2,2}(X)$, and 
\begin{align*}
\int |\Hess f|_{HS}^2 d\m\leq \int \left[\left(\Delta f\right)^2 - \ke|\nabla f|^2\right] d\m.
\end{align*}
In particular, $\mathbb{D}_{\infty}\subset W^{2,2}(X)$.
\end{theorem}
\begin{proposition}[\cite{giglinonsmooth}]\label{prop:chainrule}
Let $f\in W^{2,2}(X)\cap \Lip(X)$ and let $\phi\in C^2(\mathbb{R})$ with bounded first and second derivative. Then $\phi\circ f\in W^{2,2}(X)$ and the following formula holds
\begin{align*}
\Hess^X (\phi\circ f)(\nabla u,\nabla u) = \phi'\circ f \Hess^X f(\nabla u,\nabla u) + \phi''\circ f \langle \nabla f,\nabla u\rangle \ \ \forall u\in W^{1,2}(X).
\end{align*}

\end{proposition}

\begin{definition}
$H^{2,2}(X)$ is defined as the closure of $\mathbb{D}_{\infty}$ in $W^{2,2}(X)$.
\end{definition}
\begin{remark}\label{rem:H22}
$H^{2,2}(X)$ actually coincides with the $W^{2,2}$-closure of $D(\Delta)$ (Proposition 3.3.18 in \cite{giglinonsmooth}). In particular, 
any $f\in D(\Delta)$ is in $H^{2,2}(X)$ and admits a Hessian. 
\end{remark}
\begin{proposition}[\cite{giglinonsmooth}, Proposition 3.3.22]\label{prop:nuetzlich}
Let $f\in W^{2,2}(X)\cap\Lip(X)$ and $g_1,g_2\in H^{2,2}(X)\cap \Lip(X)$. Then, $\langle \nabla f,\nabla g_i\rangle\in W^{1,2}(X), i=1,2$ and 
\begin{align*}
2\Hess f(\nabla g_1,\nabla g_2)= \langle \nabla g_1,\nabla\langle\nabla f,\nabla g_2\rangle\rangle + \langle \nabla g_2,\nabla \langle \nabla f,\nabla g_1\rangle \rangle - \langle \nabla f,\nabla \langle\nabla g_1,\nabla g_2\rangle\rangle.
\end{align*}
\end{proposition}
\begin{remark}
In the case when $f\in D(\Delta)$ is Lipschitz $\Hess f(\nabla g_1,\nabla g_2)$ for $g_1,g_2\in H^{2.2}(X)$ can be computed explicitly by the previous Proposition since $D(\Delta)\subset H^{2,2}(X)$ by the previous remark.
\end{remark}
%
%\begin{theorem}[\cite{giglinonsmooth}]
%Let $f,g\in H^{2,2}(X)$. Then $\Hess f=\Hess g$ $\m$-a.e. on $\left\{f=g\right\}$.
%\end{theorem}
\begin{theorem}[\cite{giglistructure}, \cite{gigtam}]\label{th:secondvariation}
Let $(X,d,\m)$ be a metric measure space that satisfies the condition $RCD(\ke,N)$ for $N<\infty$. Let $\mu_0,\mu_1\in \mathcal{P}^1(X)$ such $\mu_i=\rho_i\m\leq C\m$ for $C>0$ and $i=0,1$, and let $(\mu_t)_{t\in [0,1]}$ be 
the unique $L^2$-Wasserstein geodesic. 
\begin{itemize}
\item[(i)] {\textit First variation formula.} Let $f\in W^{1,2}(X)$. Then, the map $t\in [0,1]\rightarrow \int f d\mu_t$ belongs to $C^1([0,1])$ and for every $t\in [0,1]$ it holds
\begin{align*}
\frac{d}{dt}\int f d\mu_t=\int \langle \nabla f,\nabla \phi\rangle d\mu_t.
\end{align*}
\item[(ii)] {\textit Second variation formula.}
Moreover, let $f\in H^{2,2}(X)$. Then, the map $t\in [0,1]\rightarrow \int fd\mu_t$ belongs to $C^2([0,1])$ and for every $t\in [0,1]$ it holds
\begin{align*}
\frac{d^2}{dt^2}\int fd\mu_t=\int \Hess f(\nabla\phi_t,\nabla \phi_t)d\mu_t
\end{align*}
where $\phi_t$ is the function such that for some $t\neq s\in [0,1]$ the function $-(s-t)\phi_t$ is a Kantorovich potential between $\mu_t$ and $\mu_s$.
\end{itemize}
\end{theorem}

%
%\noindent 
%As part of the proof 
%of the previous result one also gets that for every $f\in \mathbb{D}_{\infty}$ we have $|\nabla f|^2\in W^{1,2}(X)$. Then one easily checks by integration by parts that for $f\in \mathbb{D}_{\infty}$
%\begin{align*}
%2\Hess f(g_1,g_2)=-\langle \nabla f,\nabla \langle\nabla g_1,\nabla g_2\rangle\rangle + \langle \nabla g_1,\nabla \langle \nabla f,\nabla g_2\rangle\rangle +\langle \nabla g_2,\nabla \langle \nabla f,\nabla g_1\rangle\rangle.
%\end{align*}
%Moreover, Savar\'e and Gigli derive the following. 
\begin{theorem}[\cite{savareself, giglinonsmooth}]If $(X,d,\m)$ satisfies the condition $RCD(\ke,\infty)$, and $f\in \mathbb{D}_{\infty}\subset W^{2,2}(X)$, then $|\nabla f|^2\in W^{1,2}(X)\cap D({\bf \Delta})$ and
an improved Bochner formula holds in the sense of measures involving the Hilbert-Schmidt norm of the Hessian of $f$:
\begin{align*}
{\bf \Gamma}_2(f):=\frac{1}{2}{\bf \Delta}|\nabla f|^2- \langle\nabla f,\nabla \Delta f\rangle \m \geq \left[\ke |\nabla f|^2 + |\Hess f|_{HS}^2 \right]\m
\end{align*}
where ${\bf\Delta}$ is the measure valued Laplace operator, and ${\bf \Gamma}_2$ is called {\it measure valued $\Gamma_2$-operator}. In particular, the singular part of the left hand side is non-negative.
\end{theorem}
In the context of $RCD(\ke,N)$-spaces with finite $N$ the previous theorem was improved by Han \cite{hanriccitensor}, and in particular, he obtains the following. 
\begin{theorem}[\cite{hanriccitensor}]\label{th:tracelaplace}
Let $(X,d,\m)$ be a metric measure space that satisfies the condition $RCD(\ke,N)$ with $N<\infty$, and let $\left\{A_k\right\}_{k\in\mathbb{N}\cup\left\{\infty\right\}}$ be its dimensional decomposition. If 
$A_N\neq \emptyset$ and therefore $N\in\mathbb{N}$, then for any $f\in \mathbb{D}_{\infty}$ we have that $\Delta f=\tr\Hess f$ $\m$-a.e. in $A_N$. More precisely, if $B\subset A_N$  is a set of finite measure and
$(e_i)_{i=1,\dots,N}$ is a unit orthogonal basis on $B$, then $$\Delta f|_B=\sum_{i=1}^N \Hess f(e_i1_B,e_i1_B)=\sum_{i=1}^N\Hess f(e_i,e_i)1_{B}=:[\tr \Hess f]|_B.$$
\end{theorem}
\begin{corollary}\label{cor:trace}
Let $(X,d,\m)$ be a metric measure space that satisfies the condition $RCD(\ke,n)$ with $\m=\mathcal{H}^n$ and $n\in\mathbb{N}$. Then $A_k=\emptyset$ for $k\neq n$, and for any $f\in D_{L^2}(\Delta)$ we have that
$\Delta f=\tr\Hess f$ $\m$-almost everywhere in the sense of the previous theorem.
\end{corollary}
\begin{proof} The first claim is clear from the assumptions.
Let $f\in D_{L^2}(\Delta)$, and consider a sequence $\phi_j\in \mathbb{D}_{\infty}$ that approximates $f$ in $D_{L^2}(\Delta)$. Moreover, let $B\subset A_n$ be a Borel set of finite measure and let $(e_i)_{i=1,\dots,n}$ be a unit orthogonal basis and 
$\tr$ be the corresponding trace.
It follows
\begin{align*}
&\int 1_B (\Delta f- \tr \Hess f|_B)^2 d\m = \int 1_B (\Delta f - \Delta\phi_j + \tr\Hess\phi_j|_B - \tr\Hess f|_B)^2d\m\\
&\hspace{3cm}\leq 2\int 1_B\left(\Delta f - \Delta\phi_j\right)^2d\m + 2\int 1_B\left(\tr[\Hess (\phi_j-f)]|_B\right)^2d\m\\
&\hspace{3cm}\leq 2\left\|\Delta f-\Delta\phi_j\right\|^2_{L^2(\m)} + n\left\||\Hess (f-\phi_j)|_{HS}\right\|^2 \leq \epsilon
\end{align*}
where the last inequality holds for arbitrary $\epsilon>0$ provided $j$ is sufficiently large. We obtain that $\Delta f|_{B}=\tr \Hess f|_B$ $\m$-a.e. and therefore the claim.
\end{proof}

\subsection{Upper and lower sectional curvature bounds for metric spaces}\label{CAT and CBB}
We recall the following notions of spaces with curvature bounded below (above) for geodesic metric spaces

\begin{definition}
We say that a complete geodesic metric space $(X,d)$ is $CBB(\kappa)$ or  has curvature bounded  below by $\kappa\in\mathbb{R}$ (respectively is $\CAT(k)$)  if for any triple of points $x,y,z\in X$ with $d(x,y)+d(y,z)+d(z,x)<2\pi_\kappa$
the following condition holds.

For any geodesic $[xz]$ and $q\in ]x z[$, we have

\begin{equation} \label{point-on-a-side-comp}
d(y,q)\ge d(\bar y, \bar q)\, (\text{ respectively }d(y,q)\le d(\bar y, \bar q))
\end{equation}

where  $\triangle(\bar{x},\bar{y},\bar{z})_{\S2k}$ is a comparison triangle in $\S2k$ with $d(\bar x,\bar y)=d(x,y), d(\bar x,\bar z)=d(x,z), d(\bar z,\bar y)=d(z,y)$ and $\bar q\in  ]\bar x \bar z[$ satisfies $d(\bar q, \bar x)=d(q,x)$.

If $\uk>0$ and a $X$ is a 1-dimensional manifold with possibly nonempty boundary for $X$ to be $CBB(\uk)$ we additionally require that $\diam X\le \pi_\kappa$.

Property \eqref{point-on-a-side-comp} is equivalent to saying that for any  unit speed geodesic $\gamma : [0,l]\to X$ 
% $\log\left[-\cos_{\kappa} d_x\right]\circ\gamma$ is $(\kappa,1)$-convex in the sense of Definition \ref{def:metricconvex} 
such that 
\begin{equation}\label{unique-comp-tr}
d(y,\gamma(0))+l+d(\gamma(l),y)<2\pi_k,
\end{equation}
it holds that 

\begin{align}\label{CBB-CBA-mdk}
\left[\md_{\kappa}(d_y\circ\gamma)\right]''+\md_{\kappa}(d_y\circ\gamma)\leq 1 (\text{ respectively } \ge 1)
\end{align}

A reformulation of this inequality is - taking into account the definition of $\md_{\kappa}$ via $\cos_{\kappa}$ - 
\begin{align*}
\left[\frac{1}{\kappa}\cos_{\kappa}(d_y\circ\gamma)\right]''\geq & - \cos_{\kappa}(d_y\circ\gamma)\ \  (\text{respectively } \le \cos_{\kappa}(d_y\circ\gamma) )  \ \mbox{ if }\ \kappa\neq 0 \ \ \&\ \\
\left[\frac{1}{2}d_y^2\circ\gamma\right]'' \leq& 1\ \  (\ge 1)\ \ \mbox{ if }\ \kappa=0.
\end{align*}

\end{definition}

In particular, $(X,d)$ has $\curv\ge 0$ ( is $\CAT(0)$)   if and only if for any $y\in X$ the function $\frac{1}{2}d_y^2$ is $1$-concave ($1$-convex).

We will refer to CBB version of inequality  \eqref{point-on-a-side-comp} as  ${\eqref{point-on-a-side-comp}}_{(CBB)}$ and to the CAT version of it as   ${\eqref{point-on-a-side-comp}}_{(CAT)}$. We will employ the same convention for ~\eqref{CBB-CBA-mdk}.
\begin{comment}
\begin{remark}\label{mdk-on-X-miuns y}
Since ${\eqref{point-on-a-side-comp}}_{(CBB)}$ and ${\eqref{CBB-CBA-mdk}}_{(CBB)}$ always hold for degenerate triangles for trivial reasons, $X$ satisfies $CBB(\kappa)$ iff property   ${\eqref{CBB-CBA-mdk}}_{(CBB)}$ holds for all $y,\gamma$ satisfying \eqref{unique-comp-tr} such that $\gamma$ does not pass through $y$.
\end{remark}
 
 \end{comment}
\begin{comment}
Moreover -- by straightforward computations -- $(X,d)$ has curvature bounded from below by $\kappa$ for $\kappa> 0$ if and only if the function $v:=-\log(\frac{1}{\kappa}\cos_{\kappa} d_y):X\rightarrow \mathbb{R}$ is $(\kappa,1)$-concave in the sense of Definition \ref{def:metricconvex}. 
More precisely, for any $1$-speed geodesic $\gamma$ the function $v:=-\log(\frac{1}{\kappa}\cos_{\kappa}d_y)\circ \gamma$ satisfies 
$
v''\leq \kappa+(v')^2.
$
\end{comment}

\begin{remark}\label{cont-geod}
It's immediate from the definition that in a $CAT(\kappa)$ space $X$ geodesics of length $<\pi_\kappa$ are unique and depend continuously on their endpoints. Also, local geodesic of length $<\pi_\kappa$ are \emph{distance minimizing}, i.e. are geodesics.
\end{remark}
 \begin{definition}
 We say that a complete geodesic space $(X,d)$ is $CBA(\kappa)$ (has curvature bounded above by $\kappa$ ) if for every point $p\in X$ there is $r_p>0$ such that  ${\eqref{point-on-a-side-comp}}_{(CAT)}$  holds for any $x,y,z\in B_{r_p}(p)$.
 
 \end{definition}
 \begin{remark}\label{rem:convex}
 In the above definition of $CBA(\kappa)$ we do not require that the geodesics $[xy], [xz], [yz]$ lie in $B_{r_p}(p)$. However, it immediately follows from  ${\eqref{point-on-a-side-comp}}_{(CAT)}$  that $B_r(p)$ is convex for any $r<\min(r_p,\pi_k/2)$.  If $X$ is $\CAT(\kappa)$ this gives that $B_r(p)$ is convex and $\bar B_r(p)$ is $\CAT(\kappa)$ for any $r<\pi_\kappa/2$.

  \end{remark}
%or $\md_{\kappa}\circ\gamma:[0,|\dot{\gamma}|]\rightarrow \mathbb{R}$ satisfies 

\begin{example}
The standard sphere $\mathbb{S}^n$ and any geodesically convex subsets of $\mathbb{S}^n$ are $\CAT(1)$. On the other hand, for $n\ge 2$  any non simply connected $n$-manifold of $\sec\equiv 1$ (e.g. $\mathbb{PR}^n$) is $CBA(1)$ but not $\CAT(1)$.
\end{example}

\begin{remark}\label{rem:distconv}
It follows directly from the definition of $CAT(\kappa)$ and from the corresponding computations in $\S2k$ that if $X$ is $CAT(\kappa)$ then $d_y$,  and $\md_{\kappa}(d_y)$ are \emph{convex} in $B_{\pi_k/2}(y)$ for any $y$ in $X$. 
\end{remark}
\noindent

\begin{remark}\label{rem-globalization}
Similarly to the definition of $CBA(\kappa)$ one can define locally $CBB(\kappa)$ spaces by requiring that they satisfy   ${\eqref{point-on-a-side-comp}}_{(CBB)}$  locally. However, it turns out that this is equivalent to requiring that they satisfy ${\eqref{point-on-a-side-comp}}_{(CBB)}$  \emph{globally} by the Globalization Theorem~\cite{bgp}.
\end{remark}

If $(X,d)$ is $CAT(\uk_1)$ ($CBB(\uk_1)$)  and $\uk_1\le \uk_2$ ( $\uk_1\ge \uk_2$) then $(X,d)$ is $CAT(\uk_2)$ ($CBB(\uk_2)$).

If $(X,d)$ is $CAT(\kappa)$ ($CBB(\kappa)$) then $(X,\lambda  d)$ is $CAT(\kappa/\lambda^2)$ ($CBB(\kappa/\lambda^2)$). Therefore after appropriate rescaling any $CAT(\kappa)$ ($CBB(\kappa)$) becomes $\CAT(1)$ ($CBB(-1)$).

% By rescaling this automatically generalizes to general positive upper curvature bounds, and also to non-positive curvature. We collect several important facts about $\CAT(1)$-spaces.
\begin{theorem}[\cite{bbi} Theorem 4.7.1]
Let $(X,d)$ be a complete geodesic space. Then the euclidean cone $C(X)$ over $(X,d)$ is $\CAT(0)$  ($CBB(0)$) if and only if $(X,d)$ is $\CAT(1)$ ($CBB(1)$).
\end{theorem}
Let $X$ be $CAT(\kappa)$ and let $y\in X$.

By remark~\ref{rem:distconv}  $d_y$,  $\md_{\kappa}(d_y)$ are convex in $B_{\pi_k/2}(y)$. Together with the first variation formula  \cite[Chapter 4]{bbi} this shows that they  admit obvious  (inverse) gradient flows on $B_{\pi_k/2}(y)$ following appropriately parameterized  unique geodesics connecting points in $B_{\pi_k/2}(y)$ to $y$.
Moreover, the $CAT(\kappa)$ condition  implies that all three  of these  flows of are $1$-Lipschitz on  $B_{\pi_k/2}(y)$ for any positive time $t$ since this holds in $\S2k$.
\begin{lemma}\label{lem-contr-flow}
Let $X$ be $CAT(\kappa)$ with $\diam X<\pi_\kappa/2$. Let $y\in X$ and let  $\Phi_t$ be the gradient flow of $\md_{\kappa}(d_y)$ and $\Psi_t$ be the gradient flow of $d_y$. Then for any $n\ge 1, t\ge 0$ we have

\begin{enumerate}
\item For any Borel $A\subset X$ it holds that $\mathcal{H}^n(\Phi_t(A))\le \mathcal{H}^n(A)$.
\item For any  Borel $A\subset \Phi_t(X)$
\begin{align*}
(\Phi_t)_{\star}\mathcal{H}^n(A)\geq \mathcal{H}^n(A).
\end{align*}
\end{enumerate}
Furthermore, the same properties hold for $\Psi_t$.
\end{lemma}
\begin{definition}
Given a point $p$ in a $CAT(\kappa)$ space $X$ we say that two unit speed geodesics starting at $p$ define the same direction if the angle between them is zero. This is an equivalence relation by the triangle inequality for angles and the angle induces a metric on the set $S_p^g(X)$ of equivalence classes. The metric completion  $\Sigma_p^gX$ of $S_p^gX$ is called the \emph{space of geodesic directions} at $p$.
The Euclidean cone $C(\Sigma_p^gX)$ is called the \emph{geodesic tangent cone} at $p$ and is denoted by $T^g_pX$.
\end{definition}
The following theorem is due to Nikolaev~\cite[Theorem 3.19]{BH99}:
\begin{theorem}\label{geod-tangent-cone}
$T_p^gX$ is $CAT(0)$ and $\Sigma_p^gX$ is  $CAT(1)$. 
\end{theorem}
Note that this theorem in particular implies that $T_p^gX$ is a geodesic metric space which is not obvious from the definition.
Note further that $\Sigma_p^gX$ need not be path connected. In this case the above theorem means that each path component of $\Sigma_p^gX$ is $CAT(1)$ and the distance between points in different components is $\pi$.
\begin{comment}
We will also need the following simple lemma

\begin{lemma}\label{lem-susp}
Let $(\Sigma, d)$ be $CAT(1)$. Suppose there exist two points $y_+,y_-\in\Sigma$ such that $d(y_+,y_-)=\pi$ and for any point $x\in\Sigma$ it holds that  $d(y_+,x)+d(x,y_-)=\pi$.
Then the midpoint set $\hat \Sigma=\{x: d(x,y_+)=d(x,y_-)=\pi/2$ is convex in $\Sigma$ and $\Sigma$ is isometric to the spherical suspension over  $\hat \Sigma$ with vertices at $y_+$ and $y_-$.
\end{lemma}
\begin{proof}
%This follows by a straightforward application of the rigidity case of Toponogov comparison in $CAT(\kappa)$ spaces.
This lemma is well-known but we sketch a proof for completeness. Let $x,z$ be points in $\Sigma$ with $d(x,y_\pm)=d(z,y_\pm)=\pi/2$ and  $d(x,z)<\pi$ and let $q\in]xz[$.
Then applying ${\eqref{point-on-a-side-comp}}_{(CAT)}$  to the triangles $\Delta(y_\pm, x ,z)$ we get that $d(y_\pm,q)\le d(\bar y_\pm, \bar q)=\pi/2$. Hence  $\pi=d(y_-,q)+d(y_+,q)\le d(\bar y_-, \bar q)+d(\bar y_+, \bar q)=\pi$.

Hence, $d(y_\pm, q)=d(\bar y_\pm, \bar q)=\pi/2$ i.e. we have an equality in  ${\eqref{point-on-a-side-comp}}_{(CAT)}$  in both cases. By the rigidity case of CAT comparison ~\cite[Proposition 2.9]{BH99} this means that $y_-,x,z$ span a (unique) isometric copy of $\Delta(\bar y_-,\bar x,\bar z)$ in $\Sigma$ and the same is true for $y_+$. This easily implies the lemma.

\end{proof}
\end{comment}

\subsection{Spaces with two sided sectional Alexandrov curvature bounds}

Spaces with two sides Alexandrov bounds (i.e. spaces satisfying $CBA(\uk_1), CBB(\uk_2)$ for some $\uk_1,\uk_2\in\R$) have been studied by Alexandrov, Nikolaev and Berestovsky. The following structure theorem holds:
\begin{theorem}[\cite{nikolaev}]\label{th:nikolaev}
Let $(X,d)$ be a complete finite dimensional geodesic metric space which is $CBA(\uk_1), CBB(\uk_2)$  for some $\uk_1,\uk_2\in\R$.

Then $\kappa_2\le \kappa_1$ and $X$ is an $n$-dimensional topological manifold (possibly with boundary) for some $n\ge 1$. Moreover,  $\Int X$ possesses a canonical  $C^{3,\alpha}$-atlas for $\alpha\in (0,1)$ of harmonic coordinate charts such that in each chart $d$ is induced by a Riemannian tensor $g$ whose coefficients
$g_{i,j}$ w.r.t. this chart are in the class $W^{2,p}\cap C^{1,\alpha}$ for any $1\le p<\infty, 0<\alpha<1$.

\begin{comment}
\medskip
\begin{itemize}
\item[1.]
Then, if $X$ has no boundary in the sense of Alexandrov spaces, $X$ is a manifold and possesses a $C^{3,\alpha}$-atlas for $\alpha\in (0,1)$ of harmonic coordinate charts such that in each chart $d$ is induced by a Riemannian tensor $g$ whose coefficients
$g_{i,j}$ w.r.t. this chart are in the class $W^{2,p}$ for any $p>1$. In particular, $g_{i,j}$ is in $C^{1,\alpha}$ for any $\alpha\in (0,1)$.
\medskip
\item[2.]
If $X$ has boundary in the sense of Alexandrov spaces, then the previous applies to $\Int X:=X\backslash \partial X$ but the  boundary $\partial X$ might be non-smooth.
\medskip
\item[3.]
If $X$ has curvature bounded from above and below by the same constant $k\in\mathbb{R}$ then $X$ is a geodesically convex subset of a Riemannian space form with constant curvature $k$.
\end{itemize}
\end{comment}
\end{theorem}

\begin{remark}
If $X$ has nonempty boundary then the boundary need not be smooth. E.g. if $X$ is a closed convex body in $\R^n$ then it's $CBB(0)$ and $CBA(0)$.
\end{remark}
The following lemma is elementary and is left to the reader as an exercise.
\begin{lemma}\label{CBB-CBA-equal}
Let $(X,d)$ be an $n$-dimensional space which is $CAT(\uk)$ and $CBB(\uk)$. Then $X$ is isometric to a convex subset of $\SS^n_\kappa$.
\end{lemma}
\section{Lower bounds for the measure-valued Laplace operator}\label{sec:laplace bounds}

The following lemma is well-known (see e.g. \cite{BH99}) but we include the proof for completeness.
\begin{lemma}\label{extend-geod}
Let $X$ be $CAT(\kappa)$ and let $p\in X$. Suppose $B_r(p)$ is a topological $n$-manifold for some $r<\pi_\kappa/2$. Then every geodesic $[xy]\subset B_r(p)$ can be extended to a geodesic with end points on $S_r(p)$.
\end{lemma}
\begin{proof}
By completeness  of $X$  it's enough to show that geodesics can not terminate at points in $B_r(p)$. Suppose to the contrary that a geodesic $[xy]\subset B_r(p)$ can not be extended past $y$. By possibly changing $x$ we can assume that $\bar B_{2l}(x)\subset U\subset B_r(p)$ where $l=d(x,y)$ and $U$ is homeomorphic to $\mathbb R^n$.  Since $H_{n-1}(U\backslash \{y\})\cong \mathbb Z\ne 0$,  the inclusion $i: \bar B_{2l}(y)\backslash \{y\}\to U\backslash \{y\}$ is not homotopic to a point. On the other hand, since $[xy]$ can not be extended past $y$, the "straight line" homotopy (which is continuous by remark~\ref{cont-geod})  along geodesics emanating from $x$ gives a homotopy of $i$ and the constant map $\bar B_{2l}(y)\backslash \{y\}\to\{x\}$. This is a contradiction and hence all geodesics in $B_r(p)$ can be extended till they hit the sphere $S_r(p)$.
\end{proof}

\begin{theorem}\label{th:lowerbound} Let $(X,d)$ be a metric space that is $\CAT(\uk)$ and $\diam_X< \pi_{\uk}/2$. 
Assume $(X,d,\mathcal{H}^n)$ is a metric measure space satisfying the condition $RCD(\ke,n)$ for $n\in\mathbb{N}$.
Let $x_0\in X$ be a point such that there is an open neighbourhood $U$ of $x_0$ that is homeomorphic to an $n$-manifold.
Then, there exists $\epsilon>0$ such that for any $y\in X$ we have 
%${\bf\Delta}\frac{1}{2}d^2_y$ satisfies
\begin{align*}
 [{\bf\Delta} \md_{\uk} (d_y)]|_{B_{\epsilon}(x_0)}\geq 0\ \ \ \& \ \ \ [{\bf \Delta} d_y]|_{B_{\epsilon}(x_0)}\geq 0
 % -\cot_{\uk}(d_y)|_{B_{\epsilon}(x_0)\backslash\left\{y\right\}}.
\end{align*}
\end{theorem}
\begin{remark}
We note that measure valued Laplacians of  $\md_{\uk} (d_y)|_{\overline{B}_{\epsilon}(x_0)}$ and of $d_y|_{\overline{B}_{\epsilon}(x_0)}$ on $\overline{B}_{\epsilon}(x_0)$ will have  negative singular parts on the boundary sphere  $S_{\epsilon}(x_0)$.
%, and in general the measure-valued Laplace of 
%$\frac{1}{2}d_y^2$ on $X$ will have a negative singular part in the complement of $\overline{B_{\epsilon}(x_0)}$, for instance a cut points of $x_0$.
\end{remark}

\begin{proof}[Proof of Theorem \ref{th:lowerbound}]
We first give a proof for $\md_{\uk} (d_y)$.

\textbf{1.} Assume w.l.o.g. that $B_{8\epsilon}(x_0)\subset U$ and $U$ is homeomorphic to $\mathbb{R}^n$. 
%By Reifenbergs theorem, the $RCD$-condition and the $CAT$-property we can find $\epsilon>0$ such that $B_{8\epsilon}(x_0)$ is homeomorphic to an open subset $U\subset \mathbb{R}^n$.

By the assumptions on the diameter of $X$ the ball  $\overline{B_{4\epsilon}(x_0)}$ is geodesically convex and geodesics in it are unique. Let $(Y,d_{Y},m)=(\overline{B_{4\epsilon}(x_0)},d,\mathcal H_n)$

In particular, $Y$ again satisfies  $RCD(\ke,n)$ and is $CAT(\uk)$.\medskip\\
Then by Lemma~\ref{extend-geod} there is $\delta=\delta(\epsilon,\uk)>0$ such that any  unit speed geodesic $\gamma:[0,L]\rightarrow \overline{B_{4\epsilon}(x_0)}$ 
 we have that
 \begin{equation}\label{extend-gamma}
 \gamma  \text{ can be extended to a geodesic 
$\hat{\gamma}:[-\delta,L+\delta]\rightarrow B_{8\epsilon}(x_0)$ with $\hat{\gamma}|_{[0,L]}=\gamma$. }
 \end{equation}

%Moreover, there exists $\delta_0>0$ such that $\delta(\gamma)\geq \delta_0$ for any geodesic $\gamma$ in $B_{4\epsilon}(x_0)$. Otherwise
%we find a sequence of geodesics $\gamma_i$ with $\delta_i=\delta(\gamma_i)\rightarrow 0$. By compactness of $\overline{B_{4\epsilon}}(x_0)$ $\gamma_i$ converges uniformly to a limit geodesic $\gamma_{\infty}$

\textbf{2.} Now, let $y\in X$ and $B_{4\epsilon}(x_0)$ as before.
%Then $B_{\epsilon}(x_0)\subset B_{2\epsilon}(x_0)\subset B_{3\epsilon}(y)\subset {B_{4\epsilon}(x_0)}\subset B_{8\epsilon}(x_0).$
By Lemma~\ref{lem-contr-flow} there exists an $\mathcal{H}^n$-contracting gradient flow $\Phi^y_t:X\rightarrow X$ for the function $\md_{\kappa}(d_y)\in D({\bf \Delta})$ and 
if $x\in B_{2\epsilon}(x_0)$, then
$t\mapsto \Phi^y_t(x)$ is precisely the geodesic $\gamma:[0,d(x,y)]\rightarrow X$ that connects $x$ with $y$, appropriately parameterized. 

%By above there is $\theta>0$ and an extension 
%$$\hat{\gamma}:[-\theta,d(x,y)]\rightarrow \overline{B_{3\epsilon}(y)}$$ 
%such that $\hat{\gamma}(d(x,y))=y$ and $\hat{\gamma}(-\theta)\in \partial B_{3\epsilon}(x_0)$.

From \eqref{extend-gamma} it easily follows that $B_{\epsilon}(x_0)\subset \Phi_t^y(B_{3\epsilon}(x_0))$ for all sufficiently small $t$.
%, and we observe for $\eta \in [-\theta,0]$ that
%$\Phi_{\eta}^y(\hat{\gamma}(-\eta))=\hat{\gamma}(0)=x$.
%Hence $B_{2\epsilon}(x_0)\subset \Phi_t^y(B_{3\epsilon}(x_0))$, and $B_{\epsilon}(x_0)\subset \Phi_t^y(B_{3\epsilon}(x_0))$ as well for every $y\in B_{\epsilon}(x_0)$.
\\

\textbf{3.} Since 
$
(\Phi_t)_{\star}\mathcal{H}^n(A)\geq \mathcal{H}^n(A)
$
for any subset $A\subset \Phi_t(X)$, we obtain for any Lipschitz function $g$ with compact support in $B_{\epsilon}(x_0)$ 
\begin{align*}
\int g (\Phi_t)_{\star}\mathcal{H}^n\geq \int g d\mathcal{H}^n.
\end{align*}
Since $(\Phi_t(x))_{t\geq 0}$ is a gradient flow curve of $\md_{\uk} (d_y)$ for any $x\in X$, we compute 
\begin{align*}
\int g {\bf\Delta} \md_{\kappa}(d_y) d\mathcal{H}^n&= -\int \langle\nabla g,\nabla \md_{\kappa}(d_y)\rangle d\mathcal{H}^n\\
%&=\int \langle\nabla g (x),\dot{\gamma}_x(0)\rangle d\mathcal{H}^n(x) \\
&= \lim_{t\rightarrow 0}\frac{1}{t}\left[\int g\circ\Phi_t d\mathcal{H}^n-\int g d\mathcal{H}^n\right]\\
&=\lim_{t\rightarrow 0}\frac{1}{t}\left[\int g  d(\Phi_t)_{\star}\mathcal{H}^n-\int g d\mathcal{H}^n\right]  \geq 0
\end{align*}
for any $g\in \Lip_c(B_{\epsilon}(x_0))$. The second equality is the first variation formula. Note, that there is a version of the first variation formula in the class of $RCD$-spaces that is sufficient for our purposes
(see Theorem \ref{th:secondvariation}), but the first variation formula is 
also well-known  for metric spaces which satisfy a $\CAT$-condition \cite{bbi}[Chapter 4].  Hence
$
{\bf\Delta}\md_{\kappa}(d_y)|_{B_{\epsilon}(x_0)}\geq 0
$ for any $y\in X$. 

\begin{comment}
{\color{blue}
\textbf{4.} $\md_{\uk}\in C^2_{loc}((0,\infty))$ has a smooth inverse on $(0,\pi_\uk/2)$ . Then, by the cain rule for the measure valued Laplace (Theorem \ref{prop:chainrulelaplacian}) it follows 
$d_y|_{B_{\epsilon}(x_0)\backslash\left\{y\right\}}\in D({\bf\Delta},B_{\epsilon}(x_0)\backslash \left\{y\right\})$. Moreover, we have the formula
\begin{align*}
0\leq {\bf\Delta} \md_{\kappa}(d_y)|_{B_{\epsilon}(x_0)\backslash \left\{y\right\}}\leq  \sin_{\kappa}d_y{\bf\Delta}d_y|_{B_{\epsilon}(x_0)\backslash \left\{y\right\}} + \cos_{\kappa} d_y |\nabla d_y|^2|_{B_{\epsilon}(x_0)\backslash \left\{y\right\}}
\end{align*}
that gives already the lower estimate for ${\bf \Delta}d_y|_{B_{\epsilon}(x_0)\backslash\left\{y\right\}}$.}
\end{comment}
The proof for $d_y$ is essentially the same as for $\md_{\uk} (d_y)$ in view of Lemma~\ref{lem-contr-flow} with the following difference. By Theorem~\ref{th:upperbound} 
 and the chain rule it follows that  $d_y|_{X\backslash\{y\}}\in  D({\bf\Delta})$ and ${\bf\Delta}d_y$ is locally bounded above on $X\backslash\{y\}$. However, on a general $RCD(k,n)$ space $d_y$ need not lie in $D({\bf\Delta})$.
Nevertheless, under the assumptions of Theorem~\ref{th:lowerbound}, the same proof as above shows that the distributional Laplacian of $d_y|_{B_{\epsilon}(x_0)}$ is nonnegative as a distribution and hence it \emph{is} a measure and  $d_y\in D(\bf\Delta)$. 

\end{proof}

\begin{corollary}\label{cor:important}
Let $(X,d,\m)$, $d_y$, $x_0\in X$ and $U\subset X$ be as in Theorem \ref{th:lowerbound}. 
Then there exists $\epsilon>0$ such that for any cutoff function $\chi\in \mathbb{D}_{\infty}^{X}$ with $\Delta\chi\in L^{\infty}(\mathcal{H}^n)$, $\supp\chi\subset B_{\epsilon}(x_0)$ and $\chi|_{B_{\epsilon/2}(x_0)}=1$ it holds that
$\chi\cdot\md_{\kappa}( d_y)\in D_{L^{\infty}}(\Delta)\cap L^{\infty}(\mathcal{H}^n)\cap \Lip(X)$.  

Further, if $\supp\chi\subset B_{\epsilon}(x_0)\backslash\left\{y\right\}$ then we also have that  
$\chi\cdot d_y\in D_{L^{\infty}}(\Delta)\cap L^{\infty}(\mathcal{H}^n)\cap \Lip(X)$.
\end{corollary}
%\begin{remark}
%Since $\md_{\kappa}$ is smooth and strictly monotone, from the chain rule of the measure valued Laplace operator we also have $\md_{\kappa}(\chi\cdot d_y), \chi\cdot d_y\in D_{L^{\infty}}(\Delta)\cap L^{\infty}(\mathcal{H}^n)\cap \Lip(X)$.

%{\color{red} This was not true for $\chi\cdot d_y$ as stated because it's not semiconcave at $y$ and has no upper laplacian bound at $y$. It's only semiconcave on $X\backslash\{y\}$.}
%\end{remark}

\begin{proof} We choose $\epsilon>0$ as in the previous theorem, and a corresponding cutoff function $\chi$. 
%Then, Lemma \ref{lem:cutoff} provides a cutoff function $\chi\in\mathbb{D}_{\infty}^X$ with $\Delta\chi\in L^{\infty}(\mathcal{H}^n)$, $\supp\chi\subset B_{\epsilon}(x_0)$ and $\chi|_{B_{\epsilon/2}(x_0)}=1$. 
Clearly it holds that $\md_{\kappa}(\chi\cdot d_y)\in L^{\infty}(\mathcal{H}^n)\cap \Lip(X)$. 
By Corollary \ref{cor-mdk-laplace-comp} we also have $\md_{\kappa}(d_y)\in D({\bf\Delta})$.
%By the Laplace comparison statement for $RCD$-spaces in Theorem \ref{th:upperbound} we also have $\md_{\kappa}(d_y)\in D({\bf\Delta})$. This can be seen as in the proof of Corollary \ref{cor-mdk-laplace-comp}.

Hence, the Leibniz rule for the measure valued Laplacian \cite[Theorem 4.12]{giglinonsmooth} yields
\begin{align*}
{\bf\Delta}(\chi \md_{\kappa}(d_y))=\chi {\bf\Delta}\md_{\kappa}(d_y)|_{B_{\epsilon}(x_0)} + \md_{\kappa}(d_y){\Delta}\chi + 2\langle \nabla \md_{\kappa}(d_y),\nabla \chi\rangle.
\end{align*}
By Theorem \ref{th:lowerbound} and again by Theorem \ref{th:upperbound} we know that ${\bf\Delta}\md_{\kappa}(d_y)\in L^{\infty}(\mathcal{H}^n)$. It follows that ${\bf\Delta}(\chi\cdot\md_{\kappa}(d_y))\in L^{\infty}(\mathcal{H}^n)$. 
Since $\chi\cdot\md_{\kappa}(d_y)$ is compactly supported in $B_{\epsilon}(x_0)$, we also get that $\chi\cdot\md_{\kappa}(d_y)\in D(\Delta)$ and therefore ${\bf \Delta}(\chi\cdot\md_{\kappa}(d_y))=\Delta(\chi\cdot\md_{\kappa}(d_y))$.

The proof for $d_y$ is the same.
\end{proof}

\section{On the relation between convexity and the Hessian}\label{sec:convexity-and-hessian}
\noindent 
In this section we explore the relation between convexity and almost everywhere lower bounds for the Hessian of a function $f$ that is in a sufficiently regular subspace of $W^{2,2}(X)$. 
This relation has already been studied in previous publications \cite{ketterer3, GKKO, hanconvexity, gigtam}.
%Especially, the second variation formula in \cite{gigtam} provides a very complete understanding of the matter in the context of $RCD(\ke,N)$-spaces with finite $N$.
A novelty of our situation is that we give a localized statement that is needed in the course of the paper. Moreover, we will show that $\kappa$-convexity implies a lower $\kappa$-bound 
for the Hessian. By the second variation formula this lower bound holds if the Hessian is evaluated {on gradients of  }Kantorovich potentials.  
However, we require the estimate for the Hessian evaluated  {on gradients of } test functions.

Throughout this section let $(X,d_X,\m_X)$ be a compact metric measure space satisfying the condition $RCD(\ke,N)$, and let $Z$ be a closed subset of $X$ 
such that $\m_X(\Int Z)>0$, $\m(\partial Z)=0$ and $(Z,d_Z,\m_{Z})$ is a metric measure space that also satisfies
the condition $RCD(\ke,N)$. We denote by $\Delta^X, \Gamma_2^X$ ect. and $\Delta^Z,\Gamma_2^Z$ ect. the Laplace operator, 
the $\Gamma_2$-operator ect.
of $X$ and $Z$, respectively. In particular, $(Z,d_Z)$ is geodesically convex and compact as well. 

Let $f\in D(\Delta^X)\cap L^{\infty}(\m_X)\cap \Lip(X)$ with $\left\|f\right\|_{L^{\infty}}$, $\left\||\nabla f|\right\|_{L^{\infty}}\leq C$ for $C\in (0,\infty)$. In particular, $f\in H^{2,2}(X)$.
\noindent
We introduce the transformed measures
\begin{equation}
%\tilde{\m}_X:=e^{-f}\m_X\ \ \mbox{and}\ \ 
\tilde{\m}_Z:=\left[e^{-f}\m_X\right]|_Z\ \ \
\end{equation}
and consider the metric measure space
%s $(X,d_X,\tilde{\m}_X)=\tilde{X}$ and 
$(Z,d_Z,\tilde{\m}_Z)=\tilde{Z}$.
%We will denote with $\tilde{\Delta}${, $\tilde{\Gamma}_2$ 
%and $\tilde{\Gamma}_2'$ the corresponding Laplacian, $\Gamma_2$-operator and modified $\Gamma_2$-operator. 
We remark that $f|_{Z}\in L^{\infty}(\m_{Z})\cap \Lip(Z)$ with $\left\|f\right\|_{L^{\infty}}, \left\||\nabla f|\right\|_{L^{\infty}}\leq C$ but $f\notin D(\Delta^Z)$. 
%Later, we also consider a family $(f_{\alpha})_{\alpha>0}$ of functions as above. In this case we will indicate the dependence on $\alpha$ by a
% superscript, 
%for instance we will write ${\Delta}^{\alpha}$ for the Laplace operator of the transformed space $(Z,d_Z,[e^{-f_{\alpha}}\m]|_{Z})$. For instance, we will write $\tilde{\m}_Z=[e^{-f_{\alpha}}\m]|_{Z}=:\tilde{\m}^{\alpha}_{Z}$.
%
%
%\footnote{\color{red} Earlier, $n$ denotes the dimension. Probably better to use a different letter here}
We observe that, for $p \in [1,\infty]$,
\[ \e^{-C/p} \|u\|_{L^p({\m_Z})} \le \|u\|_{L^p(\tilde{\m}_Z)} \le \e^{C/p} \|u\|_{L^p(\m_Z)} \]
for all $u \in L^p(\m_Z)=L^p(\tilde{\m}_Z)$, and
\[ \e^{-C/p} \big\| |\nabla u| \big\|_{L^p({\m}_Z)} \le \big\| |\nabla u| \big\|_{L^p(\tilde{\m}_Z)}
 \le \e^{C/p} \big\| |\nabla u| \big\|_{L^p(\m_Z)} \]
for all {$u \in W^{1,2}(Z)=W^{1,2}(\tilde{Z})$. 
In addition, the minimal weak upper gradient of $u \in W^{1,2} (\tilde{Z})$ 
induced by $\tilde{\m}_Z$ coincides with $|\nabla u|$} 
(see \cite[Lemma~4.11]{agsheat}). 

\begin{lemma}\label{importantlemma2}
Let $f$ and $(Z,d_Z,\tilde{\m}_Z)$ be as above.
Then we have $D(\widetilde{\Delta}^Z)=D(\Delta^Z)$ and, for any $u \in D(\widetilde{\Delta}^Z)$,
\begin{itemize}
\item[(i)] $\widetilde{\Delta}^Zu =\Delta^Z u- \langle \nabla f,\nabla u \rangle$,
\medskip
\item[(ii)] $\|\widetilde{\Delta}^Z u\|_{L^2(\tilde{\m}_Z)}^2
 \le 2\e^{C/2} \left( \| \Delta^Z u\|_{L^2({\m}_Z)}^2
 +\| {| \nabla f |} \|_{L^{\infty}}^2 \big\| |\nabla u| \big\|_{L^2({\m}_Z)}^2 \right)$,
\medskip
\item[(iii)] $\| \Delta^Z u \|_{L^2({\m}_Z)}^2
 \le 2\e^{C/2} \left( \|\widetilde{\Delta}^Z u\|_{L^2(\tilde{\m}_Z)}^2
 +\| {| \nabla f |} \|_{L^{\infty}}^2 \big\| |\nabla u| \big\|_{L^2(\tilde{\m}_Z)}^2 \right)$.
\end{itemize}
In particular, if $u\in D(\widetilde{\Delta}^Z)$,
then $P^Z_t u \in D(\widetilde{\Delta}^Z)$ and $P^Z_t u \to u$ in $D(\widetilde{\Delta}^Z)$ as $t \to 0$.
\end{lemma}
\begin{proof} The lemma can be found in \cite[Lemma 3.4]{GKKO} where it is assumed that $f\in\mathbb{D}_{\infty}^Z$. However, one can easily check that the proof works for $f\in L^{\infty}(\m)\cap \Lip(Z)$.
%Consider $u \in D(\widetilde{\Delta}) \subset W^{1,2}({X}, {\tilde\m})$,
%then $\widetilde{\Delta}u+\langle \nabla u,\nabla f\rangle\in L^2({X},\tilde{\m})=L^2({X},\m)$.
%Given $g\in W^{1,2}({X},\m)$, we have $\tilde{g}:=\e^f g\in W^{1,2}({X},\m) =W^{1,2}({X},\tilde{\m})$ and
%\begin{align*}
%\int_X \big( \widetilde{\Delta} u +\langle \nabla f,\nabla u \rangle \big) g \,d\m
%&= \int_X \tilde{g} \widetilde{\Delta} u \,d\tilde{\m}+\int_X \langle \nabla f,\nabla u \rangle g \,d\m\\
%&= -\int_X \langle \nabla \tilde{g},\nabla u \rangle \,d\tilde{\m}+\int_X \langle \nabla f,\nabla u \rangle g \,d\m\\
 %= -\int_X \langle \nabla \tilde{g},\nabla u \rangle \e^{-f} \,d{\m}+\int_X \langle \nabla f,\nabla u \rangle g \,d\m\\
%&= -\int_X \left\{\langle \nabla g,\nabla u\rangle +\langle \nabla f,\nabla u \rangle g\right\} \,d\m+\int_X \langle \nabla f,\nabla u \rangle g \,d\m\\
%&= -\int_X \langle \nabla g,\nabla u \rangle \,d\m.
%\end{align*}
%This shows $u \in D(\Delta)$ and the equation in (i).
%Similarly, $u\in D(\Delta)$ implies $u\in D(\widetilde{\Delta})$ (hence $D(\Delta)=D(\widetilde{\Delta})$)
%and the equation in (i).
%(ii) and (iii) follow easily from (i).
\end{proof}
\begin{proposition}\label{prop-conf-conv}
Let $(X,d,\m)$, $(Z,d_Z,\m_Z)$, $f$ and $\tilde{Z}$ be as above. Assume $f|_Z$ is $\kappa$-convex  on $(Z,d_Z,\m|_Z)$ for $\kappa\in \mathbb{R}$.
Then $\tilde{Z}$ satisfies the condition $RCD(\ke+\kappa,\infty)$, and for $u\in \mathbb{D}_{\infty}^Z$ with $\supp u\subset \Int Z$, $\phi\in \Lip(Z),\, \phi\ge 0$ and $\tilde{\phi}:=e^f\phi$
%and $\phi\in D_{L^{\infty}}(\Delta)\cap L^{\infty}(\m)\cap \Lip(X)$ 
we have $u\in \mathbb{D}^{\tilde{Z}}_{\infty}$ and 
\begin{align}\label{gammas}
\int_Z(\kappa+\ke)|\nabla u|^2\phi d\m &\leq \int\tilde{\phi} d{\bf \Gamma}_2^{\tilde{Z}}(u)\nonumber\\
&=\int \phi d{\bf \Gamma}_2^Z(u)+\int_Z \Hess^Xf(\nabla u,\nabla u)\phi d\m.
\end{align}
\end{proposition}

\begin{remark}
The conditions $RCD(K,N)$  and $\m(Z)>0$ for $(Z,d_Z,\m|_Z)$ imply that $(Z,d_Z)$ is a complete, compact, \textit{geodesic} metric space. Therefore, it makes sense to consider functions $f$ on $Z$ that are $\kappa$-convex in the sense of Definition \ref{def:metricconvex}.
\end{remark}

\begin{proof} \textbf{1.} That $\tilde{Z}$ satisfies the condition $RCD(\kappa+\ke,\infty)$, follows from Fact \ref{fact:cd} (iii), Lemma~\ref{conv->CD} and from the fact that $\tilde{Z}$ is again infinitesimally Hilbertian.
\\

\textbf{2.} We show that $u\in \mathbb{D}^{\tilde{Z}}_{\infty}$. Since $u\in \mathbb{D}^{Z}_{\infty}$, we have by definition that $u\in D_{W^{1,2}}(\Delta^Z)\cap \Lip(Z)\cap L^{\infty}(\m_Z)$.
From Lemma~\ref{importantlemma2} we know that $D(\Delta^Z)=D(\Delta^{\tilde{Z}})$ and $\Delta^{\tilde{Z}}u=\Delta^Zu-\langle \nabla u,\nabla f\rangle$. Since $u\in D_{W^{1,2}}(\Delta^Z)$ we already know that 
$\Delta^Zu\in W^{1,2}(Z)=W^{1,2}(\tilde{Z})$. Moreover, since $\supp u\in \Int Z$, it easily follows that $u\in\mathbb{D}^X_{\infty}$. Then, since $f\in W^{2,2}(X)$ and since $u$ and $f$ are Lipschitz,
it follows by Proposition \ref{prop:nuetzlich} that $\langle \nabla u,\nabla f\rangle \in W^{1,2}(X)$. Hence, $\langle\nabla u,\nabla f\rangle\in W^{1,2}(Z)$. Consequently, $\langle \nabla u,\nabla f\rangle\in W^{1,2}(\tilde{Z})$ and 
therefore $\Delta^{\tilde{Z}}u\in W^{1,2}(\tilde{Z})$.
Moreover, since $u\in \mathbb{D}_{\infty}^X$ $\Hess f(\nabla u,\nabla u)\in L^2(\m_X)$ is well-defined.
\\

\textbf{3.} Since $\tilde{Z}$ satisfies the condition $RCD(\kappa+\ke,\infty)$, the improved Bochner inequality yields for $u\in \mathbb{D}_{\infty}^{\tilde{Z}}$
\begin{align*}
{\bf\Gamma}_2^{\tilde{Z}}(u)=\frac{1}{2}{\bf\Delta}^{\tilde{Z}}|\nabla u|^2-\langle \nabla u,\nabla \Delta u\rangle\left[e^{-f}\m\right]|_Z\geq (\kappa + \ke)|\nabla u|^2\left[e^{-f}\m\right]|_Z.
\end{align*}
Recall that $|\nabla u|^2\in D({\bf\Delta}^{\tilde{Z}})$ if $u\in \mathbb{D}_{\infty}^{\tilde{Z}}$.
If we integrate $\tilde{\phi}=\phi e^f\in \Lip(Z)$ w.r.t. the previous measures, the definition of the measure valued Laplacian yields for the left hand side
\begin{align*}
\int \tilde{\phi}d{\bf\Gamma}_2^{\tilde{Z}}(u)&=-\frac{1}{2}\int_Z  \phi e^{f}d{\bf\Delta}^{\tilde{Z}} |\nabla u|^2 - \int_Z \langle \nabla u,\nabla \Delta^{\tilde{Z}}u\rangle \phi d\m\\
&=-\frac{1}{2}\int_Z \langle \nabla |\nabla u|^2,\nabla \phi e^{f}\rangle e^{-f}d\m - \int_Z \langle \nabla u,\nabla \Delta^{\tilde{Z}}u\rangle \phi d\m\\
&=-\frac{1}{2}\int_Z \langle \nabla |\nabla u|^2,\nabla \phi\rangle d\m - \frac{1}{2}\int_Z\langle \nabla |\nabla u|^2,\nabla f\rangle \phi d\m\\
&\hspace{3cm}- \int_Z \langle \nabla u,\nabla \Delta^{Z}u\rangle \phi d\m + \int_Z\langle \nabla u,\nabla \langle \nabla u,\nabla f\rangle\rangle \phi d\m\\
&= \int \phi d{\bf\Gamma}_2^{Z}(u) + \int_Z\Hess^X f(\nabla u,\nabla u)\phi d\m.
\end{align*}
For the last equality also recall that for every $g\in W^{1,2}(X)$ we have $g\in W^{1,2}(Z)=W^{1,2}(\tilde{Z})$ and $|\nabla^Z g|=|\nabla^{\tilde{Z}}g|=|\nabla^Xg||_Z$.
This completes the proof of the proposition.
\end{proof}
Let us first recall the following lemma from \cite[Lemma 6.7]{amslocal}.
\begin{lemma}\label{lem:cutoff}
Let $(X,d,\m)$ be a metric measure space satisfying a $RCD$-condition. Then for all $E\subset X$ compact and all $G\subset X$ open such that $E\subset G$ there exists a Lipschitz function $\chi:X\rightarrow [0,1]$ with
\begin{itemize}
 \item[(i)] $\chi=1$ on $E_h=\left\{x\in X:\exists y\in E: d(x,y)<h\right\}$ and $\supp\chi\subset G$,
 \medskip
 \item[(ii)] ${\bf\Delta}\chi\in L^{\infty}(\m)$ and $|\nabla\chi|^2\in W^{1,2}(X)$.
\end{itemize}
\end{lemma}
\begin{remark}
Following the proof of this Lemma  in \cite{amslocal} we see that one can choose $\chi$ to be in $\mathbb{D}_{\infty}^X$.
\end{remark}

\begin{corollary}
Let $X$, $Z$ and $f$ be as in Proposition ~\ref{prop-conf-conv}. Then
\begin{align*}
\Hess f(1_Z\nabla u ,1_Z\nabla u)=1_Z\Hess^Xf(\nabla u,\nabla u)\geq \left[\kappa|\nabla u|^2 
%- \frac{1}{\beta}\langle \nabla f,\nabla u\rangle
\right] 1_Z\ \ \ \m\mbox{-a.e.}
\end{align*}
$\mbox{for every }u\in \mathbb{D}^X_{\infty}.$
\end{corollary}
\begin{proof} \textbf{1.} Let $\alpha>0$ and define $f/\alpha=f_{\alpha}$. Then $f_{\alpha}$ is $\frac{\kappa}{\alpha}$-convex. 
Let $u\in\mathbb{D}_{\infty}^X$ and choose a cut-off function $\chi\in \mathbb{D}_{\infty}^X$ such that $\chi=1$ on $A\subset \Int Z$ for a closed set $A$ and $\supp u\subset \Int Z$. Then $\chi\cdot u\in \mathbb{D}_{\infty}^{Z}$ with 
$\supp \chi\cdot u\subset \Int Z$. Therefore by  Proposition ~\ref{prop-conf-conv}
\begin{align}
\int_Z(\kappa/\alpha+\ke)|\nabla (\chi\cdot u)|^2\phi d\m 
%+ \frac{1}{N+\frac{\beta}{\alpha}}\int (\Delta^Z u- \langle \nabla f/\alpha,\nabla u\rangle)^2 \phi d\m \\
\leq \int \phi d{\bf \Gamma}_2^Z(\chi\cdot u)+\int_Z \Hess^X(f/\alpha)(\nabla (\chi\cdot u),\nabla (\chi\cdot u))\phi d\m.
\end{align}
Hence, multiplying with $\alpha>0$ and letting $\alpha\rightarrow 0$ this  yields 
\begin{align*}
\kappa\int |\nabla (\chi\cdot u)|^2\phi 
%+ \frac{1}{\beta}\int \langle \nabla f,\nabla(\chi\cdot u)\rangle^2 \phi 
d\m\leq \int_Z \Hess^Xf(\nabla(\chi\cdot u),\nabla(\chi\cdot u))\phi d\m
\end{align*}
for every nonnegative $\phi\in\Lip(X)$. By standard approximation the same holds  for any nonnegative $\phi\in C_b(Z)$.
\\

\textbf{2.} We choose a sequence of nonnegative $\phi_k\in C_b(X)$, $k\in\mathbb{N}$, compactly supported in $\Int Z$ such that $\phi_k\uparrow 1$ pointwise $\m$-a.e.\ . Moreover, we choose cut-off functions $\chi_k$ as in \textbf{1.}
with $A=\supp\phi_k$. Then
\begin{align*}
\kappa\int |\nabla u|^2\phi_k d\m %+ \frac{1}{\beta}\int \langle\nabla f,\nabla u\rangle \phi d\m 
\leq \int \Hess^Xf(\nabla u,\nabla u)\phi_kd\m
\end{align*}
for every $k\in\mathbb{N}$ and $u\in \mathbb{D}_{\infty}^X$.  Since $\m(\partial Z)=0$, letting $k\rightarrow \infty$ yields the claim.
\end{proof}
\begin{theorem}\label{th:important}
Let $(X,d,\m)$ be a compact metric measure space that satisfies the condition $RCD(K,N)$ for $K\in \mathbb{R}$ and $N>0$, and $Z\subset X$ be a closed subset such that $\m(\Int Z)>0$, $\m(\partial Z)=0$ and $(Z,d_Z,\m|_Z)$ satisfies
the condition $RCD(K,N)$ as well. Let $f\in D_{L^{2}}(\Delta)\cap \Lip(X)\cap L^{\infty}$.  Let $\kappa\in \mathbb{R}$. % and $\beta\in (0,\infty]$.
%Then $f$ is $\kappa$-convex on $Z$ if and only if $\Hess^X f\geq \kappa$ $\m$-a.e. on $Z$. 

%Moreover, the 
Then the following statements are equivalent\smallskip
\begin{itemize}
 \item[(i)] $f$ is $\kappa$-convex on $(Z,d_Z)$,\medskip
 \item[(ii)] $\Hess^X f(1_Z\nabla u,\nabla u)\geq \uk|\nabla u|^21_Z %+ \frac{1}{\beta}\langle \nabla f,\nabla u\rangle^21_Z
 $ $\m$-a.e. for $u\in \mathbb{D}^X_{\infty}$.
\end{itemize}

\end{theorem}

\begin{proof} The implication
(i)\ $\Rightarrow$\ (ii) is precisely the content of the previous corollary. \smallskip\\
For (ii)\ $\Rightarrow$\ (i) assume $\Hess^Xf(\nabla u,\nabla u)\geq \kappa|\nabla u|^2 %+ \frac{1}{\beta}\langle \nabla f,\nabla u\rangle^2
$ $\m$-a.e. on $Z$ for any $u\in W^{1,2}(X)$. 
Then the  second variation formula in Theorem \ref{th:secondvariation} implies that $t\in [0,1]\mapsto \mathcal{F}(\mu_t)=\int fd\mu_t$ is in $C^2([0,1])$ for a $L^2$-Wasserstein geodesic $(\mu_t)_{t\in [0,1]}$ with $\mu_t\leq C\m$ for some constant $C>0$, and 
\begin{align*}
\frac{d^2}{dt^2}\mathcal{F}(\mu_t)=\int \Hess^X(\nabla\phi_t,\nabla \phi_t)d\mu_t&\geq \kappa\int |\nabla \phi_t|^2 d\mu_t 
%+ \int \langle \nabla f,\nabla \phi_t\rangle d\mu_t
=\kappa W_2(\mu_0,\mu_1)^2.
%+ \frac{d}{dt}\mathcal{F}(\mu_t).
\end{align*}
Hence, $t\in [0,1]\mapsto \mathcal{F}(\mu_t)$ is $\kappa$-convex, and %In the case $\beta<\infty$ this says $\mathcal{V}(\mu_t)=e^{-\frac{1}{\beta}\mathcal{F}(\mu_t)}$ satisfies 
%\begin{align}
%\frac{d^2}{dt^2}\mathcal{V}(\mu_t)\leq - \frac{\kappa}{\beta} W_2(\mu_0,\mu_1)^2 \mathcal{V}(\mu_t).
%\end{align}
%In the case when $\beta=\infty$ w
we obtain that 
\begin{align}\label{third}
\mathcal{F}(\mu_t)\leq (1-t)\mathcal{F}(\mu_0)+t\mathcal{F}(\mu_1)- \frac{1}{2}K(1-t)tW_2(\mu_0,\mu_1)^2
\end{align}
%for every Wasserstein geodesic $(\mu_t)_{t\in [0,1]}$ such that $\mu_t\leq C\m$ for some constant $C>$.
%and in the case $\beta<\infty$ we obtain that
%\begin{align}\label{first}
%\mathcal{V}(\mu_t)\geq \sigma_{\kappa/\beta}^{(1-t)}(W_2(\mu_0,\mu_1)) \mathcal{V}(\mu_0)+\sigma_{\kappa/\beta}^{(t)}(W_2(\mu_0,\mu_1)) \mathcal{V}(\mu_1).
%\end{align}
Now, we know that for every point $x_0\in X$ and $\m$-a.e. point $x_1\in X$ there exists a unique geodesic $\gamma$ (Corollary 1.4 in \cite{grs}).
We pick two such points in $Y$, and sequences of $\m$-absolutely continuous probability measures $(\mu_0^k)_{k\in\mathbb{N}}$ and $(\mu_1^k)_{k\in\mathbb{N}}$ (for instance $\mu_i^k=\m(B_{1/k}(x_i))^{-1}\m|_{B_{1/k}(x_i)}$) 
such that $\mu_i^k\rightarrow \delta_{x_i}$ weakly for every $k\in \mathbb{N}$. 
Assume moreover that $\mu_i^k$, $i=0,1$, satisfies $\mu_i^k\leq C(k)\m$. Then the Wasserstein geodesic $(\mu_t^k)_{t\in[0,1]}$ between $\mu_0^k$ and $\mu_1^k$ satisfies $\mu_t^k\leq \tilde{C}(k,C(k),K,N)\m$ by \cite{rajala2}.
Hence, (\ref{third}) holds for $(\mu_t^k)_{t\in [0,1]}$. We can extract a subsequence such that $(\mu_t^k)_{k\in\mathbb{N}}$ converges for $t\in [0,1]\cap\mathbb{Q}$ to $\nu_t$ for a geodesic $(\nu_t)_{t\in [0,1]}$ between $\delta_{x_0}$ and $\delta_{x_1}$.
Since $x_0$ and $x_1$ are chosen such that there is only one geodesic $\gamma$ in $X$
between them, we must have $\nu_t=\delta_{\gamma(t)}$. Moreover, since $f$ is continuous, by weak convergence of $\mu_t^k$, $\mathcal{F}(\mu^k_t)\rightarrow f(\gamma(t))$ for every $t\in [0,1]\cap\mathbb{Q}$.
Hence, $\forall t\in [0,1]\cap \mathbb{Q}$, we have
\begin{align}\label{second}
f(\gamma_t)\leq& (1-t)f(\gamma_0)+tf(\gamma_1) - \frac{1}{2}K(1-t)t d^2(\gamma_0,\gamma_1)
%f(\gamma_t)^{-\frac{1}{\beta}}\geq& \sigma_{\kappa/\beta}^{(1-t)}(|\dot{\gamma}|) f(\gamma_0)^{-\frac{1}{\beta}}+\sigma_{\kappa/\beta}^{(t)}(|\dot{\gamma}|) f(\gamma_1)^{-\frac{1}{\beta}} \ \ \ \ \ \ \  \mbox{ if }\beta<\infty,
\end{align}
and by continuity this holds for every $t\in [0,1]$.
Now, one can easily see that this implies (i) for $f:Z\rightarrow \mathbb{R}$ in the weak sense. That is for any pair of points $x_0,x_1\in Z$ we can find a geodesic $\gamma$ such that the inequality holds. 
Indeed, if $x_0$ and $x_1$ are arbitrary, we can find a point $\tilde{x}^k_1$ such that $d(x_1,\tilde{x}^k_1)\rightarrow 0$ if $k\rightarrow \infty$  and such that the geodesic $\tilde\gamma^{k}=[x_0,\tilde x_k]$ is unique. Then (\ref{second}) holds for $\tilde{\gamma}^{k}$. 
Since $X$ is locally compact, by passing to a subsequence we can assume that $[x_0,\tilde x_k]\to [x_0,x_1]$. Since $f$ and $d$ are continuous by passing to the limit we obtain that (\ref{second}) holds for $[x_0,x_1]$ as well.

%we can extract a subsequence of geodesics satisfying (\ref{second}) such that this subsequence uniformly converges to a geodesic between $x_0$ and $x_1$. Since $f$ and $d$ are continuous our claim follows.

Finally, we recall the following theorem by Sturm.
\begin{theorem}[\cite{sturmgradient}]\label{th:evi}
Let $(X,d,\m)$ be a metric measure space that is locally compact and satisfies the condition $RCD(\ke,N)$ for $\ke\in \mathbb{R}$ and $N\in (0,\infty]$. Let $V:X\rightarrow (-\infty,\infty]$ be a function that is continuous and satisfies 
$V(x)\geq -C_2 - C_1d(x_0,x)^2$ for constants $C_1,C_2>0$ and $x_0\in X$. Let $\kappa\in \mathbb{R}$. Then, the following properties are equivalent:
\begin{itemize}
 \item[(i)] $V$ is weakly $\kappa$-convex,
 \item[(ii)] $V$ is $\kappa$-convex,
 \item[(iii)] For any $x_0\in X'$ there exists a curve $(x_t)_{t\geq 0}$ in $X'$ such that for all $z\in X'$ and every $t>0$ we have 
 \begin{align*}
 \frac{1}{2}\frac{d}{dt}d(x_t,z)^2+ \frac{\kappa}{2}d(x_t,z)^2\leq V(z)-V(x_t).
 \end{align*}
where $X'$ is the closure of $\Dom V$ in $X$.
We say $(x_t)_{t\geq 0}$ is an $EVI_{\kappa}$ gradient flow curve. 
\end{itemize}
\end{theorem}
This finishes the proof. \end{proof}
\begin{remark}
An $EVI_{\kappa}$ gradient flow curve $(x_t)_{t\geq 0}$ of $V$ comes with the parametrization such that 
\begin{align*}
\frac{d}{dt}V(x_t)=-|\nabla^- V|^2(x_t)=-|\dot{x}_t|^2,
\end{align*}
where $|\nabla^-V|(x)=\limsup_{y\rightarrow x}\frac{(V(y)-V(x))^-}{d(x,y)}$ is the descending slope of $V$ that is general different from the minimal weak upper gradient that was defined before. 
%However, $|\nabla ^-V|$ is an upper gradient for $V$.
%Therefore, if $\Psi:X\rightarrow C([0,1],X)$ with $\Psi(x_0)=(x_t)_{t\in [0,1]}$, the measure $(\Psi)_{\star}\m=\Pi$ represents the gradient of $V$ in the sense of Gigli \cite{giglistructure}.
%Consequently, for any Lipschitz function $g\in \Lip(X)$, 
%\begin{align*}
%\frac{d}{dt}\Big|_{t=0}\int g(x)d(\Phi_t)_{\star}\m=
%\frac{d}{dt}\Big|_{t=0}\int g(\gamma_t)d\Pi(\gamma)=\int \langle \nabla g,\nabla V\rangle d \m.
%\end{align*}
 Hence, an $EVI_{\kappa}$ gradient flow is actually an inverse gradient flow in the standard sense.
\end{remark}

\section{RCD+CAT implies Alexandrov}\label{sec:RCD+CAT-to-Alexandrov}
The goal of this section is to prove the following Theorem which  implies Theorem~\ref{main-thm} by globalization under the extra assumption that $X$ is infinitesimally Hilbertian. In section~\ref{CD+CAT to RCD+CAT} we will show that the infinitesimal hilbertianness assumption can be dropped which will finish the proof of Theorem~\ref{main-thm}  in full generality.
\begin{theorem}\label{cor:4}\label{rcd+cat-to-cbb} 
Let $2\le n\in\mathbb{N}$ and $(X,d,\m)$ be a metric measure space satisfying $RCD(\ke,n)$ for $\ke\in\mathbb{R}$  with $\m=\mathcal{H}^n$, and assume $(X,d)$ is also $CAT(\uk)$. Then
\smallskip
\begin{enumerate}
\item \label{low-curv-bound} $\uk (n-1)\geq \ke$ and $(X,d)$ is an Alexandrov space of  curvature bounded below by $\ke-\uk (n-2)$.
\smallskip
%\item[(2)] $X_{reg}$ is an open $C^3$-manifold, and $d$ is induced by a Riemannian tensor $g$  on $X_{reg}$ which is in $C^{1,\alpha}\cap W^{2,p}(X_{reg})$ for every $1\le p<\infty, 0<\alpha<1$.
%\medskip
\item \label{equal-case} If $\uk(n-1)=\ke$, then $(X,d)$ is isometric to a geodesically convex subset of the simply connected space form $\SS^n_\uk$ of constant curvature $\uk$.
\end{enumerate}
\end{theorem}
\noindent

In the proof of Theorem \ref{cor:4} we will need the following elementary lemma
\begin{lemma}\label{cot-ineq}
Let $\uk,K\in R, n\ge 2$.
Let $\hat\kappa< \ke-(n-2)\uk$. There is $\nu=\nu(n,K,\kappa,\hat\kappa)>0$ such that
\begin{equation}\label{eq:cot-ineq}
(n-1)\cot_{K/(n-1)}(t)- (n-2) \cot_{\uk}(t)<\cot_{\hat\kappa}(t) \mbox{  for all }0<t<\nu.
\end{equation}

\end{lemma}
\begin{proof}
For any real $k$ we have the following Taylor expansions at $0$ 
\[
\sin_k(t)=t-\frac{kt^3}{6}+\ldots,\quad \cos_k(t)=1-\frac{k t^2}{2}+\ldots
\]
Hence
\[
\cot_k(t)=\frac{1-\frac{k t^2}{2}+\ldots}{t-\frac{kt^3}{6}+\ldots}=\frac{1}{t}(1-\frac{k t^2}{2}+\ldots)(1+\frac{kt^2}{6}+\ldots) =\frac{1}{t} -\frac{kt}{3}+\ldots
\]
Applying this to both sides of \eqref{eq:cot-ineq}  yields the Lemma.
\end{proof}

\textit{Proof of Theorem \ref{cor:4}.}
We will prove the theorem via induction w.r.t. $n\in \mathbb{N}$.
\\

\begin{comment}
$n=1$: In this case $(X,d,\m)$ satisfies $RCD(K,1)$. Then the following cases occur
\begin{align}\label{dim=1}
(X,d)=\begin{cases}
      \left\{pt\right\}\ \ & \ \mbox{ if }\ke>0,\\
      I, \ \mathbb{R},\RR_+ \mbox{ or } c\mathbb{S}^1 \mbox{ for }c>0 \ & \ \mbox{ if }\ke\leq 0.
      \end{cases}
\end{align}
where $I\subset \mathbb{R}$ is an interval, and $c\mathbb{S}^1$ is the rescaled circle with radius $c\pi$. Then (1), (2) and (3) follow with $\uk=0$.
%Therefore, if $\ke>0$, (1), (2) and (3) are true. If $\ke\leq 0$, then each of the admissible spaces is $CBA(\uk)$ for any $\uk\leq 0$, therefore $\uk(1-1)=0\geq \ke$ and (1) and (2) hold. Moreover, if $\uk (1-1)=0=\
\end{comment}

{\bf 1.} Let $n\ge 2$ and suppose Theorem \ref{cor:4} is true for $n-1$ if $n>2$.  Let $(X,d,\mathcal H_n)$ be $RCD(\ke,n)$ and $\CAT(\uk)$.
The base of induction $n=2$ will be handled in the same way as the general induction step with one small difference which we'll explicitly indicate.

Following Gigli and Philippis~\cite{GP-noncol} for any $x\in X$ we consider the monotone quantity $\frac{m(B_r(x))}{v_{k,n}(r)}$ which is non increasing in $r$ by the Bishop-Gromov volume comparison. 
Let $\theta_{n,r}(x)=\frac{m(B_r(x))}{\omega_nr^n}$. Consider the density function $\theta_{n}(x)=\lim_{r\to 0}\theta_{n,r}(x)=\lim_{r\to 0}\frac{m(B_r(x))}{\omega_nr^n}$.

 Since $n$ is fixed throughout the proof we will drop the  subscripts $n$  and from now on use the notations $\theta(x)$ and $\theta_{r}(x)$ for $\theta_{n}(x)$ and $\theta_{n,r}(x)$ respectively.

Note that $\theta(x)=1$ a.e. by ~\cite{GP-noncol} but it still makes sense and is well defined  pointwise. {Also, by~\cite{GP-noncol}  $\theta$ is  lower semicontinuous and hence $0<\theta(x)\le 1$ for any $x\in X$.

Let $x\in X$ be arbitrary. Let $(T_xX, d_x,m_x,o)=\lim_{r_i\to 0}(X, \frac{1}{r_i}d, \frac{1}{r_i^n}m,x)$ be a tangent cone at $x$.
Note that under this normalization along the sequence the unit balls around $x$ \emph{do not} have measure 1 but the measure $\frac{1}{r_i^n}m$ is equal to $\mathcal H_n$ with respect to the rescaled metric $\frac{1}{r_i}d$.} Obviously, $T_xX$ is $CAT(0)$ and $RCD(0,n)$. By~\cite{GP-noncol} it is also noncollapsed i.e. $m_x=\mathcal H_n$. Therefore, by the defintion of $\theta$ we have that
$\theta(x)=\frac{\mathcal H_n(B(o,1))}{\omega_n}$ where $o$ is the apex of  $T_xX$. Note that this is true even if the tangent cone $(T_xX, d_x,m_x,o)$ is not unique.

Further,   $(T_xX, d_x,m_x,o)$  is  a volume metric cone by \cite[Proposition 2.7]{GP-noncol}. Therefore, by the \textit{volume-cone-implies-metric-cone} theorem in \cite{DGi}, $T_xX=C(\Sigma)$ where $(\Sigma, d_\Sigma, m_\Sigma)$ is both $CAT(1)$ and $RCD(n-2,n-1)$ and $m_\Sigma=\mathcal H_{n-1}$. 

{\it Claim:} $\Sigma$ is isometric to $\SS^{n-1}$ or to a convex subset of  $\SS^{n-1}$ with nonempty interior and nonempty boundary.

Indeed, if $n=2$ then  $(\Sigma,d_\Sigma,m_\Sigma)$ is a noncollapsed compact $RCD(0,1)$ space which is also $CAT(1)$.
Hence  by ~\cite{Kit-Lak} it's isometric to either a circle $\SS^1_R$ of some radius $R>0$  or to a closed interval $I$. 

Suppose $\Sigma\cong\SS^1_R$. Since $\Sigma$ is $\CAT(1)$ we have that $R\ge 1$. On the other hand, since $C(\Sigma)$ is  $RCD(0,2)$ we must have $R\le 1$. Hence $R=1$ and $\Sigma\cong \SS^1$.
 
 If $\Sigma=I$ then length of $I$ is at most $\pi$ since otherwise $C(\Sigma)$ is not $RCD(0,2)$ by the splitting theorem.
 
 If $n>2$ then the Claim follows by the induction assumption.
This is the only place in the proof of the induction step  where the induction assumption is used and where the induction step differs from the proof of the base of induction $n=2$.

Since a proper convex subset of $\SS^{n-1}$ is contained in a hemisphere
the above Claim implies that we have the following gap phenomena:
\begin{equation}
\begin{gathered}\label{gap}
\text{A point $x\in X$ is either regular, in which case $\theta(x)=1$,}\\
\text{ or $x\in X$ is singular, in which case $\theta(x)\le \textstyle{\frac 1 2}$ and $\partial \Sigma\ne \emptyset$}.
\end{gathered}
\end{equation}

Next we prove the following lemma.
\begin{lemma}\label{lem: reg-convex}
The set of regular points $X_{reg}$ is  open and convex in $X$.
\end{lemma}

\noindent

Since $\theta$ is lower semicontinuous, property~\eqref{gap} immediately implies that $X_{reg}$ is open. It remains to verify that it is convex.

Let $f_r(x)=\theta_{r}(x)^{1/n}$ and $f(x)=\theta(x)^{1/n}$

\begin{lemma}\label{dens-semiconcave}
$f(x)$ is  semiconcave on $X$.
\end{lemma}
\begin{proof}
To prove semiconcavity we  need to verify  that there is a constant $C$ such that every point in $X$ has a neighborhood $U$ such that for any constant speed geodesic $\gamma\co [0,1]\to U$  and any $0\le t\le 1$ it holds that
\begin{equation}
f(\gamma(t))\ge (1-t)f(0)+tf(1)-\frac{C}{2}t(1-t)d(\gamma(0),\gamma(1))^2
\end{equation}

To simplify the exposition we will only treat the case $t=1/2$, i.e. we will verify that

\begin{equation}\label{den1:concav}
f(\gamma (1/2))\ge \frac{1}{2}f(0)+\frac{1}{2}f(1)-\frac{C}{8}d(\gamma(0),\gamma(1))^2
\end{equation}

The proof below easily adapts to the case of  general $t$.

Let $x=\gamma(0), y=\gamma(1), z=\gamma(1/2)$ and $l=d(x,y)$. 

%Since $\lim_{r\to 0} \frac{(v_{k,n}(r))^{1/n}}{\omega_n^{1/n}r}=1$ we will assume  that $f_r(x)=\frac{m(B_r(x))^{1/n}}{\omega_n^{1/n}r}$. This will make no difference in the proof.

By rescaling we can assume that $X$ is $CAT(1)$ and $RCD(-n,n)$.

We will need the following general lemma.

\begin{lemma}\label{cat(K)-conv}
Let $X$ be a $CAT(\kappa)$ space. Let $\gamma_1,\gamma_2\co [0,1]\to X$ be constant speed geodesics with length of $\gamma_1$ equal to $ l<\pi_\kappa/100$ and suppose $d(\gamma_1(0),\gamma_2(0))\le\delta, d(\gamma_1(1),\gamma_2(1))\le\delta$ with $\delta<l/100$.

Then $d(\gamma_1(1/2),\gamma_2(1/2))\le \delta(1+C(\kappa)l^2)$ for some universal $C(\kappa)\ge 0$.

\end{lemma}
\begin{proof}
It's well known that when $\kappa\le 0$ one can take $C(\kappa)=0$ since in this case $t\mapsto d(\gamma_1(t),\gamma_2(t))$ is  convex. We will therefore restrict our attention to the case $\kappa>0$. By rescaling we can assume that $\kappa=1$.
Let $X$ be $CAT(1)$.

Fix a point $\bar p$ in the  unit round sphere $\mathbb S^2$. Let $0\le t \leq 1$. Consider the "$t$-homothety" map $\phi_{t}^{\bar p}\co B_{\pi/100}( \bar p)\to B_{\pi/100}(\bar  p)$ sending any point $x$ to the point $y$ on the unique geodesic connecting $\bar p$ to $x$ with $d(\bar p,y)=td(\bar p,x)$. A direct Jacobi field computation shows that the Lipschitz constant of $\phi_t^{\bar p}$ at $x$ with $|\bar px|=l$ is $\frac{\sin (t l)}{\sin (l)}=\sigma_{1,1}^t(l)$.

Taylor expanding in $l$ we get:

\[
\sigma_{1,1}^t(l)=\frac{\sin (t l)}{\sin (l)}=\frac{tl-(tl)^3/6+\ldots}{l-(l)^3/6+\ldots}=t(1-t^2l^2/6+\ldots)(1+l^2/6+\ldots)=
\]
\[
=t(1+l^2(1-t^2)/6+\ldots)\le t(1+l^2/3) \textrm{ if } l<\pi/100
\]

which immediately gives that if $d(x,y)<d(\bar p, x)/10$ then $d(\phi_t^{\bar p}(x),\phi_t^{\bar p}(y))\le t (1+d(\bar p, x)^2)d(x,y)$.

The definition of a $CAT(1)$ space immediately gives that the same inequality holds for a similarly defined map $\phi_t^p$ for any $p\in X$.

\begin{equation}\label{lip-est1}
d(\phi_t^{ p}(x),\phi_t^{ p}(y))\le t(1+d(p, x)^2)d(x,y) \textrm{ if } d(p,x)<\pi/100,\, d(x,y)<d(p,x)/100.
\end{equation}

Let $\gamma_1,\gamma_2$ be as in the lemma.  Let $x=\gamma_1(1/2)=\phi_{1/2}^{\gamma_1(0)}(\gamma_1(1))=\phi_{1/2}^{\gamma_1(1)}(\gamma_1(0))$, $y=\gamma_2(1/2)=\phi_{1/2}^{\gamma_2(0)}(\gamma_2(1))=\phi_{1/2}^{\gamma_2(1)}(\gamma_2(0))$. Let $z$ be the midpoint between $\gamma_1(0)$ and $\gamma_2(1)$.

Then by \eqref{lip-est1} we have that $d(z,x)\le \frac 1 2 (1+l^2)d(\gamma_1(1),\gamma_2(1))\le \frac{\delta}{2}(1+2l^2)$ and $d(z,y)\le \frac 1 2 (1+(l+\delta)^2)d(\gamma_1(0),\gamma_2(0))\le \frac{\delta}{2}(1+2l^2)$.

By the triangle inequality this gives that $d(x,y)\le d(x,z)+d(z,y)\le  \frac{\delta}{2}(1+2l^2)+ \frac{\delta}{2}(1+2l^2)=\delta(1+2l^2)$ which finishes the proof of Lemma \ref{cat(K)-conv} with $C(1)=2$.
\end{proof}

We are now ready to continue with the proof of Lemma~\ref{dens-semiconcave}.

Let $A$ be the Minkowski sum $\frac{1}{2}B_r(x)+\frac{1}{2}B_r(y)$. 
By the Brunn-Minkowski inequality \cite{cavallettimondinogap} we have that 
\[
m(A)^{1/n}\ge \sigma_{-n,n}^{1/2}(l+2r)m(B_r(x))^{1/n}+\sigma_{-n,n}^{1/2}(l+2r)m(B_r(y))^{1/n}
\]
\[
=\frac{\sinh \frac{l+2r}{2}}{\sinh(l+2r)}m(B_r(x))^{1/n}+\frac{\sinh \frac{l+2r}{2}}{\sinh(l+2r)}m(B_r(y))^{1/n}
\]
\[
=\frac{1}{2\cosh \frac{l+2r}{2}}m(B_r(x))^{1/n}+\frac{1}{2\cosh \frac{l+2r}{2}}m(B_r(y))^{1/n}
\]
Thus 

\begin{equation}
(1+c_1 l^2)m(A)^{1/n}\ge \cosh \frac{l+2r}{2} m(A)^{1/n}\ge \frac{1}{2}m(B_r(x))^{1/n}+ \frac{1}{2}m(B_r(y))^{1/n}
\end{equation}
where the first inequality holds when  $l<1/100$ and $r\ll l$.

By Lemma \ref{cat(K)-conv} we have that $A\subset B_{r(1+c_2l^2)}(z)$. Therefore

\[
(1+c_1 l^2)m(B_{r(1+c_2l^2)}(z))^{1/n}\ge \frac{1}{2}m(B_r(x))^{1/n}+ \frac{1}{2}m(B_r(y))^{1/n}
\]

Dividing by $\omega_n^{1/n} r$ this gives

\[
(1+c_1 l^2)(1+c_2l^2)\frac{m(B_{r(1+c_2l^2)}(z))^{1/n}}{\omega_n^{1/n}r(1+c_2l^2)}\ge \frac{1}{2r\omega_n^{1/n}}m(B_r(x))^{1/n}+  \frac{1}{2r\omega_n^{1/n}}m(B_r(y))^{1/n}
\]
or

\[
(1+c_3l^2)f_{r(1+c_2l^2)}(z)\ge \frac{1}{2}f_r(x)+\frac{1}{2}f_r(y)
\]

Taking the limit as $r\to 0$ this gives

\[
(1+c_3l^2)f(z) \ge \frac{1}{2}f(x)+\frac{1}{2}f(y)
\]
Taking into the account that $0\le f(z)\le 1$ this gives 

\[
f(z)\ge \frac{1}{2}f(x)+\frac{1}{2}f(y)-c_3l^2
\]

which finishes the proof of \eqref{den1:concav} and hence of Lemma \ref{dens-semiconcave}.
\end{proof}
Since a bounded semiconcave function on a closed interval $I\subset \mathbb R$ is continuous on the interior of $I$, the openness of $X_{reg}$ together with the gap property \eqref{gap} immediately imply that 
$f$ must be equal to 1 along any geodesic with endpoints in  $X_{reg}$. Again using   \eqref{gap} we conclude that $X_{reg}$ is convex.

This finishes the proof of Lemma~\ref{lem: reg-convex}.\qed

\noindent
\textbf{2.} 

\begin{comment}
\textit{Claim: If $x_0\in X_{reg}$, then under the assumption of Theorem \ref{cor:4} there exists an open neighborhoud $U$ of $x_0$ that is homeomorphic to an open ball in $\mathbb{R}^n$.}
\noindent
\begin{proof}
Since $(X,d,\mathcal{H}^n)$ satisfies the condition $RCD(\ke,n)$, the map $x\mapsto \mathcal{H}^n(B_{\delta}(x))$ is continuous for every $\delta>0$. 
We fix a small neighborhoud $U$ of $x_0$. Then for every point $x\in U$ the tangent cone $T_xX$ is a metric cone over some metric space $\Sigma_x$. Hence by stability of the conditions $RCD$ and $\CAT$ under Gromov-Hausdorff 
convergence the tangent cone $T_xX$ satisfies $RCD(0,n)$ and $\CAT(0)$. Consequently $(\Sigma_x,\mathcal{H}^{n-1})$ satisfies $RCD(n-2,n-1)$ and $\CAT(1)$. Applying Corollary \ref{cor:4} for $n-1$ yields that 
$\Sigma_x$ is a convex subset of a space form of constant curvature $1$ and therefore it is either the sphere or a convex subset of the upper hemisphere. Then, by continuity of the volume of metric balls the latter is excluded, and
therefore $x\in U$ is a regular point provided the neighborhoud $U$ is small enough. 
Then, by Reifenberg's principle $U$ is homeomorphic to a simply connected open neighborhoud in $\mathbb{R}^n$.
\end{proof}
\end{comment}

\textit{Claim: $X_{reg}$ is a topological $n$-manifold.}
\noindent
\begin{proof}
By Lemma~\ref{lem: reg-convex}  $X_{reg}$ is open. Therefore it is an $n$-manifold by  Reifenberg's principle~\cite[Theorem A.1.1]{cheegercoldingI}.
\end{proof}

\noindent
\textbf{3.}

Fix an arbitrary $\hat\kappa< \ke-(n-2)\uk$ and let $\nu=\nu(n,K,\kappa,\hat\kappa)$ be provided by the Lemma~\ref{cot-ineq}.

We pick $x_0\in X_{reg}$ and a positive $\epsilon<\min \{\nu/2, \pi_{\uk}/2\}$ such that $\overline{B}_{\epsilon}(x_0)=:Y$ is contained in $X_{reg}$.

Then $Y$ is geodesically convex, uniquely geodesic, $(Y,d_Y,\m|_Y)$ satisfies $RCD(\ke,n)$ and $(Y,d_Y)$ is $\CAT(\uk)$.

Let $y\in B_{\epsilon}(x_0)$ and consider $d_y:Y\rightarrow [0,\infty)$. 
We pick any point $z\in B_{\epsilon}(x_0)\backslash \{y\}$  and $0<\delta< \min(\frac{\epsilon-d(x_0,z)}{2}, \frac{d(z,y)}{2})$. Then $\overline{B_{\delta}(z)}$ is convex, it is contained in $B_{\epsilon}(x_0)$   and   $y\notin \overline{B_{\delta}(z)}$.
% and $B_{\delta}(z)\cap \left(\left\{y\right\}\cup B_{\epsilon}(x_0)\right)=\emptyset$. 
We also can pick a 
cutoff function $\chi\in \mathbb{D}^Y_{\infty}$ (see Lemma \ref{lem:cutoff}) that {$\Delta\chi\in L^{\infty}(\mathcal{H}^n)$, }$\supp\chi\subset B_{\delta}(z)$ and $\chi|_{B_{\delta/2}(z)}=1$.

Then,  by Corollary \ref{cor:important}  $\chi\cdot d_y\in D_{L^{\infty}}(\Delta)\cap L^{\infty}(\mathcal{H}^n)\cap \Lip(X)$ and therefore $\chi\cdot d_y\in H^{2,2}(X)$ by Remark \ref{rem:H22}. 

In particular,
%$\chi\cdot d_y\in D(\Delta)\subset H^{2,2}(X)$,\footnote{\color{red} this needs to be justified } and 
$\chi\cdot d_y$ induces an element $\nabla (\chi\cdot d_y)=:u\in L^2(TY)$ such that $|u|=1\neq 0$ $\m$-a.e. on $B_{\delta/2}(z)$. 
We consider the submodule $\mathcal{N}\subset L^2(TY)$ that is generated by $u$. The orthogonal submodule is defined as
\begin{align*}
\mathcal{N}^{\perp}=\left\{v\in L^2(TY):\langle v,u\rangle = 0 \ \m\mbox{-a.e.}\right\}.
\end{align*}
It is not hard to check that $\mathcal{N}^{\perp}$ is an $L^{\infty}$-premodule in the sense of Definition 1.2.1 in 
\cite{giglinonsmooth}. $\mathcal{N}$ and $\mathcal{N}^{\perp}$ are Hilbert spaces and hence $L^2$-normed $L^{\infty}$-modules 
(compare with Proposition 1.2.21 in \cite{giglinonsmooth}). Moreover
$\mathcal{N}^{\perp}$ is the orthogonal complement of $\mathcal{N}$
in the sense of Hilbert spaces, and $\mathcal{N}\oplus\mathcal{N}^{\perp}=L^2(TY)$. 

According to Proposition \ref{prop:dim} $\mathcal{N}^{\perp}$ yields a partition $\left\{B_k\right\}_{k\in \mathbb{N}}$ of $Y$ such that 
the local dimension (in the sense of Definition \ref{def:basis}) of $\mathcal{N}^{\perp}|_{B_k}$ is $k\in\mathbb{N}$. Note that $\m(B_{\infty})=0$ since the local dimension of $L^2(TY)$ is finite on $X$, and therefore $\mathcal{N}^{\perp}$ is finitely generated $\m$-a.e.\ as well.
Hence, if $k\in \mathbb{N}$ such that $\m(B_k)>0$, any subset $B$ of $B_k$ with finite measure admits a unit 
orthogonal module basis $v_1,\dots, v_k$, and for any $v\in \mathcal{N}^{\perp}$ we have that $v1_B$ is the $L^2$-limit of finite $L^{\infty}$-linear combinations in $\mathcal{N}^{\perp}$ of the form
\begin{align*}
w=\sum_{j=1}^k f_{j}v_j \ \mbox{ for }f_{j}\in L^{\infty}(\m),\ j=1,\dots,k.
\end{align*}

At the same time we have that any $w\in L^{2}(TY)$ can be written as
a sum $\alpha u + \beta v$ with $\alpha, \beta\in\mathbb{R}$ and $v\in \mathcal{N}^{\perp}$. Hence $u,v_1,\dots,v_k$ generates $L^2(TY)|_B=(\mathcal{N}\oplus\mathcal{N}^{\perp})|_{B}$ in the sense of modules. 
Moreover, since $\mathcal{N}$ and $\mathcal{N}^{\perp}$ are orthogonal w.r.t. the pointwise inner
product $\langle\cdot,\cdot\rangle$ and since $|u|=1$ $\m$-a.e. on $B_{\delta/2}(z)$, it is easy to check that $u,v_1,\dots,v_k$ are linearly independent on $B_{\delta/2}(z)\cap B$ (again in the sense of Definition \ref{def:basis}), and hence form
a module basis of $L^2(TY)$ on $B_{\delta/2}(z)\cap B$. Since the local dimension of $L^2(TY)$ is $n$ this implies $k=n-1$  whenever $B_{\delta/2}(z)\cap B\neq \emptyset$, and $u=:E_1, v_2=:E_2,\dots, v_{n-1}=:E_n$ is a unit orthogonal basis of $L^2(TY)$ on $B_{\delta/2}(z)\cap B$.
In particular, for the decomposition $\left\{B_k\right\}_{k\in\mathbb{N}}$ we have $\m(B_k)=0$ if $k\neq n-1$, and we can choose $B$ as the ball $B_{\delta/2}(z)$ itself.
\\
\

\textbf{4.} Again from Corollary \ref{cor:important} and  we have that $\md_{\ke/(n-1)}(\chi \cdot d_y)\in D_{L^{\infty}}(\Delta)\cap L^{\infty}(\mathcal{H}^n)\cap \Lip(X)$. Then
the precise estimate in the Laplace operator comparison statement (Theorem \ref{th:upperbound}) for $\md_{\ke/(n-1)}d_y$ on $Y$ yields
\begin{align*}
\Delta(\md_{\ke/(n-1)}(\chi\cdot d_y))|_{B_{\delta/2}(z)}={\bf\Delta}\md_{\ke/(n-1)}(d_y)|_{B_{\delta/2}(z)}\leq \frac{\ke n}{n-1}\md_{\ke/(n-1)}(d_y)|_{B_{\delta/2}(z)}\ \m\mbox{-a.e.}
\end{align*}
where we used the locality property of ${\bf\Delta}$. Applying the chain rule for the Laplacian yields
\begin{align}\label{est:one}
\Delta(\chi\cdot d_y)|_{B_{\delta/2}(z)}\leq (n-1)\cot_{K/(n-1)}(d_y)|_{B_{\delta/2}(z)}.
\end{align}

On the other hand the condition $\CAT(\uk)$ on $Y$ implies the following. First, by continuity reasons for any $\vartheta>0$  there exists $\eta>0$ as above such that for any $\eta$-ball $|\md_{\uk}d_y(x)-\md_{\uk}d_y(z)|\leq \vartheta$ for $x\in B_{\eta}(z)$.
Therefore, if we choose $\delta/2\leq \eta$
\begin{align*}
(\md_{\uk}(\chi d_y)\circ\gamma)''\geq 1 - \md_{\uk}d(z,y) - \vartheta=:\lambda(\uk,\vartheta, z) \ \mbox{ for any unit speed geodesic }\gamma\mbox{ in }B_{\delta/2}(z). 
\end{align*}
Hence, $\md_{\uk}(\chi d_y)$ is $\lambda(\uk,\vartheta,z)$-convex on $B_{\delta/2}(z)$.  Now, since $\md_{\kappa}(\chi d_y)\in D_{L^{\infty}}(\Delta)\cap L^{\infty}(\mathcal{H}^n)\cap \Lip(X)$ - and again in particular $\md_{\kappa}(\chi\cdot d_z)\in H^{2,2}(X)$ - , 
we can apply Theorem \ref{th:important} where $f=\md_{\kappa}(\chi d_y)$, $X=Y$ and $Z=B_{\delta/2}(z)$. 
We obtain for $V\in L^2(TY)$
\begin{align}\label{esti}
\Hess(\md_{\uk}(\chi d_y))(V,V)\geq \lambda(\uk,\vartheta,z) |V|^2\geq (1-\md_{\uk}d_y -2\vartheta)|V|^2\ \m\mbox{-a.e. on }B_{\delta/2}(z) .
\end{align}
Now, we also can cover $B_{\delta/2}(z)$ with $\eta$-balls as above. Since any of these balls is geodesically convex by Remark \ref{rem:convex} and our choice of $\vartheta>0$,
the estimate (\ref{esti}) holds with $\delta/2$ replaced by $\eta$. Hence, the estimate holds $\m$-a.e. on $B_{\delta/2}(z)$ with arbitrary small $\vartheta>0$, so we actually have
\begin{align}\label{esti-2}
\Hess(\md_{\uk}(\chi d_y))(V,V)+\md_{\uk}d_y|V|^2\geq |V|^2\ \m\mbox{-a.e. on }B_{\delta/2}(z) .
\end{align}
Applying another time the chain rule for the Hessian  (Proposition \ref{prop:chainrule}) in particular yields
\begin{align}\label{est:two}
\Hess (\chi d_y)(\nabla \chi d_y,\nabla \chi d_y)|_{B_{\delta/2}(z)}=0\ \ \&\ \ \Hess (\chi d_y)(E_i,E_i)|_{B_{\delta/2}(z)}\geq  \cot_{\uk}d_y|_{B_{\delta/2}(z)} \mbox{ for } \ i=2,\dots,n
\end{align}
where the first identity follows -- for instance -- from the claim below, and the second one follows from (\ref{esti-2}) after applying the chain rule for the Hessian.  

Then, since $\chi\cdot d_y\in D(\Delta)$, Corollary \ref{cor:trace} and the fact that we have the unit orthogonal module basis $(E_i)_{i=1,\dots,n}$  from {\bf 3.}  together with (\ref{est:one}) and (\ref{est:two}) immediately gives us
\begin{align*}
(n-1)\cot_{\ke/(n-1)}d_y \geq {\Delta} (\chi\cdot d_y)=\sum_{i=2}^n \Hess (\chi d_y)(E_i,E_i)\geq (n-1) \cot_{\uk}d_y \ \m\mbox{-a.e. on }B_{\delta/2}(z).
\end{align*}
Since $k\mapsto \cot_k$ is monotone decreasing this implies that $\kappa(n-1)\geq \ke$ and
\begin{align}\label{eq:hess-est1}
\Hess (\chi d_y)(E_i,E_i)\leq (n-1)\cot_{K/(n-1)}d_y - (n-2) \cot_{\uk}d_y \ \m\mbox{-a.e. on }B_{\delta/2}(z).
\end{align}

Therefore, by Lemma~\ref{cot-ineq}
\begin{align}
\label{ineq:some}
\Hess(\chi d_y)(E_i,E_i)\leq \cot_{\hat{\kappa}}d_y \ \ \mbox{on } B_{\delta/2}(z)  \mbox{ for } \ i=2,\dots,n.
\end{align}
\noindent
\textbf{5.} 
\textit{Claim:} Let $h,\phi_k\in H^{2,2}(X)$ with $|\nabla h|=1$ $\m$-a.e. on $B_{\delta/2}(z)$ , and $V=\sum_{k=1}^m f_k\nabla \phi_k$, then
\begin{align*}
\Hess h(\nabla h,V)|_{B_{\delta/2}(z)}=0\ \m\mbox{-a.e.}\ .
\end{align*}
\textit{Proof of the claim.}
Since $h,\phi_k\in H^{2,2}(X)$, $k=1,\dots,m$, using Proposition \ref{prop:nuetzlich} we can compute
\begin{align*}
&2\Hess h(\nabla h,V)=2\sum_{k=1}^m f_k\Hess h(\nabla h,\nabla \phi_k)\\
&\hspace{1.3cm}= \sum_{k=1}^m f_k\left(\langle h,\nabla \langle \nabla\phi_k,\nabla h\rangle + \langle \nabla \phi_k,\nabla |\nabla h|^2\rangle-\langle h,\nabla \langle \nabla h,\nabla \phi_k\rangle \right)= \langle V,\nabla |\nabla h|^2\rangle\ \m\mbox{-a.e.}\ .
\end{align*}
By locality of $|\nabla \cdot |$ this yields $2\Hess h(\nabla h,V)|_{B_{\delta/2}(z)}=\langle X,\nabla 1\rangle|_{B_{\delta/2}(z)}=0$ $\m$-a.e. .\qed
\\
\\
The claim in particular applies for $\chi\cdot d_y=h$ and $V=\sum_{i=1}^{l}f_i\nabla \phi_k$ with $f_k\in L^{\infty}(\m)$ and $\phi_k\in \mathbb{D}_{\infty}^Y$. 
Now any $E_i$, $i=1,\dots,n$ can be approximated by vector fields  of this form. Therefore, by $L^2(TY)-L^0$-continuity of the Hessian
\begin{align}\label{ineq:some1}\Hess(\chi d_y)(\nabla(\chi d_y),E_i)=\Hess (\chi d_y)(E_1,E_i)=0\ \m\mbox{-a.e.}\ \mbox{on} \ B_{\delta/2}(z).\end{align}
\\

\textit{Claim:} Let $V\in \mathcal{N}^{\perp}$ be arbitrary. Then
\begin{align}\label{ineq:some4}
\Hess (\chi d_y)(V,V)\leq \cot_{\hat{\kappa}}d_y |V|^2\ \m\mbox{-a.e. on }B_{\delta/2}(z).
\end{align}
\textit{Proof of the claim.} Let $E_2,\dots,E_n$ be the unit orthogonal basis of $\mathcal{N}^{\perp}$ on $B_{\delta/2}(z)=:B$ as in \textbf{3.}. 
By definition of a module basis {and Remark~\ref{rem:generator},} for every $k\in\mathbb{N}$ we find functions $f^k_{i}\in L^{\infty}$, $i=2,\dots,n$, such that 
\begin{align*}
\left\|V-V^{k}\right\|_{L^2(TY)|_B}<\frac{1}{k} \ \mbox{ where }V^{k}\in L^2(TY)|_B\mbox{ such that } \ V^{k}=\sum_{i=2}^n f^k_{i}E_i.
\end{align*}
We can approximate every $f^k_{i}$ in $L^2(\mathcal{H}^n)$ by step functions that take only finitely many values. Therefore, it is sufficient to assume that $f^k_{i}=\alpha^k_{i}\in \mathbb{R}$. 
Moreover, since $B$ has finite $\mathcal{H}^n$-measure we can assume that $|\alpha_{i}^k|\in [\frac{1}{nk},\frac{k}{n}]$. Hence, it follows that $|V^k|\in [\frac{1}{k},k]$ for every $k\in \mathbb{N}$.

We can define $W^k=|V^k|^{-1}V^k\in L^2(TY)|_B$. Then $W^k$ satisfies $|W^k|=1$ $\mathcal{H}^n$-a.e. on $B$.  Now, we can choose the unit orthogonal basis $E_1,\dots,E_{n}$ of $L^2(TY)$ on $B_{\delta/2}(z)$ such that $E_2=W^k$. 
This is achieved in the same way as we were able to choose $E_1=\nabla(\chi\cdot d_y)$ in step \textbf{3.} since $|W^k|=1$ $\mathcal{H}^k$-a.e.\ .
Hence, we obtain (\ref{ineq:some4}) first for $W^k$ and by $L^{\infty}$-homogeneity of $\Hess(\chi d_y)$ also for $V^k$. Finally, since $V^k$ approximates $V$ in $L^2$-sense, and since $\Hess(\chi d_y)$ is a $L^2-L^0$-continuous bilinear form on $L^2(TY)$, 
we obtain the desired estimate for $V\in \mathcal{N}^{\perp}$. \qed
%The first inequality holds since $\epsilon\in (0,\pi_{\uk}/2)$, and therefore $\cot_{\uk}$ in the range of $d_y$.
%Therefore
%\begin{align*}
%\Hess (\chi d_y)(E_i,E_j)\leq \cot_{\hat{\kappa}}d_y\ \m\mbox{-a.e.}\mbox{ on }B_{\delta/2}(z)
%\end{align*}
%From (\ref{ineq:some}) we therefore get
%\begin{align}
%\Hess (\chi d_y)(V,V)\leq \cot_{\hat{\kappa}}d_y|V|^2 \ \m\mbox{-a.e.}\mbox{ on }B_{\delta/2}(z) \mbox{ for } V=\sum_{i=2}^n f_i E_i \mbox{ and }f_i\in L^{\infty}(\mathcal{H}^n).
%\end{align}
\\

Hence, again applying the chain rule together with (\ref{ineq:some1}) and (\ref{ineq:some4}) yields
\begin{align*}
\Hess \md_{\hat{\kappa}}(\chi d_y)(V,V)+ \hat{\kappa}\md_{\hat{\kappa}}(d_y)|V|^2\leq |V|^2\ \m\mbox{-a.e. on } B_{\delta/2}(z) \ \forall V\in L^2(TY).
\end{align*}
%
%\begin{align*}
%\Hess \frac{1}{\hat{\kappa}}\cos_{\hat{\kappa}}(\chi d_y)(W,W)\geq -\cos_{\hat{\kappa}}d_z|W|^2 \ \m\mbox{-a.e.}\mbox{ on }B_{\delta/2}(z)\ \mbox{ for } W\in L^2(TY).
%\end{align*}
%Hence $\Hess \frac{1}{\hat{\kappa}}\cos_{\hat{\kappa}}(\chi d_z)\geq -\cos_{\hat{\kappa}}d_z$ on $B_{\delta/2}(z)$ $\m$-a.e. .
\\
\\
%Now, at this point we have to emphasize that is not possible to apply Theorem \ref{th:important} again. This is simply due to fact that Theorem \ref{th:important} only can handle functions 
%that are $(\kappa,\beta)$-convex if $\beta\in (0,\infty)$. A negative $\beta$ is excluded. Therefore, w
{\textbf 6.} We consider another time the second variation formula (Theorem \ref{th:secondvariation}). It yields that the function 
$\mathcal{F}(\mu)=\int \md_{\hat{\kappa}}(\chi d_y)d\mu$ for $\mu\in \mathcal{P}^2(Y,\m)$ satisfies 
\begin{align}\label{ineq:difff}
\int_0^1\mathcal{F}(\mu_t)\phi''(t)dt+\hat{\kappa} \int_0^1\int \md_{\hat{\kappa}}(\chi d_y)|\nabla \psi_t|^2d\mu_t \phi(t)dt\leq \int_0^1 \int |\nabla \psi_t|^2d\mu_t \phi(t) dt,\ \ \phi\in C^2((0,1))
\end{align}
for Wasserstein geodesics $(\mu_t)_{t\in[0,1]}$ supported in $B_{\delta/2}(z)$, and $\psi_t\equiv \phi_t$ as in Theorem \ref{th:secondvariation}. Note that
\begin{align*}
\int |\nabla \psi_t|^2 d\mu_t = \frac{1}{(s-t)^2}\int |\nabla \psi|^2 d\mu_t=\frac{1}{(s-t)^2}W_2(\mu_t,\mu_s)^2=W_2(\mu_0,\mu_1)^2
\end{align*}
where $\psi$ is a Kantorovich potential between $\mu_t$ and $\mu_s$ for some $s\in [0,1]$, $s\neq t$. Furthermore, by the metric Brenier theorem 
$|\nabla \psi_t|^2(\gamma_t)=\frac{1}{(s-t)^2}|\nabla \psi|^2(\gamma_t)=d(\gamma_t,\gamma_s)^2$ for $\Pi$-a.e. every geodesic where $\Pi$ is the optimal 
dynamical plan associated to the geodesic $(\mu_t)_{t\in [0,1]}$. Hence
\begin{align*}
\int \md_{\hat{\kappa}}(\chi d_y)|\nabla \psi_t|^2 d\mu_t= \int \md_{\hat{\kappa}}(\chi d_y)(\gamma_t)d(\gamma_t,\gamma_s)^2 d\Pi(\gamma).
\end{align*}
Now, we choose points $x_0,x_1\in B_{\delta/2}(z)$, and $\mu^k_0,\mu^k_1\in \mathcal{P}^2(X,\mathcal{H}^n)$ with $\supp\mu^k_i\subset B_{\delta/2}(z)$ and $\mu^k_i\rightarrow \delta_{x_i}$ weakly $k\rightarrow \infty$.
Then, by compactness of $\overline{B_{\delta/2}(z)}$ the Wasserstein geodesic $(\mu_t^k)_{t\in [0,1]}$ between $\mu^k_0$ and $\mu^k_1$ converges to the Wasserstein geodesic $(\nu_t)_{t\in [0,1]}$ between $\delta_{x_0}$ and $\delta_{x_1}$. 
By uniqueness of geodesics between points in $B_{\delta/2}(z)$ -- because of the $\CAT$-condition -- , $\nu_t=\delta_{\gamma(t)}$ where $\gamma$ is the geodesic between $x_0$ and $x_1$.
Moreover, since $\md_{\hat{\kappa}}(\chi d_y)$ is continuous and bounded, we obtain from the definition of weak convergence
\begin{align*}
\mathcal{F}(\mu_t^k)\rightarrow \mathcal{F}(\nu_t)=\int \md_{\hat{\kappa}}(\chi d_y) d\nu_t = \md_{\hat{\kappa}}d_y(\gamma(t)), \ \forall t\in (0,1).
\end{align*}
Similarly, after taking another subsequence the associated dynamical plans $\Pi^k$ weakly converge as well, and again by uniqueness of geodesics they converge to the measure $\delta_{\gamma}$ that is supported on 
the single geodesic between $x_0$ and $x_1$. Since $\gamma\mapsto \md_{\hat{\kappa}}(\chi d_y)(\gamma_t)d(\gamma_t,\gamma_s)^2$ is a continuous and bounded function on the space of geodesics, we get by the $CAT$-condition that
\begin{align*}
\frac{1}{(s-t)^2}\int \md_{\hat{\kappa}}(\chi d_y)(\gamma_t)d(\gamma_t,\gamma_s)^2 d\Pi^k(\gamma)\rightarrow \md_{\hat{\kappa}}(d_y(\gamma_t))d(\gamma_0,\gamma_1)^2, \ \forall t\in (0,1).
\end{align*}

Hence, by the dominated convergence theorem we obtain the differential inequality (\ref{ineq:difff}) for $\md_{\hat{\kappa}}d_y\circ\gamma$ along geodesics $\gamma$ in ${B_{\delta/2}(z)}$.

Since $z\in B_{\epsilon}(x_0)\backslash \{y\}$ was arbitrary we have that $ \md_{\hat{\kappa}}(d_y)$ satisfies ${\eqref{CBB-CBA-mdk}}_{(CBB)}$ for any $\gamma\subset B_{\epsilon}(x_0)\backslash \{y\}$.

%By Remark \ref{mdk-on-X-miuns y}  
If $\gamma$ passes through $y$ then the same property holds for trivial reasons. As this holds for any $y\in  B_{\epsilon}(x_0)$  and  $B_{\epsilon}(x_0)$ is convex we get that 
 $CBB(\hat\uk)$ property  ${\eqref{point-on-a-side-comp}}_{(CBB)}$ holds for all triangles with vertices in $B_{\epsilon}(x_0)$.

 Since $x_0\in X_{reg}$ was arbitrary and $X_{reg}$ is convex in $X$, by Petrunin's globalization theorem \cite{petr-globalization}  (cf. \cite{Li-globalization}) it follows that $X$ satisfies $CBB(\hat\uk)$.

Since this holds for arbitrary $\hat \kappa<  \ke-(n-2)\uk$ we conclude that $X$ satisfies $CBB(\ke-(n-2)\uk)$.

This proves  part \eqref{low-curv-bound} of Theorem~\ref{cor:4}.

Now part \eqref{equal-case} follows by Lemma \ref{CBB-CBA-equal}.

This concludes the proof of the induction step and hence of Theorem~\ref{cor:4}.
\qed
%\begin{corollary}
%Let $(X,d,\m)$ be a metric measure space satisfying $RCD(k(n-1),n)$ with $\m=\mathcal{H}^n$, and such that $(X,d)$ is $\CAT_{loc}(k)$ for $k\geq 0$. 
%Then $X$ is a convex subset of a Riemannian space form of constant curvature $k$.
%\end{corollary}
%\begin{proof}
%If $k=0$, we obtain that $(X,d,\m)$ has curvature bounded from below by $0$, and therefore Theorem \ref{th:nikolaev} yields the statement. If $k>0$, we can reduce to the previous case by considering the euclidean cone over $(X,d)$.
%\end{proof}
%\include{cd-to-rcd-new}

\section{CD+CAT to RCD+CAT}\label{CD+CAT to RCD+CAT}
%The goal of this section is to provide various sufficient conditions that guarantee that a space satisfying  curvature dimension and $\CAT$ conditions is infinitesimally Hilbertian.

In this section we study  metric measure spaces $(X,d,m)$ satisfying
\begin{equation}
\begin{gathered}\label{eq:cd+cat}
\mbox{$(X,d,m)$ is $\CAT(\uk)$ and satisfies any of the conditions $CD(K,N)$, $CD^*(K,N)$ }\\
\mbox{or  $CD^e(K,N)$ for $1\le N<\infty$, $K,\uk<\infty$.}
\end{gathered}
\end{equation}

{
\begin{remark}\label{rem:equiv}
Proposition \ref{prop:nonbra} at the end of this section shows that a space $X$ satisfying  (\ref{eq:cd+cat}) is non-branching, and 
hence it is essentially non-branching. For essentially non-branching spaces conditions $CD(K,N)$, $CD^*(K,N)$ and $CD^e(K,N)$ are known to be equivalent \cite[Theorem 3.12]{erbarkuwadasturm} and \cite[Theorem 1.1]{cavmil}. Therefore, for $\CAT(\uk)$ spaces conditions  $CD(K,N)$, $CD^*(K,N)$ and $CD^e(K,N)$ with $1\le N<\infty$ are equivalent.
  However, we will not use this fact in the proof of Theorem ~\ref{CD+CAT implies RCD} below.

\end{remark}}

The main goal of this section is to prove the following theorem.
\begin{theorem}\label{CD+CAT implies RCD}
Let $(X,d,\m)$ satisfy \eqref{eq:cd+cat}. Then $X$ is infinitesimally Hilbertian.

In particular, $(X,d,\m)$ satisfies $RCD(K,N)$.
\end{theorem}

 Together with Theorem~\ref{rcd+cat-to-cbb} this  immediately gives Theorem~\ref{main-thm}.
 Let us mention that the proof of Theorem~\ref{CD+CAT implies RCD} can be simplified under various extra regularity assumptions such as requiring  $X$  to have extendible geodesics or to have metric measure tangent cones equal to geodesic tangent cones. Note that it's easy to see that for a space {$(X,d,m)$ satisfying \eqref{eq:cd+cat} for any $p$ and \emph{any} tangent cone $(T_pX,d_p,m_p)$ at $p$ there is a canonical distance preserving embedding $T_p^gX\subset T_pX$. However, it is not a priori clear if this embedding is always onto. 
 
Let us give two instructive examples to keep in mind.

\begin{example}\label{ex1-cat}
Let $X$ be the union of $\R$ with closed intervals of length $2^k$ attached at the point $2^k$ for all integer $k$. It's easy to see that $X$ is $CAT(0)$. Let $p=0$ be the base point. Then any two geodesics starting at $p$  to the right have a common beginning and hence the geodesic tangent space $T^g_pX$ is isometric to $\R$. On the other hand, $(X,p)$ is self similar with respect to multiplication by $2$ and hence the tangent cone $T_pX=\lim_{k\to\infty}(2^kX,p)$ is isometric to $X$. Note that this tangent cone is not a metric cone. 

%Further, let $f\co X\to \R$ be given by the distance to $\R\subset X$. Then $f$ is 1-Lipschitz and convex. It's easy to see that $\Lip f(p)=1$ while $\gnabla f(p)=0$.
\end{example}
%The following example is also instructive
\begin{example}\label{ex2-cat}
Let $\Gamma$ be a binary tree, Let $\eps_n$ be a sequence of positive numbers  such that $R=\sum_n\eps_n<\infty$.
Define the metric on $\Gamma$
by prescribing the  length of any edge from level $n$ to the level $n+1$ to be  $\eps_n$.  Let $X$ be the metric completion of $\Gamma$. Then $X$ is $CAT(0)$ of topological dimension 1. Let $p$ be the root of the tree $\Gamma$. 
Then the cut locus of $p$ coincides with the metric sphere of radius $R$ at $p$. It is a Cantor set and for an appropriately chosen sequence $\eps_n$ it can have arbitrary large Hausdorff dimension.
Furthermore, for any $q\in S_R(p)$ the geodesic tangent space $T_q^gX$ is still a ray and is different from any tangent cone obtained as a blow up limit.
\end{example}

The key step in the proof of  Theorem~\ref{CD+CAT implies RCD} is the following splitting theorem

\begin{proposition}[Splitting theorem]\label{splitting theorem}
Let $(X,d)$ be $CAT(0)$ and $CD(0,N)$ or $CD^e(0,N)$ for some finite $n$.
%$CD(0,n)$ and  and assume $(X,d,m)$ admits a Brunn-Minkowski inequality (\ref{ineq:BM}) as in Theorem \ref{th:stability+BM} for some finite $n$.  
 Suppose $X$ contains a line. 
 Then $(X,d)\cong (Y,d_Y)\times (\R,d_{Eucl})$  for some $CAT(0)$ metric space $(Y,d_Y)$.

%Then $(X,d,m)\cong (Y,d_Y,m_Y)\times (\R,d_{Eucl}, \mathcal H_1)$ where $Y$ is  $CAT(0)$.
%{\color{red} Does this imply that $Y$ is  $CD(0,n-1)$? I think it does even though it's not needed in the sequel.}

\end{proposition}

\begin{proof}
Let $\gamma\co (-\infty,\infty)\to X$ be a line in $X$.   Let $\gamma_\pm$ be the rays  $\gamma_+(t)=\gamma(t)$ and $\gamma_-(t)=\gamma(-t)$ for $t\ge 0$. 
Let $b_\pm(x)=\lim_{t\to\infty}d(x,\gamma_\pm(t))-t$ be the corresponding Busemann functions.  Note that $b_\pm$ are both convex and 1-Lipschitz since they are limits of 1-Lipschitz convex functions.

For any $r>0$ let $f_r(x)=m^{1/N}(B_r(x))$.

%As in the proof of Theorem \ref{sec:RCD+CAT-to-Alexandrov} consider the density functions  $\theta_{n,r}(x)=\frac{m(B_r(x))}{\omega_nr^n}$ where $x\in X$ and $r>0$. Let $f_r(x)=\theta_{n,r}^{1/n}(x)$.
By the same proof as in Lemma~\ref{dens-semiconcave}, for any fixed $r$ the function $f_r(x)$ is \emph{concave} on $X$. We recall the argument which is particularly simple in our case because the lower Ricci and the upper curvature bounds are both zero.

It's well-known that geodesics in $CAT(0)$ spaces satisfy the following "fellow travel" property:

If two constant speed geodesics $\sigma_1,\sigma_2\co [0,1]\to X$ satisfy $d(\sigma_1(0),\sigma_2(0))\le r, d(\sigma_1(1),\sigma_2(1))\le r$ then $d(\sigma_1(t),\sigma_2(t))\le r$ for all $0\le t\le 1$.

This immediately follows from the fact that the function $t\mapsto d(\sigma_1(t),\sigma_2(t))$ is convex in $t$ which is an easy consequence of the $CAT(0)$ condition.

Let $x,y$ be any two points in $X$ and let $\sigma\co [0, 1]\to X$ be a constant speed geodesic from $x$ to $y$.
Let $A=\bar B_r(x), B=\bar B_r(y)$. Let $0\le t\le 1$ and let $C_t=(1-t)A+tB$ be their $t$-Minkowski sum. By the  "fellow travel" property we have that $C_t \subseteq \bar B_r(\sigma(t))$. Also, $C_t$ is clearly closed. By the Brunn-Minkowski inequality ( Theorem \ref{th:stability+BM}) we have that  $m^{1/N}(C_t)\ge (1-t) m^{1/N}(A)+ t m^{1/N}(B)$.  Using that $m(C_t)\le m(\bar B_r(\sigma(t)))$ this gives
$f_r(\sigma(t))\ge (1-t)f_r(x)+tf_r(y)$ i.e. $f_r$ is concave.

Thus we have that for any $r>0$ the map $t\to f_r(\gamma(t))$ is concave and positive on $\R$. This implies that it's constant. Therefore, $m(\bar B_r(\gamma(t)))$ is 
constant in $t$. This means that in the proof of concavity of $f$ along $\gamma(t)$ all inequalities  must be equalities and hence for any $t_1,t_2$ it holds that the $1/2$-Minkowski sum 
of $\bar B_r(\gamma(t_1))$ and $\bar B_r(\gamma(t_2))$ is equal to $\bar B_r(\gamma(\frac{t_1+t_2}{2}))$.  Since the $1/2$-Minkowski sum is
closed, the open complement in $\bar B_r(\gamma(\frac{t_1+t_2}{2}))$ must be empty.

Let $q\in X$ be any point an let $r=d(q,\gamma(0))$. By above, for any $t\ge 0$ there exist $q_t\in B_r(\gamma(t))$ and $q_{-t}\in B_r(\gamma(-t))$ such that $q$ is the midpoint of $[q_t,q_{-t}]$. Moreover, again using the "fellow travel" property we get that the whole geodesic $[q_t,q_{-t}]$ lies in the $r$-neighbourhood of $\gamma$. By letting $t\to\infty$ and passing to the limit along a subsequence we obtain a line $\gamma_q\co (-\infty,\infty)\to X$ such that $\gamma_q(0)=q$ and the whole $\gamma_q$ lies in the $r$-neighbourhood of $\gamma$. By the triangle inequality this implies that $d(\gamma(t), \gamma_q(t))\le 3r$ for any $t\in\R$. By the Flat Strip Theorem  ~\cite[Theorem 2.13]{BH99} this implies that the convex hull of $\gamma\cup \gamma_q$ is isometric to the flat strip $[0,D]\times \R$ for some $D\le r$ with $\gamma$ and $\gamma_q$ corresponding to $\{0\}\times\R$ and $\{D\}\times \R$ respectively.  We will call two lines in $X$ parallel if they bound such flat strip.

The above trivially implies that $b=b_++b_-\equiv 0$ on $\gamma_q$ and since $q$ was arbitrary, $b\equiv 0$ on all of $X$. Since  $b_\pm$ are both convex this implies that they are both affine and hence $\{b_+=c\}$ is convex in $X$ for any real $c$. Further,  because of the flat strip property above it holds that $b_+(\gamma_q(t))=b_+(q)-t$ for any $t\ge 0$. Thus, $\gamma_q(t)$ is an (inverse) gradient curve of $b_+$ and we have a similar property for $b_-$. This easily implies that $\gamma_q$ is unique. That is we claim that for every $q$ there is a unique line through $q$ parallel to $\gamma$. Indeed, by possibly changing $\gamma(t)$ to $\gamma(t+c)$ we can assume that $b_+(q)=b_-(q)=0$. By the above $b_+(\gamma_q(t))=-t$ for any $t\ge 0$. Since $b_+$ is 1-Lipschitz this means that $\gamma_q(t)$ is the closest point in $\{b_+\le -t\}$ to $q$. Since $b_+$ is a convex function, the set $\{b_+\le -t\}$ is convex. In $CAT (0)$ spaces nearest point projections to convex subsets are unique (this is immediate from the definition of $CAT(0)$). Hence $\gamma_q(t)$ is uniquely determined by $q$ and $t\ge 0$. The same works for $t\le 0$ using $b_-$.

 %Let $x=\gamma(t_1), y=\gamma(t_2), z=\gamma(\frac{t_1+t_2}{2})$. The above means that for any $q$ with $d(q,z)=r$ there exist $q_1,q_2$ such that $d(x,q_1)\le r$, $d(y,q_2)\le r$ and $q$ is the midpoint of $[q_1,q_2]$. But convexity of distance between two geodesics implies that this is only possible if  $d(x,q_1)=d(y,q_2)= r$. By the flat quadrilateral theorem ~\cite[Theorem 2.11]{BH99} this implies that there is a flat isometrically embedded parallelogram $(x,q_1,q_2,y)$ in $X$. By fixing  $z=\gamma(0)$ and setting $t_1=t, t_2=-t$ and letting $t\to\infty$ this means that \emph{any} point $q\in X$ lies in an infinite isometrically embedded flat strip $I\times\R\subset X$ containing $\gamma$.
%The "upper" edge of this strip is a line $\gamma_q$ passing through $q$ and parallel (i.e. equidistant) to $\gamma$. 

Everything we've shown for the line $\gamma$ applies to the line $\gamma_q$ as well. In particular any point $q'$ in $X$ is contained in a flat strip containing $\gamma_q$ with edges $\gamma_q$ and $\gamma_{q'}$. This easily implies the metric splitting $X\cong Y\times \R$ with $Y=\{b_{+}=0\}$. $Y$ will obviously be $CAT(0)$.

\begin{comment}
Furthermore, recall that  we've shown that $m(B_r(\gamma_q(t)))$ is constant in $t$ for any $q\in X, r>0$.  Since $m$ is doubling this easily implies that $m$ is invariant with respect to  translations along the $\R$ factor in the splitting $X\cong Y\times \R$.  This in turn easily implies that $m$ splits as  $m_Y\times \mathcal {H}_1|_\R $.

Indeed, without loss of generality we can assume that $Y$ is compact (if not change $Y$ to a closed ball of any size $R$ (which is convex)  by the $CAT (0)$ condition). Let $\pi\co X\cong Y\times \R\to \R$ be the coordinate projection onto the $\R$ factor. Let $\nu=\pi_*(m)$.  By above $\nu$  is translation invariant. Hence it must be a constant multiple of the standard Lebesgue measure on $\R$. Now applying disintegration and using the fact that $m$ is translation invariant we immediately get the splitting of $m$.
\end{comment}
\end{proof}
\begin{remark}
Since $m(B_r(\gamma_q(t)))$ is constant in $t$ for any $q\in X, r>0$ it is easy to see that one gets the splitting of the measure in the above proposition as well, that is $(X,d,m)\cong (Y,d_Y,m_Y)\times (\R,d_{Eucl}, \mathcal H_1)$  for some measure  $m_Y$ on $Y$. We omit the details since we don't need it for the proof of Theorem~\ref{CD+CAT implies RCD} and the measure splitting will follow by Gigli's splitting theorem ~\cite{giglisplitting} anyway once Theorem~\ref{CD+CAT implies RCD} is proved. 
\end{remark}
\begin{comment}
{\color{red}  I think because of the metric-measure splitting $Y$ also has to be  $CD(0,n-1)$  although we don't need it. Given any two probability measures $\mu_1,\mu_2$ on $Y$ absolutely continuous with respect to $m_Y$ take $\mu_{1,r}\mu_1\times \delta_r$ and $\mu_{2,r}=\mu_2\times \delta_r$
where $\delta_r$ is the uniform measure on $[-r,r]$. Let $\mu_{t,r}$ be a Wasserstein geodesic between them along which the $n-$entropy is convex. Then as $r\to 0$ $\mu_{t,r}\to\mu_t$  - a geodesic between $\mu_1$ and $\mu_2$. Lower semicontinuity of entropy plus Jensen's inequality should imply that $(n-1)$-entropy is convex along $\mu_t$. some further argument is needed here. This statement is not needed for the proof of the main result so I'll skip it for now...}

As an immediate corollary of the splitting theorem we obtain
\begin{corollary}\label{cor:R^n-RCD}
Let $(\R^n,d_{Eucl},m)$ be $CD^*(0,N)$ for some $n\le N<\infty$.
Then $m=c\cdot \mathcal H_n$ for some $c>0$.
\end{corollary}

{\color{red} Is this Corollary already known? I assumed that it was as it follows from ~\cite{aststability} for example but I don't know a reference. If it's known we can remove the corollary and just give  a reference.}
\end{comment}
\begin{proposition}\label{prop:inf-hilb}
Let $(X,d,m)$ satisfy any of the conditions $CD(K,N)$, $CD^*(K,N)$ or $CD^e(K,N)$ with $N<\infty$.
Then $X$ is infinitesimally Hilbertian if and only if for almost all points $p\in X$ some tangent cone $T_pX$ is isometric to $\R^k$ for some $k\le N$.
\end{proposition}
\begin{proof}
The "only if" direction is well-known and follows from ~\cite{gmr} and Remark \ref{rem-equiv-rcd}. 

We observe that the "if" direction easily follows from Cheeger's generalization of Rademacher's theorem to doubling metric-measure spaces satisfying the Poincar\'e inequality~\cite{cheegerlipschitz}.

Indeed, suppose $(X,d,m)$ satisfies the assumption of the theorem and for almost every $p\in X$ some tangent cone $T_pX$ is Euclidean as a metric space.

Let $f$ be a Lipschitz function on $X$. Recall that by Theorem~\ref{th:rajala}, if $X$ satisfies any of the conditions $CD(K,N), CD(K,N)^e$  or $CD^*(K,N)$, it admits
a weak type 1-1 Poincar\'e inequality and hence  by \cite[Theorem 5.1]{cheegerlipschitz} it holds that $\Lip f =|\nabla f|$ a.e. on $X$.

Further, by ~\cite[Theorem 10.2]{cheegerlipschitz} there is a set of full measure $B_f$ such that for every $p\in B_f$ and every tangent cone $T_pX$ the differential $df_p\co T_pX\to \R$ (which always exists after possibly passing to a rescaling subsequence) is generalized linear (see ~\cite{cheegerlipschitz} for the definition). By above we can also assume that $\Lip f(p) =|\nabla f(p)|$ for any $p\in B_f$.

%By Corollary ~\ref{cor:R^n-RCD} if $T_pX\cong \R^k$ then the background measure on  $T_pX$ is the standard Lebesgue measure. 
By~\cite[Theorem 8.1]{cheegerlipschitz}  if  for $p\in B_f$ it holds that $T_pX\cong \R^k$ as a metric space then (irrespective of the limit measure on $T_pX$)  $df_p\co\R^k\to\R$ is linear in the ordinary sense and $\Lip df_p=\Lip f(p)$.  Given two  Lipschitz functions $f,g$ on $X$, by passing to a subsequence we see that the same works simultaneously for both $df_p, dg_p$ for any $p\in B_f\cap B_g$ which is still a set of full measure in $X$.  Using \cite[Theorem 5.1]{cheegerlipschitz} again we can further assume that $\Lip (f\pm g) =|\nabla (f\pm g)|$ everywhere on $B_f\cap B_g$.

Since $\Lip$ satisfies the parallelogram rule on the set of linear functions on $\R^k$ and $d(f\pm g)_p=df_p\pm dg_p$ this gives

\[
|\nabla(f+g)(p)|^2+|\nabla (f-g)(p)|^2=2|\nabla f(p)|^2+2|\nabla g(p)|^2
\]
for any $p\in B_f\cap B_g$. Therefore the parallelogram rule holds for the Cheeger energies of $f$ and $g$:

\begin{equation}\label{parallelogram-grad}
\int_X|\nabla(f+g)|^2+|\nabla (f-g)|^2dm=\int_X2|\nabla f|^2+2|\nabla g|^2dm
\end{equation}
Since   Lipschitz functions are is dense in  $W^{1,2}(X)$ this implies that \eqref{parallelogram-grad} holds for all $f,g\in W^{1,2}(X)$. This means that $X$ is infinitesimally Hilbertian and hence $RCD(K,N)$ by Remark~\ref{rem-equiv-rcd}. This finishes the proof of Proposition~\ref{prop:inf-hilb}.

\end{proof}
The above Proposition shows that for $CD(K,N)$-spaces ($CD^*(K,N)$-space, $CD^e(K,N)$-spaces, respectively) with finite $N$ "analytic" infinitesimal Hilbertianness in the sense of the original definition is equivalent to the "geometric" infinitesimal Hilbertianness ( i.e. requiring that tangent spaces almost everywhere be Euclidean).

We are now ready to finish the proof of Theorem~\ref{CD+CAT implies RCD}.

\begin{proof}[Proof of Theorem~\ref{CD+CAT implies RCD}]
{First note that by stability of each condition $CD$, $CD^*$, $CD^e$ and $CAT$ under measured Gromov-Hausdorff and Gromov-Hausdorff convergence respectively, it follows that tangent cones satisfy $CD(0,N)$, $CD^*(0,N)$, $CD^e(0,N)$ and $CAT(0)$ respectively. }
%and in each case we obtain the Brunn-Minkowski inequality (\ref{ineq:BM}) in Theorem \ref{th:stability+BM}.

Since $m$ is locally doubling, By \cite[Theorem 3.2]{gmr} there is a set $A\subset X$ of full measure such that for \emph{every} point $p\in A$ for \emph{any} tangent cone $(T_pX,d_p,m_p)$ and \emph{any} point $y\in T_pX$  any tangent cone $(T_y(T_pX), d_y,m_y)$ is a tangent cone at $p$. Let $p\in A$. Let $k$ be the largest integer such that \emph{some} tangent cone  $T_pX$ splits isometrically as $\R^k\times Y$. Clearly $k\le N$. We claim that $Y=\{pt\}$. If not then take a point $y\in Y$ which is a midpoint on some non-constant geodesic segment. Then $T_{(0,y)}(\R^k\times Y)\cong \R^k\times T_yY$ contains a line $l$ contained in $\{0\}\times T_yY$.
Moreover, since any line parallel to $l$ is equidistant from $l$ it easily follows that a line parallel to $l$ and passing through a point in  $\{0\}\times T_yY$ is entirely contained in  $\{0\}\times T_yY$. The splitting theorem then implies that $\{0\}\times T_yY$ is isometric to $\R\times Z$ for some metric space $Z$ and hence $T_{(0,y)}(\R^k\times Y)\cong \R^k\times T_yY\cong \R^{k+1}\times Z$ 
 But it's a tangent cone at $p$ which contradicts the maximality of $k$. Hence there is a tangent cone at $p$ isometric to some $\R^k$ with $k\le N$. 
 Now the result follows by Proposition ~\ref{prop:inf-hilb}.

\end{proof}

\begin{comment}
\begin{remark}\footnote{{\color{blue} again I would remove this Remark.}}
Since $CD_{loc}(K_-,n)$ is equivalent to $CD^*(K,n)$\cite{bast}, Theorem~\ref{CD+CAT implies RCD} remains true if the condition that $X$ is  $CD(K,n)$ is replaced by the weaker condition that it is $CD^*(K,n)$. 
\end{remark}
\end{comment}

Next we give an example of a space satisfying  \eqref{eq:cd+cat} which is not Alexandrov of $curv\ge \hat \uk$ for any $\hat \uk$. 
%This provides a positive answer to Question~\ref{cd+cat implies cbb?}.
\begin{example}\label{cd+cat-not-alex}
Let $(Y,d,m)$ be the closed unit ball $\bar B_1(0)$ in $\R^2$ with the standard Euclidean metric and $m=\mathcal H_2$. We are going  to show that there exist two $C^1$ functions $\phi, v\co Y\to \R$ such that $X=(\bar B_1(0),e^{\phi}d, e^vm)$ is $RCD(-100,3)$ and $CAT(0)$. The functions $\phi,v$ will be $C^4$ on $B_1(0)$ with the infimum of sectional curvature of $e^{2\phi}g_{Eucl}$  on $B_1(0)$ equal to $-\infty$. This will obviously imply that $X$ does not satisfy  $curv\ge \hat \uk$ for any $\hat \uk$.

 Recall that given a Riemannian manifold  $(M^n,g)$ if we change the Riemannian metric conformally $\tilde g=e^{2\phi} g$ then for any smooth function $f$ on $M$ its hessian changes by the formula

\begin{equation}\label{hess-conf}
\tilde \Hess_f(V,V)=\Hess_f(V,V)-2\langle \nabla \phi, V\rangle \langle \nabla f, V\rangle+|V|^2\cdot \langle \nabla \phi, \nabla f\rangle
\end{equation}
here and in what follows $\langle\cdot,\cdot\rangle$ and $\langle\cdot,\cdot\rangle_\sim$ are the inner products with respect to $g$ and $\tilde g$ respectively.

Also, recall that when $n=2$ Ricci tensors of $g$ and $\tilde g$ are related as follows

\begin{equation}\label{eq-ric-conf}
\tilde \Ric (V,V)=\Ric(V,V)-\Delta \phi |V|^2
\end{equation}
Then for any $N>2$ the weighted $N$-Ricci tensor of $(M^2, \tilde g, e^{-f}d\vol_{\tilde g})$ is equal to

\begin{equation}\label{N-ric-conf}
\begin{gathered}
\tilde \Ric_f^N(V,V)=\tilde \Ric(V,V)+\tilde\Hess_f(V,V)-\frac{\langle \tilde\nabla f, V\rangle_\sim^2}{N-2}=\\
=\Ric(V,V)-\Delta \phi |V|^2+\Hess_f(V,V)-2\langle \nabla \phi, V\rangle \langle \nabla f, V\rangle+|V|^2\cdot \langle \nabla \phi, \nabla f\rangle-\frac{\langle \nabla f, V\rangle^2}{N-2}
\end{gathered}
\end{equation}

By smoothing out  functions of the form $\delta( |(x,y)|-10)$ on $B_{10}(0)\subset \R^2$ ( with $0<\delta\ll 1$) it's easy to show that for every $\eps>0, L>0$ there exist smooth functions $f_{\eps,L}, \phi_{\eps,L}\co B_5(0)\to \R$ with the following properties

\begin{enumerate}
\item $|f_{\eps,L}|\le \eps, |\phi_{\eps,L}|\le\eps$;
\item  $f_{\eps,L}$ and $ \phi_{\eps,L}$ are convex and $\eps$-Lipschitz;
\item $\sup \Delta \phi_{\eps,L}=L$;
\item $\Hess f_{\eps,L}\ge \Delta\phi_{\eps,L}$ everywhere on $B_5(0)$;
\item $| f_{\eps,L}|_{C^4}\le \eps,| \phi_{\eps,L}|_{C^4}\le \eps$ outside $B_\eps(0)$.
\end{enumerate}

Now let $p_n=(0,1-\frac{1}{n})$ and let $f_n(x)=f_{10^{-n},n}(x-p_n), \phi_n(x)=\phi_{10^{-n},n}(x-p_n)$.

Lastly, let $f=\sum_{n=2}^\infty f_n, \phi=\sum _{n=2}^\infty\phi_n\co \bar B_2(0)\to\R$. Then it's easy to see that $f$ and $\phi$ satisfy the following properties
\begin{enumerate}
\item $|f|\le1, |\phi|\le 1$;
\item  $f$ and $ \phi$ are convex and $1/10$-Lipschitz;
\item $f$ and $\phi$ are $C^4$ on $B_1(0)$;
\item \label{sup-lap} $\sup \Delta \phi|_{B_1(0)}=+\infty$;
\item $\Hess f\ge \Delta\phi$ everywhere on $B_1(0)$.
\end{enumerate}

We claim that the space $(X,e^{\phi}d_{Eucl}, e^{-f}d\vol_{\tilde g})$ is $RCD(-100,3)$ and $CAT(0)$ where $\tilde g=e^{2\phi} g_{Eucl}$ on $B_1(0)$. Indeed, by \eqref{eq-ric-conf}  $\tilde g$ has $\sec\le 0$ since $\phi$ is convex and $g=g_{Eucl}$ is flat. In fact, more is true.

Let $u(x,y)=x^2+y^2$. Since $f$ and $\phi$ are $1/10$-Lipschitz, \eqref{hess-conf} easily implies that $\tilde \Hess_u\ge 0$ on $\bar B_1(0)$, i.e. $u$ is convex with respect to $\tilde g$ on $\bar B_1(0)$. Therefore, for any $0\le R< 1$ the space $X_R=(\bar B_R(0), e^{\phi}d)$ is $CAT(0)$ since it's locally $CAT(0)$, complete and simply connected. The same holds for $R=1$ because Gromov-Hausdorff limits of $CAT(0)$ spaces are again $CAT(0)$.

Convexity of $u$ with respect to $\tilde g$ also implies 
%In fact, by a theorem of Lytchak and Stadler ~\cite{Lyt-Stad} if $(Y,d_Y)$ is $CAT(0)$ and $u\co Y\to\R$ is bounded and convex then $(Y,e^ud_Y)$ is again $CAT(0)$. This immediately implies that $X_R$ is $CAT(0)$.
%Therefore, for any $0\le R\le 1$ the space $X_R=(\bar B_R(0), e^{\phi}d)$ is $CAT(0)$. This in particular implies
 that   $X_{R_1}$ is a convex subset of $X_{R_2}$ for any  $0<R_1<R_2\le 1$.
From formula \eqref{N-ric-conf} using the properties of $f$ and $\phi$ it easily follows that $\tilde \Ric_f^3\ge -100$ on $B_1(0)$.  Using that $X_{R_1}$ is a convex subset of $X_{R_2}$ for $R_1\le R_2$ this implies that $(X_R, e^{-f}d\vol_{\tilde g})$ is $RCD(-100,3)$ for any $0<R<1$. By passing to the limit as $R\to 1$ we get that $X=(\bar B_1(0), e^\phi d, e^{-f}d\vol_{\tilde g})$ is $RCD(-100,3)$ as well.
On the other hand, since $\sup \Delta \phi|_{B_1(0)}=+\infty$, by \eqref{eq-ric-conf}  the infimum of sectional curvature of $e^{2\phi}g_{Eucl}$  on $B_1(0)$ is equal to $-\infty$. Therefore $X$ does not satisfy $curv\ge \hat \uk$ for any  real $\hat \uk$.
\end{example}
%\begin{remark}
%It's easy to modify the above Example to construct a space which is $RCD(-100,3)$, $CAT(0)$ and not locally Alexandrov around any point.
%\end{remark}
%When dealing with $CAT(\uk)$ spaces to rule out certain pathological examples it is common to assume that all geodesics are locally extendible. 

%\include{non-branching}
Lastly, we show that spaces satisfying  \eqref{eq:cd+cat} are non-branching. This might be somewhat surprising given that branching $CAT(\uk)$ spaces are quite common (see e.g.  Examples \ref{ex1-cat} and \ref{ex2-cat}).
\begin{proposition}\label{prop:nonbra}
Let $X$ satisfy \eqref{eq:cd+cat}. Then $X$ is non-branching.
\end{proposition}
\begin{proof}
By rescaling we can assume that $X$ is $CD(-1,N)$ and $CAT(1)$. 
Suppose $\gamma_1,\gamma_2\co[0,1]\to X$ are two branching constant speed geodesics of length $<\pi$. Let $t_0\in(0,1)$ be the branching point. By uniqueness of geodesics in $CAT(1)$ spaces this means that $\gamma_1|_{[0,t_0]}=\gamma_2|_{[0,t_0]}$ and
$\gamma_1(t)\ne \gamma_2(t)$ for any $t>t_0$. Let $x=\gamma_1(0)=\gamma_2(0), y=\gamma_1(1), z=\gamma_2(1), p=\gamma_1(t_0)$. By shortening the geodesics we can assume that $|xp|=|yp|=|zp|=l<\pi/10$.
Then $t_0=1/2$. Recall that since $X$ is $CAT(1)$ the homothety map $\Phi_s^x$ centered at $x$ is $1$-Lipschitz on $B_{\pi/2}(x)$ for any $0\le t\le 1$. Therefore, for any $0<r<1/10, 0<s<1$
we have that $\Phi^x_s(B_r(\gamma_i(t))\subset B_{r}(\gamma_i(st))$. Note that $\Phi_s^x(B_r(\gamma_i(t))$ is the $s$-Minkowski sum $(1-s)\{x\}+sB_r(\gamma_i(t))$. 
Therefore $\Phi^x_{\frac{t_0}{t_0+\eps}}(B_r(\gamma_i(t_0+\eps)))\subset B_r(\gamma_i(t_0))$ for $\eps,r<1/10$.

On the other hand, by the Brunn-Minkowski inequality for $0<r<\eps$ we have that that\\ $m(\Phi^x_{\frac{t_0}{t_0+\eps}}(B_r(\gamma_i(t+\eps)))\ge (1-\delta(\eps)) m(B_r(\gamma_i(t_0+\eps)))$ where $\delta(\eps)=\delta(\eps,l,N)\to 0$ as $\eps\to 0$. 
{(Note that conditions $CD(K,N), CD^*(K,N)$ and $CD^e(K,N)$ give slightly different Brunn-Minkowski inequalities when $K\ne 0$ but all of them trivially imply the existence of $\delta(\eps)$ as above).}
Therefore, 
\[
m(B_r(\gamma_i(t_0)))\ge m(\Phi^x_{\frac{t_0}{t_0+\eps}}(B_r(\gamma_i(t_0+\eps)))\ge (1-\delta(\eps)) m(B_r(\gamma_i(t_0+\eps)))
\]
Applying the same argument to $\Phi_s^y, \Phi_s^z$ gives that $m(B_r(\gamma_i(t_0+\eps)))\ge (1-\delta(\eps))m(B_r(\gamma_i(t_0)))$. Combining the above we get

\begin{equation}\label{eq:alm-equal}
{\frac{1}{1-\delta(\eps)}}\ge \frac{m(B_r(p))}{m(B_r(\gamma_i(t_0+\eps)))}\ge 1-\delta(\eps)\,\, \textrm{ for } 0<r<\eps
\end{equation}

Fix an $0<\eps<1/10$ small enough so that $\delta(\eps)<1/100$. Since $\gamma_1(t_0+\eps)\ne \gamma_2(t_0+\eps)$, for all small $r<\eps$ we have that $B_r(\gamma_1(t_0+\eps))\cap B_r(\gamma_1(t_0+\eps))=\emptyset$.
Then  using \eqref{eq:alm-equal} and  the Brunn-Minkowski inequality for $A=\{x\}$ and $B=B_r(\gamma_1(t_0+\eps))\cup B_r(\gamma_2(t_0+\eps))$ we get that 

\[
 m(\Phi^x_{\frac{t_0}{t_0+\eps}}(B))\ge (1-\delta(\eps))m(B)\ge (1-\delta(\eps))2(1-\delta(\eps))m(B_r(p))\ge \frac{3}{2}m(B_r(p))
\]

On the other hand $\Phi^x_{\frac{t_0}{t_0+\eps}}(B)\subset B_r(p)$ and hence

\[
 m(\Phi^x_{\frac{t_0}{t_0+\eps}}(B))\le m(B_r(p))
\]
This is a contradiction and hence the proposition is proved and $X$ is non-branching.
\end{proof}
\begin{remark}
The above proof only uses the Brunn-Minkowski inequality when one of the sets is a point. Therefore, the proposition remains true if the condition that $X$ be $CD(K,N)$ is replaced by the weaker condition that it satisfies the measure-contracting property $MCP(K,N)$.
\end{remark}

\small{
%\bibliography{new}

\providecommand{\bysame}{\leavevmode\hbox to3em{\hrulefill}\thinspace}
\providecommand{\MR}{\relax\ifhmode\unskip\space\fi MR }
\providecommand{\MRhref}[2]{%
  \href{http://www.ams.org/mathscinet-getitem?mr=#1}{#2}
}
\providecommand{\href}[2]{#2}

\bibliographystyle{amsalpha}

% \nocite{*}
%\bibliography{Dissertation}
%\addcontentsline{toc}{chapter}{References}
}
\end{document}